\documentclass[review,onefignum,onetabnum]{siamart171218}

\usepackage{a4wide}
\usepackage{amsfonts}
\usepackage{amscd}
\usepackage{amssymb}
\usepackage{amsmath}
\usepackage{latexsym}
\usepackage{dsfont}
\usepackage{subfigure}
\usepackage{placeins}
\usepackage{mathrsfs}

\usepackage{floatflt}
\usepackage{nicefrac}
\usepackage{MnSymbol}
\usepackage[normalem]{ulem}
\usepackage{pgfplots}
\pgfplotsset{compat=newest}
\usepackage{booktabs}
\usepackage{epstopdf}

\usepackage{mathtools}
\usepackage{esvect} 
\usepackage{blindtext}
\usepackage[font=small,labelfont=bf]{caption}

\theoremstyle{plain}
\newtheorem{rem}[theorem]{Remark}
\newtheorem{assumption}[theorem]{Assumption}

{%
\setcounter{enumi}{0}
\renewcommand{\theenumi}{(\textbf{A}.\arabic{enumi})}

\begin{enumerate}}%
{\end{enumerate} }
{%
\setcounter{enumi}{0}
\renewcommand{\theenumi}{(\textbf{D}.\arabic{enumi})}

\begin{enumerate}}%
{\end{enumerate} }

\newcommand{\lk}{Lévy-Khinchin}

\newcommand{\cD}{\mathcal{D}}

\newcommand{\var}{{\rm Var}}

\makeatletter
\usepackage{color}

\makeatother

\newcommand{\be}{\begin {equation}}
\newcommand{\ee}{\end  {equation}}
\numberwithin{equation}{section} \allowdisplaybreaks[1]

\definecolor{darkgreen}{rgb}{0,.6,0}

\newcommand{\bee}{\begin {equation*}}
\newcommand{\eee}{\end {equation*}}

\nolinenumbers
\title{Subordinated Gaussian Random Fields}

\author{Andrea Barth \thanks{IANS\textbackslash SimTech, University of Stuttgart 
(\email{andrea.barth@mathematik.uni-stuttgart.de}).},
\and Robin Merkle \thanks{IANS\textbackslash SimTech, University of Stuttgart 
	(\email{robin.merkle@mathematik.uni-stuttgart.de}).}}

\headers{SUBORDINATED GAUSSIAN RANDOM FIELDS}{A. Barth and R. Merkle}

\begin{document}
		
	\maketitle
		
	\begin{abstract}
	  Motivated by the subordinated Brownian motion, we define a new class of (in general discontinuous) random fields on higher-dimensional parameter domains: the subordinated Gaussian random field. We investigate the pointwise marginal distribution of the constructed random fields, derive a Lévy-Khinchin-type formula and semi-explicit formulas for the covariance function. Further, we study the pointwise stochastic regularity and validate our theoretical findings in various numerical examples.
	\end{abstract}
	


	
	\section{Introduction}\label{sec:Intro}
In many applications of stochastic modeling, it is meaningful to consider random fields which are discontinous in space (e.g. in fractured porous media modeling). In the situation of a one-dimensional parameter space, like financial modeling, L\'evy processes turned out to be a very powerful class of (in general) discontinuous stochastic processes, combined with useful properties, see for example \cite{LevyProcessesInFinance, LevyProcessesAndStochasticCalculus,LevyProcessesAndInfinitelyDivisibleDistributions}. 

Whereas the extension of $\mathbb{R}$-valued L\'evy processes with one-dimensional parameter space to Hilbert space $\mathcal{H}$-valued Lévy processes is straight forward (see \cite{ApproximationAndSimulation}), the extension of Lévy processes to higher-dimensional \textit{parameter spaces} cannot follow an analogous approach. The reason can be found at the very starting point of the definition of L\'evy processes where time \text{increments} are considered: In fact, the definition of Lévy processes makes explicitly use of the total ordered structure underlying the considered time interval. The absence of such a structure on a higher-dimensional parameter space makes it impossible to simply extend the definition of a standard L\'evy process to higher-dimensional parameter spaces.

Subordinated fields did not receive much attention in the recent literature so far. In some classical papers on generalized random fields, of which \cite{Dobrushin} is an important representative (see also the references therein), subordinated fields are defined in terms of iterated It\^o-integrals. In the recent article, \cite{MakoginaandSpodarevLongTermLimits}, the authors investigate deterministic transformations of Gaussian random fields, so called Gaussian subordinated fields, and study excursion sets. The Rosenblatt distributions and long-range dependence of (subordinated) fields are looked into in \cite{LeonenkoMedinaRosenblatt}. The main contribution of our work in contrast is to prove properties of the (discontinuous) subordinated random fields and of their pointwise distributions, which are important in applications (see for example \cite{ZhangPorousMedia}, \cite{BastianPorousMedia} and  \cite{AStudyOfElliptic}).

We present an approach for an extension of a subclass of Lévy processes to more general parameter spaces: Motivated by the subordinated Brownian motion, we employ a higher-dimensional subordination approach using a Gaussian random field together with Lévy subordinators.

Figure \ref{Fig:SamplesSubordGRF} illustrates the approach with samples of a Gaussian random field (GRF) on $[0,1]^2$ with Matérn-1.5 covariance function and the corresponding subordinated field, where we used Poisson and Gamma processes on $[0,1]$ to subordinate the GRF.

\begin{figure}[ht]
	\centering
	\subfigure{\includegraphics[scale=0.29]{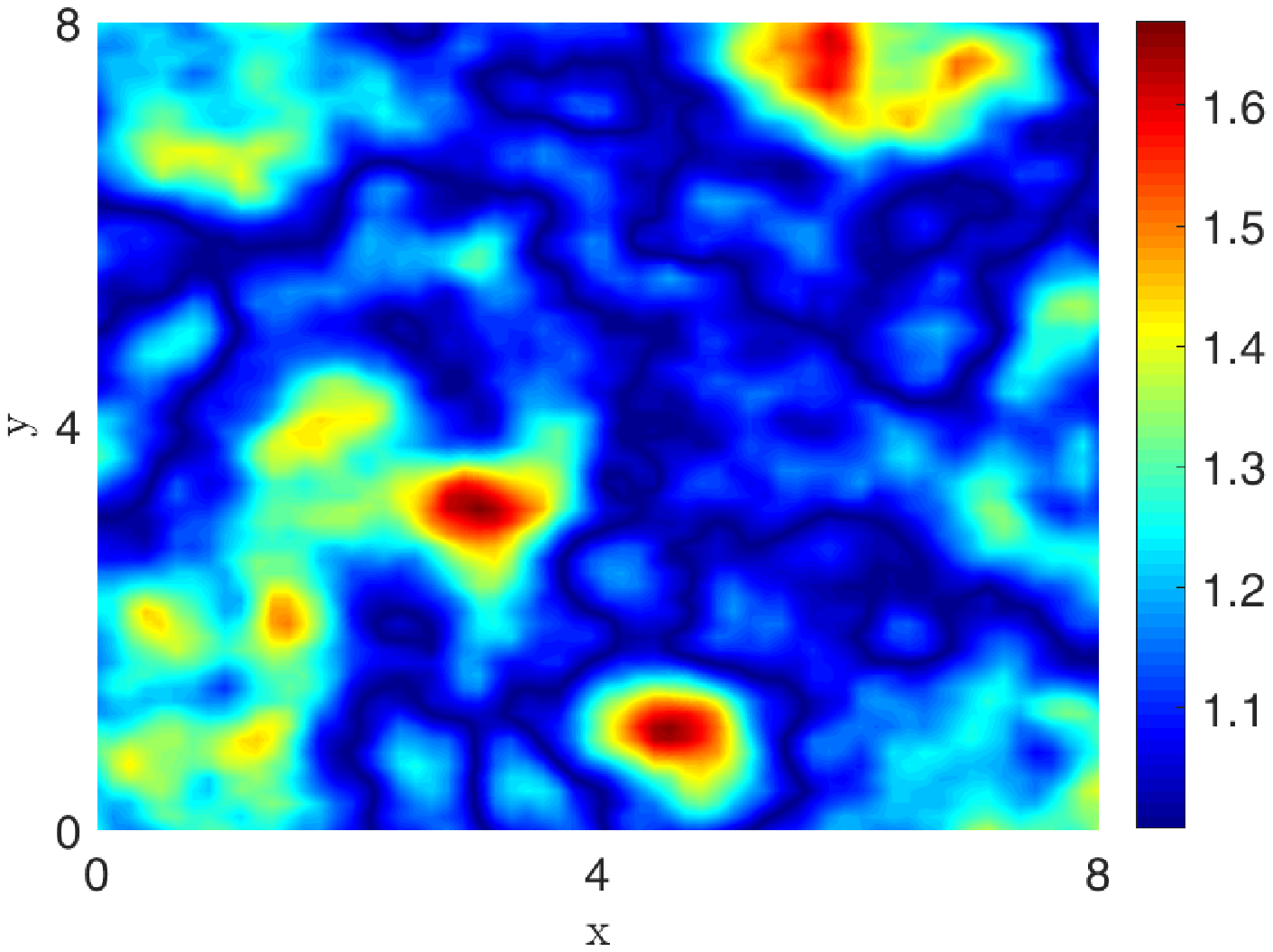}}
	\subfigure{\includegraphics[scale=0.29]{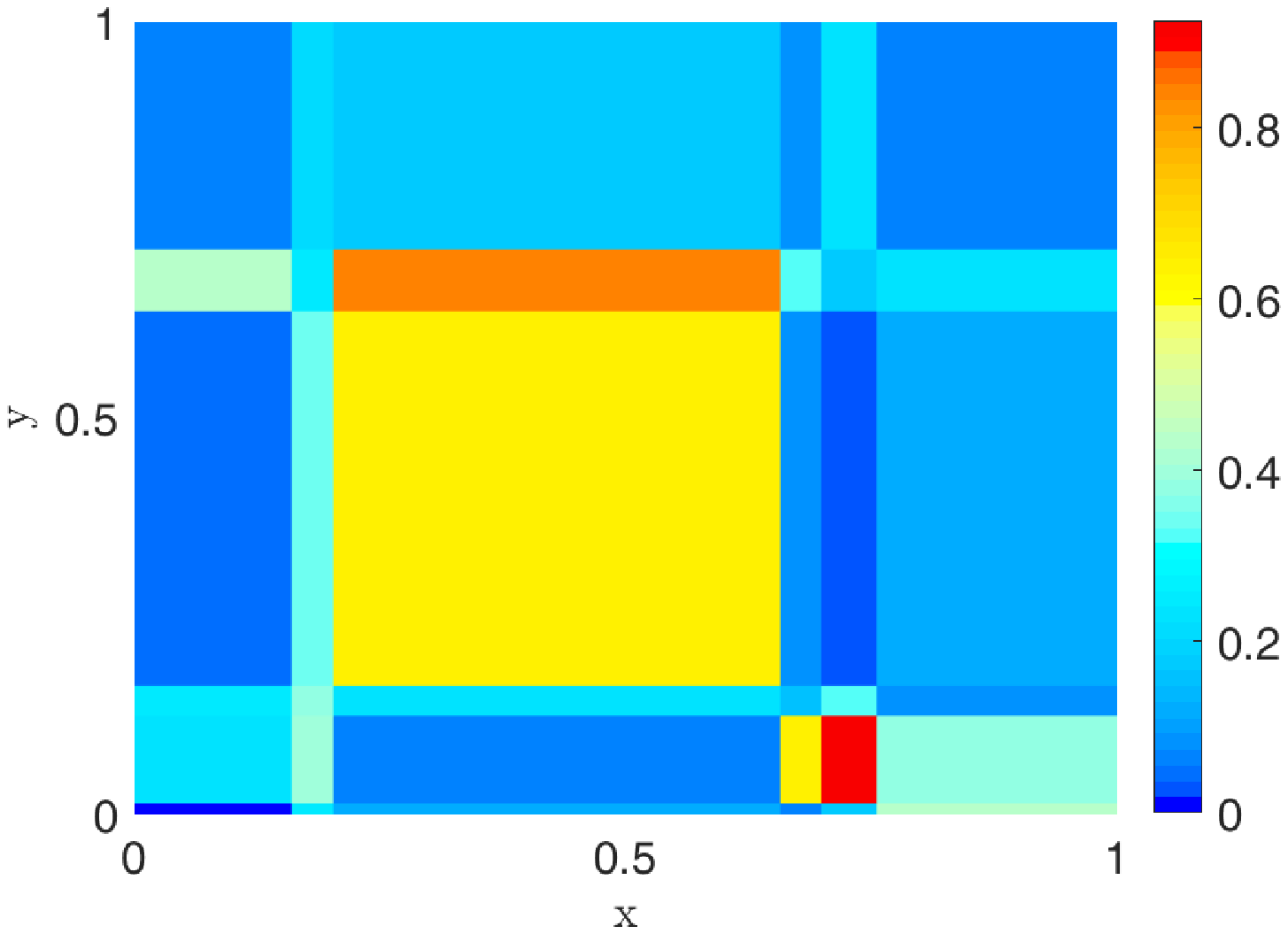}}
	\subfigure{\includegraphics[scale=0.29]{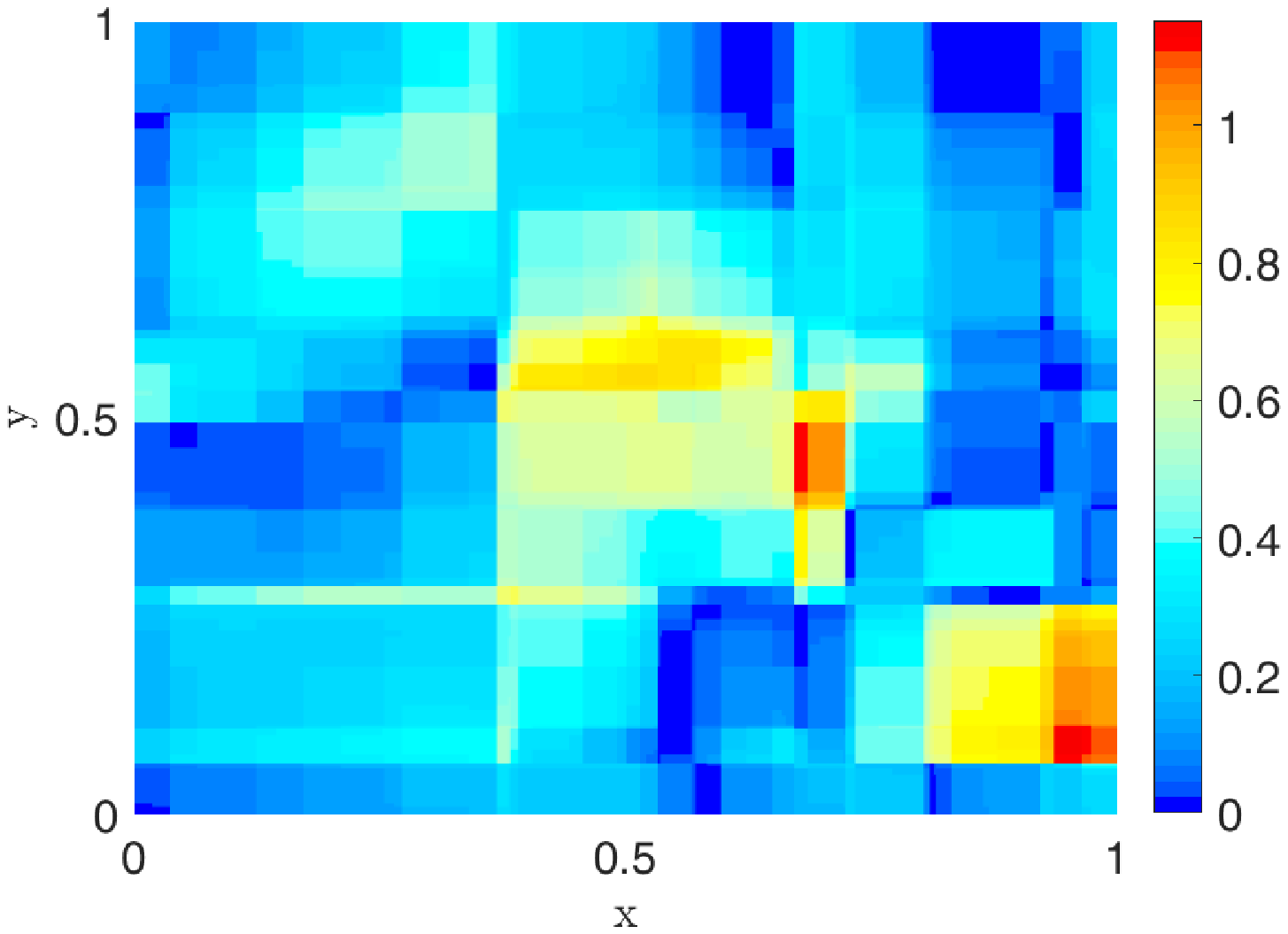}}
	\caption{Sample of Matérn-1.5-GRF (left), Poisson-subordinated GRF (middle) and Gamma-subordinated GRF (right).}\label{Fig:SamplesSubordGRF}
\end{figure}

These examples illustrate how the jumps of the Lévy subordinators produce jumps in the two-dimensional subordinated GRF. The question arises whether it is possible to transfer some theoretical results of one-dimensional Lévy processes to these random fields on higher-dimensional parameter spaces. In particular, a \lk -type formula to access the pointwise distribution of the constructed random field is of great interest (see Section \ref{SEC:PWDistLKFormula}). In Section~\ref{SEC:CovFct} we investigate the covariance structure of the subordinated fields and show how it is influenced by the choice of subordinators. The stochastic regularity of the subordinated fields is then studied in Section~\ref{SEC:StochReg}. There, we derive conditions which ensure the existence of pointwise moments. In the last section we demonstrate how one can numerically show that the findings from Sections~\ref{SEC:PWDistLKFormula}-\ref{SEC:StochReg} hold true for some exemplary fields.
	
	
	\section{Preliminaries}\label{sec:prelim}
In this section we give a short introduction to Lévy processes and Gaussian random fields as basis for the construction of subordinated Gaussian random fields. Throughout the paper, let $(\Omega,\mathcal{F},\mathbb{P})$ be a complete probability space. 

\subsection{Lévy processes}

Let $\mathcal{T}\subseteq \mathbb{R}_+:=[0,+\infty)$ be an arbitrary time domain. A stochastic process $X=(X(t),~t\in \mathcal{T})$ on $\mathcal{T}$ is a family of random variables on the probability space $(\Omega,\mathcal{F},\mathbb{P})$.  
\begin{definition}
A stochastic process $l$ on $\mathcal{T}=[0,+\infty)$ is said to be a \textit{Lévy process} if 
\begin{enumerate}
\item $l(0)=0~\mathbb{P}-a.s.$,
\item $l$ has independent increments, i.e. for each $0\leq t_1\leq t_2\leq \dots \leq t_{n+1}$ the random variables $(l(t_{j+1}) -l(t_j)),~1\leq j\leq n)$ are independent,
\item $l$ has stationary increments, i.e. for each $0\leq t_1\leq t_2\leq \dots \leq t_{n+1}$ it holds $l(t_{j+1}) - l(t_j) \stackrel{\cD}{=} l(t_{j+1}-t_j) - l(0)\stackrel{\cD}{=} l(t_{j+1}-t_j)$, where $\stackrel{\cD}{=}$ denotes equivalence in distribution,
\item $l$ is stochastically continuous, i.e. for all $a>0$ and for all $s\geq 0$ it holds 
\begin{align*}
\underset{t\rightarrow s}{\lim}\, \mathbb{P}(|l(t)-l(s)|>a)=0.
\end{align*}
\end{enumerate}
\end{definition} 		
A very important characterization property of Lévy processes is given by the so called \textit{Lévy-Khinchin formula}. 
\begin{theorem}(Lévy-Khinchin formula, see \cite[Th. 1.3.3]{LevyProcessesAndStochasticCalculus})\label{TH:LevyKhinchinFormula1d}
Let $l$ be a real valued Lévy process on $\mathcal{T}\subset \mathbb{R}_+:=[0,+\infty)$. There exist a drift parameter $b\in \mathbb{R}$, a noise parameter $\sigma_N^2\in \mathbb{R}_+$ and a measure $\nu$ on $(\mathcal{T},\mathcal{B}(\mathcal{T}))$ such that the characteristic function $\phi_{l(t)}$, for $t\in \mathcal{T}$, admits the representation 
\begin{align*}
\phi_{l(t)}(\xi) := \mathbb{E}(\exp(i\xi l(t))) = \exp\Big(t\big(ib\xi- \frac{\sigma_N^2}{2}\xi^2 + \int_{\mathbb{R}\setminus \{0\}} e^{i\xi y} - 1 - i\xi y\mathds{1}_{\{|y|\leq 1\}}(y)\,\nu(dy)\big)\Big),~\xi\in \mathbb{R}.
\end{align*}

\end{theorem}

It follows from this theorem that every Lévy process is fully characterized by the so called Lévy triplet $(b,\sigma_N^2,\nu)$. Hence, while trying to extend the concept of a Lévy process to a higher-dimensional parameter domain, one important goal is to obtain an analogous characterization property. 

Within the class of Lévy processes there exists a subclass which is given by the so called subordinators: 
A \textit{Lévy subordinator} on $\mathcal{T}\subset \mathbb{R}_+$ is a Lévy process that is non-decreasing $\mathbb{P}$-almost surely. The characteristic function of a Lévy subordinator $l$ admits the form 
\begin{align}\label{EQ:CharFctSubordinator}
\phi_{l(t)}(\xi ) = \mathbb{E}(\exp(i\xi l(t))) = \exp\Big(t\big(i\gamma \xi + \int_0^\infty e^{i\xi y} - 1 \,\nu(dy)\big)\Big),~\xi\in \mathbb{R},
\end{align}
for $t\in \mathcal{T}$ (see \cite[Theorem 1.3.15]{LevyProcessesAndStochasticCalculus}). Here, $\nu$ is the Lévy measure of $l$ and the constant $\gamma$ is a constant which does in general not coincide with the constant $b$ in the Lévy-Khinchin formula. The Lévy measure $\nu$ on $(\mathbb{R},\mathcal{B}(\mathbb{R}))$ of a Lévy subordinator satisfies
\begin{align*}
\nu(-\infty,0)=0 \text{ and } \int_0^\infty	 \min(y,1) \,\nu(dy)<\infty.
\end{align*}In the following, we always mean the triplet $(\gamma,0,\nu)$ corresponding to representation \eqref{EQ:CharFctSubordinator} if we refer to the characteristic triplet of a Lévy subordinator.

\subsection{Gaussian random fields}
Let $\mathscr{D}\subset \mathbb{R}^d$ be a spatial domain. A random field $R=(R(\underline{x}),~\underline{x}\in \mathscr{D})$ is a family of random variables on the probability space $(\Omega,\mathcal{F},\mathbb{P})$. 		
In our approach to extend Lévy processes on higher-dimensional parameter domains, one important component is given by the Gaussian random field.
	
	\begin{definition}(see \cite[Sc. 1.2]{RandomFieldsAndGeometry})
	A random field $W:\mathscr{D}\times \Omega \rightarrow \mathbb{R}$ on a $d$-dimensional domain $\mathscr{D}\subset \mathbb{R}^d$ is said to be a \textit{Gaussian random field (GRF)} if, for any $\underline{x}^{(1)},\dots,\underline{x}^{(n)} \in \mathscr{D}$ with $n\in \mathbb{N}$, the $n$-dimensional random variable $(W(\underline{x}^{(1)}),\dots,W(\underline{x}^{(n)}))$ is multivariate Gaussian distributed. In this case, we define the \textit{mean function} $\mu(\underline{x}):=\mathbb{E}(W(\underline{x}))$, for $\underline{x}\in \mathscr{D}$, as well as the covariance function $q(\underline{x}^{(1)},\underline{x}^{(2)}):=\mathbb{E}((W(\underline{x}^{(1)})-\mu(\underline{x}^{(2)}))(W(\underline{x}^{(1)})-\mu(\underline{x}^{(2)})))$ for $\underline{x}^{(1)},~\underline{x}^{(2)}\in \mathscr{D}$. The GRF $W$ is called centered, if $\mu(\underline{x})=0$ for all $\underline{x}\in \mathscr{D}$. 
	\end{definition}
	Note that every Gaussian random field is determined uniquely by its mean and covariance function. We denote by $Q:\mathcal{L}^2(\mathscr{D})\rightarrow \mathcal{L}^2(\mathscr{D})$ the \textit{covariance operator} of $W$ which is defined by 
\begin{align*}
Q(\psi)(\underline{x})=\int_{\mathscr{D}}q(\underline{x},\underline{y})\psi(\underline{y})d\underline{y} \text{, for } \underline{x}\in \mathscr{D},
\end{align*}
for $\psi \in \mathcal{L}^2(\mathscr{D})$. Here, $\mathcal{L}^2(\mathscr{D})$ denotes the set of all square integrable functions over $\mathscr{D}$. Further, if $\mathscr{D}$ is compact, there exists a decreasing sequence $(\lambda_i,~i\in \mathbb{N})$ of real eigenvalues of $Q$ with corresponding eigenfunctions $(e_i,~i\in\mathbb{N})\subset \mathcal{L}^2(\mathscr{D})$ which form an orthonormal basis of $\mathcal{L}^2(\mathscr{D})$ (see \cite[Section 3.2]{RandomFieldsAndGeometry} and \cite[Theorem VI.3.2 and Chapter II.3]{Funktionalanalysis}). The GRF $W$ is called \textit{stationary} if the mean function $\mu$ is constant and the covariance function $q(\underline{x}^{(1)},\underline{x}^{(2)})$ only depends on the difference $\underline{x}^{(1)}-\underline{x}^{(2)}$ of the values $\underline{x}^{(1)},~\underline{x}^{(2)}\in \mathscr{D}$ (see \cite{RandomFieldsAndGeometry}, p. 102).


	\section{The Subordinated Gaussian random field}	\label{sec:SubordGRF}

Throughout the rest of this paper, let $ d\in \mathbb{N}$ be a natural number with $d\geq 2$ and $T_1,\dots,T_d>0$ be positive values. We define the horizon vector $\mathbb{T}:=(T_1,\dots,T_d)$ and consider the spatial domain $[0,\mathbb{T}]_d:=[0,T_1]\times \dots \times [0,T_d]\subset \mathbb{R}^d$. After a short motivation we define next the subordinated field and show that it is indeed measurable.
	\subsection{Motivation: the subordinated Brownian motion}\label{SUBSEC:SubpordBM}
	In order to motivate the novel subordination approach for the extension of Lévy processes, we shortly repeat the main ideas of the \textit{subordinated Brownian motion} which is defined as a Lévy-time-changed Brownian motion: Let $B=(B(t),~t\in \mathbb{R}_+)$ be a Brownian motion and $l=(l(t),~t\in \mathbb{R}_+)$ be a subordinator. The subordinated Brownian motion is then defined to be the process
\begin{align*}
L(t):=B(l(t)), ~t\in \mathbb{R}_+.
\end{align*}
It follows from \cite[Theorem 1.3.25]{LevyProcessesAndStochasticCalculus} that the process $L$ is again a Lévy process. Note that the class of subordinated Brownian motions is a rich class of processes with great distributional flexibility. For example, the well known Generalized Hyperbolic Lévy process can be represented as a NIG-subordinated Brownian motion (see \cite{ApproximationAndSimulation} and expecially Lemma 4.1 therein). 

\subsection{The definition of the subordinated Gaussian random field}\label{SUBSEC:TheDefinitionOfTheSubordGRF}
Let $W=(W(\underline{x}),~\underline{x}=(x_1,\dots,x_d)\in \mathbb{R}_+^d)$ be a GRF such that $W$ is $\mathcal{F}\otimes \mathcal{B}(\mathbb{R}_+^d)-\mathcal{B}(\mathbb{R})$-measurable. We denote by $\mu:\mathbb{R}_+^d\rightarrow \mathbb{R}$ the mean function and by $q:\mathbb{R}_+^d\times \mathbb{R}_+^d\rightarrow \mathbb{R}$ the covariance function of $W$. Let $l_k=(l_k(x),~x\in [0,T_k])$ be independent Lévy subordinators with triplets $(\gamma_k,0,\nu_k)$, for $k\in\{1, \dots, d\}$, corresponding to representation \eqref{EQ:CharFctSubordinator}. Further, we assume that the Lévy subordinators are stochastically independent of the GRF $W$. We consider the random field
\begin{align*}
L:\Omega \times [0,\mathbb{T}]_d\rightarrow \mathbb{R} \text{ with } L(x_1,\dots,x_d):=W(l_1(x_1),\dots,l_d(x_d)) \text{, for } \underline{x}=(x_1,\dots,x_d)\in [0,\mathbb{T}]_d,
\end{align*}
and call it \textit{subordinated Gaussian random field (subordinated GRF)}.
\begin{rem}
Note that assuming that $W$ has $\mathbb{P}$-almost surely continuous paths is sufficient to ensure that $W$ is a jointly measurable function since $W$ is a Carathéodory function in this case (see \cite[Lemma 4.51]{InfiniteDimensionalAnalysis}). A sufficient condition for the pathwise continuity of GRFs is given, for example, by \cite[Theorem 1.4.1]{RandomFieldsAndGeometry} (see also the discussion in \cite[Section 1.3, p. 13]{RandomFieldsAndGeometry}). A specific example for a class of GRFs with at least continuous samples is given by the Matérn GRFs: for a given smoothness parameter $\nu > \frac{1}{2}$, a correlation parameter $r>0$ and a variance parameter $\sigma^2>0$ the Matérn-$\nu$ covariance function on $\mathbb{R}_+^d\times \mathbb{R}_+^d$ is given by $q_M(\underline{x},\underline{y})=\rho_M(\|\underline{x}-\underline{y}\|_2)$ with 
\begin{align*}
\rho_M(s) = \sigma^2 \frac{2^{1-\nu}}{\Gamma(\nu)}\Big(\frac{2s\sqrt{\nu}}{r}\Big)^\nu K_\nu\Big(\frac{2s\sqrt{\nu}}{r}\Big), \text{ for }s\geq 0,
\end{align*}
where $\Gamma(\cdot)$ is the Gamma function and $K_\nu(\cdot)$ is the modified Bessel function of the second kind (see \cite[Section 2.2 and Proposition 1]{QuasiMonteCarloFEMethodsForEllipticPDEsWithLognormalRandomCoefficients}). Here, $\|\cdot\|_2$ denotes the Euclidean norm on $\mathbb{R}^d$. A Matérn-$\nu$ GRF is a centered GRF with covariance function $q_M$.
\end{rem}

\subsection{Measurabiltiy}
In Subsection \ref{SUBSEC:TheDefinitionOfTheSubordGRF} we introduced the subordinated GRF $L$ as a \textit{random field}. Strictly speaking, we therefore have to verify that point evaluations of the field $L$ are random variables, meaning that we have to ensure measurability of these objects. Note that this is not trivial, since - due to the construction of $L$ - the Lévy subordinators induce an additional $\omega$-dependence in the spatial direction of the GRF $W$. The following lemmas prove measurability of point evaluations of $L$ and spatial measurability of the field, if we consider it as a (stochastically parametrized) space-dependent function.

\begin{lemma}\label{LE:MeasurabilityOfSubordGRF}
Let $L$ be a subordinated GRF on the spatial domain $[0,\mathbb{T}]_d$ as constructed in Subsection \ref{SUBSEC:TheDefinitionOfTheSubordGRF}. For a fixed $\underline{x}\in[0,\mathbb{T}]_d$, the mapping
\begin{align*}
L(\underline{x}):\Omega \rightarrow \mathbb{R}, \quad 
 \omega \mapsto L(\omega,\underline{x})=W(\omega,l_1(\omega,x_1),\dots,l_d(\omega,x_d)),
\end{align*}
is $\mathcal{F}-\mathcal{B}(\mathbb{R})$-measurable.
\end{lemma}
\begin{proof}
The mapping 
\begin{align*}
W:\Omega\times \mathbb{R}_+^d\rightarrow \mathbb{R},\quad
(\omega,\underline{x})\mapsto W(\omega,\underline{x}),
\end{align*} is $\mathcal{F}\otimes \mathcal{B}(\mathbb{R}_+^d)-\mathcal{B}(\mathbb{R})$-measurable by assumption. Further, since the mapping
\begin{align*}
l_k(x_k):\Omega\rightarrow \mathbb{R}_+,\quad
\omega\mapsto l_k(x_k,\omega),
\end{align*}
 is $\mathcal{F}-\mathcal{B}(\mathbb{R}_+)$-measurable for every $k=1,\dots,d$ and the identity mapping $id_\Omega:\Omega\rightarrow\Omega$ is $\mathcal{F}-\mathcal{F}$-measurable, we obtain by \cite[Lemma 4.49]{InfiniteDimensionalAnalysis} that the mapping
\begin{align*}
\Omega \rightarrow \Omega\times\mathbb{R}_+^d,\quad
\omega \mapsto (\omega,l_1(\omega,x_1),\dots,l_d(\omega,x_d)),
\end{align*}
is $\mathcal{F}-\mathcal{F}\otimes \mathcal{B}(\mathbb{R}_+^d)$-measurable. Hence, the composition mapping $(\omega,\underline{x})\mapsto L(\omega,\underline{x})$ is $\mathcal{F}-\mathcal{B}(\mathbb{R})$-measurable.
\end{proof}

If we further assume continuity of the GRF $W$ in the spatial variable, we obtain mesaurability of the subordinated GRF $L$ in $\underline{x}\in [0,\mathbb{T}]_d$, which allows us to take spatial integrals of the field.
\begin{lemma}
We consider the subordinated GRF $L=W(l_1(\cdot),\dots,l_d(\cdot))$, and assume that the underlying GRF $W$ has $\mathbb{P}$-almost surely continuous paths. For almost all $\omega\in \Omega$, the mapping 
\begin{align*}
L(\omega):[0,\mathbb{T}]_d\rightarrow \mathbb{R},\quad
\underline{x}=(x_1,\dots,x_d)\mapsto W(\omega,l_1(\omega,x_1)\dots,l_d(\omega,x_d)),
\end{align*}
is $\mathcal{B}([0,\mathbb{T}]_d)-\mathcal{B}(\mathbb{R})$-measurable.
\end{lemma}
\begin{proof}
For any $k\in \{1,\dots,d\}$, the Lévy process $l_k$ has $\mathbb{P}$-almost surely cádlág paths and, hence, the mapping 
\begin{align*}
l_k(\omega):[0,T_k]\rightarrow \mathbb{R}_+,\quad
x_k\mapsto l_k(\omega,x_k), 
\end{align*} 
is $\mathcal{B}([0,T_k])-\mathcal{B}(\mathbb{R}_+)$-measurable (see \cite[Chapter 1, Theorem 30]{StochasticIntegrationAndDifferentialEquations}).
Next, we consider domain-extended versions of the processes for a fixed $\omega\in \Omega$: for any $k\in\{1\dots,d\}$ and any $\underline{x}=(x_1,\dots,x_d)\in [0,\mathbb{T}]_d$, we define the function
\begin{align*}
\tilde{l}_k(\omega,\underline{x}):= l_k(\omega,x_k),\text{ for }\underline{x}\in[0,\mathbb{T}]_d.
\end{align*}
Since $\tilde{l}_k(\omega,\cdot)$ is measurable in $x_k$ and constant in the variable $(x_1,\dots,x_{k-1},x_{k+1},\dots,x_d)$, it is $\mathcal{B}([0,\mathbb{T}]_d)-\mathbb{B}(\mathbb{R}_+)$ measurable by \cite[Lemma 4.51]{InfiniteDimensionalAnalysis}. An application of \cite[Lemma 4.49]{InfiniteDimensionalAnalysis} yields the $\mathcal{B}([0,\mathbb{T}_d])-\mathcal{B}(\mathbb{R}_+^d)$-measurability of the mapping
\begin{align*}
[0,\mathbb{T}]_d \rightarrow \mathbb{R}_+^d,\quad
\underline{x}=(x_1,\dots,x_d) &\mapsto (\tilde{l}_1(\omega,\underline{x}),\dots,\tilde{l}_d(\omega,\underline{x}))= (l_1(\omega,x_1),\dots,l_d(\omega,x_d)).
\end{align*}
By assumption, $W(\omega)$ has continuous paths and, hence, it is $\mathcal{B}(\mathbb{R}_+^d)-\mathcal{B}(\mathbb{R})$ measurable. Therefore, the mapping 
\begin{align*}
L(\omega):[0,\mathbb{T}]_d\rightarrow \mathbb{R},\quad
\underline{x}=(x_1,\dots,x_d)\mapsto W(\omega,l_1(\omega,x_1)\dots,l_d(\omega,x_d)),
\end{align*}
is $\mathcal{B}([0,\mathbb{T}_d])-\mathcal{B}(\mathbb{R})$-measurable for $\mathbb{P}$-almost every $\omega \in \Omega$.
\end{proof}


	\section{The pointwise distribution of the subordinated GRF and the \lk-formula}
	\label{SEC:PWDistLKFormula}

	In this section we prove a \lk -type formula for the subordinated GRF in order to have access to the pointwise distribution. This is important, for example,  in view of statistical fitting and other applications (see Section \ref{SEC:NumEx}). In order to be able to do so we need the following technical lemma about the expectation of the composition of independent random variables. The assertion and its proof is a generalization of the corresponding assertion given in the proof of  \cite[Theorem 30.1]{LevyProcessesAndInfinitelyDivisibleDistributions}.
	
	\begin{lemma}\label{LE:ExpValueIndepRVs}
	Let $W:\Omega\times \mathbb{R}_+^d\rightarrow \mathbb{R}$ be a $\mathbb{P}-a.s.$ continuous random field and let $Z:\Omega\rightarrow\mathbb{R}_+^d$ be a $\mathbb{R}_+^d$-valued random variable which is independent of the random field $W$. Further, let $g:\mathbb{R}\rightarrow\mathbb{R}$ be a deterministic, continuous function. It holds
	\begin{align*}
	\mathbb{E}(g(W(Z))=\mathbb{E}(m(Z)),
	\end{align*}
	where $m(z):=\mathbb{E}(g(W(z))$ for deterministic $z\in \mathbb{R}_+^d$.
	\end{lemma}
	
	\begin{proof}
	\textit{Step 1:} Assume that $Z$ takes only a countable number of values in $\mathbb{R}_+^d$ and g is globally bounded.
	In this case there exists a family of vectors $(z_i,~i\in \mathbb{N})\subset \mathbb{R}_+^d$ such that $\sum_{i\in \mathbb{N}}\mathbb{P}(Z=z_i)=1$. We observe that
	\begin{align*}
	\underset{n\rightarrow\infty}{\lim}\sum_{i=1}^ng(W(z_i))\mathds{1}_{\{Z=z_i\}} = g(W(Z)) ~ \mathbb{P}-a.s., 
\end{align*}
	and, by the boundedness of $g$, we obtain for all $n\in \mathbb{N}$ the estimate
	\begin{align*}
	|\sum_{i=1}^ng(W(z_i))\mathds{1}_{\{Z=z_i\}}|\leq C \sum_{i=1}^\infty\mathds{1}_{\{Z=z_i\}}.
	\end{align*}
	Further, by the use of the monotone convergence theorem, we obtain 
	\begin{align*}
	\mathbb{E}(\sum_{i=1}^\infty\mathds{1}_{\{Z=z_i\}})= \sum_{i=1} ^\infty \mathbb{P}(Z=z_i) =1 < \infty.
	\end{align*}
	Therefore, we can apply the dominated convergence theorem together with the independence of $Z$ and $W$ to calculate
	\begin{align*}
	\mathbb{E}(g(W(Z))) &= \sum_{i\in\mathbb{N}}\mathbb{E}(g(W(z_i))\mathds{1}_{\{Z=z_i\}})=\sum_{i\in\mathbb{N}}\mathbb{E}(\mathds{1}_{\{Z=z_i\}})\mathbb{E}(g(W(z_i)))\\ &= \sum_{i\in\mathbb{N}}\mathbb{P}(Z=z_i) \,m(z_i)=\mathbb{E}(m(Z)).
	\end{align*}
	\textit{Step 2:} In this step, we do not impose any assumptions on the range of the random vector $Z$ but we still assume that the function $g$ is globally bounded on $\mathbb{R}$ and therefore also the function $m$ is globally bounded. 
 We write $Z=(Z^{(1)},\dots,Z^{(d)})$ and define 
 \begin{align*}
 Z_n^{(i)}:=\sum_{k=0}^\infty \frac{k}{n}\mathds{1}_{\{\frac{k}{n}\leq Z^{(i)}<\frac{k+1}{n}\}},
 \end{align*}
 for $i=1,\dots,d$. By construction we obtain the pointwise convergence $Z_n^{(i)}\rightarrow Z^{(i)}~\mathbb{P}-a.s.$ for $n\rightarrow \infty$ for every component $i=1,\dots,d$. Further, since $g$ is continuous and $W$ has $\mathbb{P}-a.s.$ continuous paths we obtain
 \begin{align*}
 g(W(Z_n^{(1)},\dots,Z_n^{(d)}))\rightarrow g(W(Z))\text{ and } m(Z_n^{(1)},\dots,Z_n^{(d)})\rightarrow m(Z)~\mathbb{P}-a.s.\text{ for }n\rightarrow \infty.
\end{align*}   
	Here, the continuity of $m$ follows from the boundedness and the continuity of $g$. Using the boundedness of the functions $g$ and $m$ and the dominated convergence theorem together with the first step we obtain
	\begin{align*}
	\mathbb{E}(g(W(Z))) &= \underset{n\rightarrow \infty}{\lim} \mathbb{E}(g(W(Z_n^{(1)},\dots, Z_n^{(d)})))=\underset{n\rightarrow \infty}{\lim} \mathbb{E}(m(Z_n^{(1)},\dots,Z_n^{(d)})) = \mathbb{E}(m(Z)).
	\end{align*}
	\textit{Step 3:} In this step we assume that $g(x)\geq 0$ on $\mathbb{R}$ but $g$ does not necessarily have to be bounded. It follows that $m$ is also non-negative on $\mathbb{R}_+^d$. For a fixed threshold $A>0$, we define the cut function $\chi_A:\mathbb{R}\rightarrow\mathbb{R}$:
	\begin{align*}
	\chi_A(x):=\begin{cases}x&,\text{ if } x\leq A, \\ A &,\text{ if } x>A. \end{cases}
	\end{align*}
	Since $g$ and $m$ are nonnegative we obtain the $\mathbb{P}-a.s.$ monotone convergence of $\chi_A(g(W(Z))))\rightarrow g(W(Z)))$ for $A\rightarrow +\infty$. We define $m_A(z):=\mathbb{E}(\chi_A(g(W(z)))$, for $z\in \mathbb{R}_+^d$, and obtain by the monotone convergence theorem 
	\begin{align*}
	m_A(Z)\rightarrow m(Z) ~\mathbb{P}-a.s. \text{ for } A\rightarrow +\infty. 
\end{align*}	 Using Step 2 and the monotone convergence theorem we obtain:
	\begin{align*}
	\mathbb{E}(g(W(Z))) = \underset{A\rightarrow +\infty}{\lim} \mathbb{E}(\chi_A(g(W(Z)))) = \underset{A\rightarrow +\infty}{\lim} \mathbb{E}(m_A(Z)) = \mathbb{E}(m(Z)).
	\end{align*}
	\textit{Step 4:} Finally, we consider an arbitrary continuous function $g:\mathbb{R}\rightarrow\mathbb{R}$. We write $g^+=\max\{0,g\},~g^-=-\min\{0,g\}$ as well as $\tilde{m}^+(z)=\mathbb{E}(g^+(W(z))),~\tilde{m}^-(z)=\mathbb{E}(g^-(W(z)))$ for $z\in\mathbb{R}_+^d$ and obtain the additive decomposition $g(x)=g^+(x) - g^-(x)$ for $x\in\mathbb{R}$ and $m(z)=\tilde{m}^+(z) - \tilde{m}^-(z)$ for $z\in\mathbb{R}_+^d$ by the additivity of the integral with respect to the integration domain.
	We apply Step 3 to optain
	\begin{align*}
	\mathbb{E}(g(W(Z))) = \mathbb{E}(g^+(W(Z))) - \mathbb{E}(g^-(W(Z))) = \mathbb{E}(\tilde{m}^+(Z)) - \mathbb{E}(\tilde{m}^-(Z)) = \mathbb{E}(m(Z)),
	\end{align*}
	which proves the assertion.

	\end{proof}
	\begin{rem}\label{REM:LemmaAlsoComplexFunctions}
	Note that following Steps 1 and 2 of the proof of Lemma \ref{LE:ExpValueIndepRVs} one obtains that the assertion of the lemma also holds for complex valued, bounded, continuous and deterministic functions $g:\mathbb{R}\rightarrow\mathbb{C}$.
\end{rem} 
An application of Lemma \ref{LE:ExpValueIndepRVs} gives the following semi-explicit formula for the characteristic function of a subordinated GRF.
\begin{corollary}\label{COR:GeneralFormulaCharFctSubordGRF}
Let $W$ be a $\mathbb{P}-a.s.$ continuous GRF on $\mathbb{R}_+^d$ with mean function $\mu:\mathbb{R}_+^d\rightarrow \mathbb{R}$ and covariance function $q:\mathbb{R}_+^d\times \mathbb{R}_+^d\rightarrow\mathbb{R}$. Further, let $l_k=(l_k(t),~t\in [0,T_k])$, for $k=1,\dots,d$, be independent Lévy subordinators which are independent of $W$. The characteristic function of the subordinated GRF defined by $L(\underline{x}):=W(l_1(x_1),\dots,l_d(x_d))$, for $\underline{x}=(x_1,\dots,x_d)\in [0,\mathbb{T}]_d$, admits the formula
\begin{align*}
\mathbb{E}(\exp(i\xi L(\underline{x})))&=\mathbb{E}(\exp(i\xi W(l_1(x_1),\dots,l_d(x_d)))) \\
&= \mathbb{E}(\exp(i\mu(l_1(x_1),\dots,l_d(x_d)) - \frac{1}{2}\xi^2 \sigma_W^2(l_1(x_1),\dots,l_d(x_d)))),
\end{align*}
for $\xi \in \mathbb{R}$ and any fixed point $\underline{x}=(x_1,\dots,x_d)\in [0,\mathbb{T}]_d$. Here, the variance function $\sigma_W^2:\mathbb{R}_+^d\rightarrow \mathbb{R}_+$ is given by $\sigma_W^2(\underline{x}):=q(\underline{x},\underline{x})$ for $\underline{x}\in \mathbb{R}_+^d$.
\end{corollary}
\begin{proof}
The GRF $W$ is pointwise normally distributed with parameters specified by the functions $\mu$ and $\sigma_W^2$. More precise, for a fixed point $\underline{x}\in \mathbb{R}_+^d$, the characteristic function of $W$ admits the form
\begin{align*}
\mathbb{E}(\exp(i\xi W(\underline{x})))=\exp(i\mu(\underline{x})\xi - \frac{1}{2}\sigma_W^2(\underline{x}) \xi^2),
\end{align*}
for $\xi \in \mathbb{R}$ (see \cite[Satz 15.12]{WTheorie}). The assertion then follows by an application of Lemma \ref{LE:ExpValueIndepRVs} together with Remark \ref{REM:LemmaAlsoComplexFunctions}	.
\end{proof}

In the one-dimensional case, the \lk \space formula gives an explixit representation of the characteristic function of a Lévy process. This representation also applies to the subordinated Brownian motion, since it is itself a Lévy process (see Subsection \ref{SUBSEC:SubpordBM}). Note that in the construction of the subordinated Brownian motion one cannot replace the Brownian motion by a  general one-parameter GRF on $\mathbb{R}_+$ without losing the validity of the \lk \space formula. Hence, in the case the subordinated GRF on a higher-dimensional parameter space, it is natural that we have to restrict the class of admissible GRFs in order to obtain a \lk-type formula which is the $d$-dimensional analogon of Theorem \ref{TH:LevyKhinchinFormula1d}. We recap that the characteristic function of a standard Brownian motion $B$ is given by
\begin{align*}
\phi_{B(t)}(\xi)=\mathbb{E}(\exp(i\xi B(t)))=\exp(-\frac{1}{2}t\xi^2) \text{, for }\xi \in \mathbb{R},
\end{align*}
for $t\geq 0$. Note that the Brownian motion ist not characterized by this property, i.e. not every zero-mean GRF on $\mathbb{R}_+$ with the above characteristic function is a Brownian motion, since this specific characteristic function can be attained by different covariance functions, whereas the covariance function of the Browian motion is given uniquely by $q_{BM}(s,t)=Cov(B(s),B(t))=\min\{s,t\}$ for $s,t\geq 0$ (see for example \cite[Section 3.2.2]{LevyProcessesInFinance}). Motivated by this, we impose the following assumptions on the GRF on $\mathbb{R}_+^d$.

\begin{assumption}\label{ASS:GRFCharFct}
Let $W=(W(\underline{x}),~\underline{x}\in \mathbb{R}_+^d)$ be a zero-mean continuous GRF. We assume that there exists a constant $\sigma>0$ such that the characteristic function of $W$ is given by
\begin{align*}
\phi_{W(\underline{x})}(\xi)=\mathbb{E}(\exp(i\xi W(\underline{x}))) = \exp(-\frac{1}{2}\sigma^2\xi^2(x_1+\dots+x_d)),
\end{align*}
for $\xi \in \mathbb{R}$ and every $\underline{x}=(x_1,\dots,x_d)\in \mathbb{R}_+^d$.
\end{assumption}

\begin{rem}\label{REM:StatFieldsExtensionLKFormula}
Note that for a zero-mean, continuous and \textit{stationary} GRF $\tilde{W}=(\tilde{W}(\underline{x}),~\underline{x}\in \mathbb{R}_+^d)$, the GRF $W$ defined by
\begin{align*}
W(\underline{x}):=\sqrt{x_1+\dots+x_d}\tilde{W}(\underline{x}),
\end{align*}
for $\underline{x}=(x_1,\dots,x_d)\in \mathbb{R}_+^d$ satisfies Assumption \ref{ASS:GRFCharFct}.
\end{rem}

We are now able to derive the \lk \space formula for the subordinated GRF.

\begin{theorem}[\lk \space formula]\label{TH:LevyKhinchinFormula}
Let Assumption \ref{ASS:GRFCharFct} hold. We assume independent Lévy subordinators $l_k=(l_k(x),~x\in [0,T_k])$, with Lévy triplets $(\gamma_k,0,\nu_k)$, for $k=1,\dots,d$, are given corresponding to representation \eqref{EQ:CharFctSubordinator}. Further, we assume that these processes are independent of the GRF $W$. We consider the subordinated GRF defined by $L:\Omega \times [0,\mathbb{T}]_d\rightarrow \mathbb{R} \text{ with } L(\underline{x}):=W(l_1(x_1),\dots,l_d(x_d)) \text{ for } \underline{x}=(x_1,\dots,x_d)\in [0,\mathbb{T}]_d$.
The characteristic function of the random field $L$ is then given by
\begin{align*}
\phi_{L(\underline{x})}&(\xi)=\mathbb{E}(\exp(i\xi W(l_1(x_1),\dots, l_d(x_d)))\\
&= \exp\Big(-(x_1,\dots,x_d) \cdot \big(\frac{\sigma^2\xi^2}{2}(\gamma_1,\dots, \gamma_d)^t + \int_{\mathbb{R}\setminus \{0\}} 1 - e^{i\xi z} + i\xi z \mathds{1}_{\{|z|\leq 1\}}(z)\nu_{ext}(dz)\big)\Big), ~\xi\in\mathbb{R},
\end{align*}
for $\underline{x}=(x_1,\dots,x_d)\in[0,\mathbb{T}]_d$. Here, the jump \glqq measure\grqq \space $\nu_{ext}$ is defined through
\begin{align*}
\nu_{ext}([a,b]):=\left(\begin{array}{c} \nu_1^\#([a,b]) \\ \vdots \\  \nu_d^\#([a,b]) \end{array}\right),
\end{align*}
for $a,b\in \mathbb{R}$ where the measure $\nu_k^\#$, for $k =1,\dots,d$ and $a,b\in \mathbb{R}$, is given by
\begin{align*}
\nu_k^\#([a,b]):=\int_0^\infty \int_a^b\frac{1}{\sqrt{2\pi\sigma^2t}}\exp\left(-\frac{x^2}{2\sigma^2t}\right)dx\,\nu_k(dt).
\end{align*}

\end{theorem}

\begin{proof}
For notational simplicity we prove the assertion for $d=2$. For general $d\in\mathbb{N}$ the assertion follows by the same arguments. 

\textit{Claim 1:} For a Lévy measure $\nu$ on $(\mathbb{R}_+,\mathcal{B}(\mathbb{R}_+))$ it holds for every $\xi \in \mathbb{R}$:
\begin{align*}
\int_0^\infty \exp(-\frac{\xi^2}{2}y)-1\nu(dy) = \int_{\mathbb{R}\setminus \{0\}} \exp(i\xi x) - 1 - i\xi x \mathds{1}_{\{|x|\leq 1\}}(x)\nu^\sharp(dx),
 \end{align*}
where the measure $\nu^\sharp$ is defined by $\nu^\sharp(\mathcal{I})=\int_0^\infty \int_a^b \frac{1}{\sqrt{2\pi t}} \exp(-\frac{x^2}{2t})dx \nu(dt)$, for $\mathcal{I}=[a,b]$ with $a,b\in \mathbb{R}$. We use the notation $f_s(x):= \frac{1}{\sqrt{2\pi s}}\exp(-\frac{x^2}{2s})$ for $s>0$ and $x\in \mathbb{R}$ and derive this equation by a direct calculation using the definition of the measure $\nu^\sharp$:
\begin{align*}
\int_{\mathbb{R}\setminus \{0\}} \exp(i\xi x) - 1 &- i\xi x \mathds{1}_{\{|x|\leq 1\}}(x)\nu^\sharp(dx)\\&= \int_{\mathbb{R}\setminus \{0\}}(\exp(i\xi x) - 1 - i\xi x \mathds{1}_{\{|x|\leq 1\}}(x)) \int_0^\infty f_s(x)\nu(ds)dx\\
&=\int_0^\infty \int_{\mathbb{R}\setminus \{0\}}\exp(i\xi x)f_s(x)dx - 1 - i\xi \int_{-1}^1 xf_s(x)dx \nu(ds)\\
&=\int_0^\infty \exp(-\frac{s\xi^2}{2}) - 1\nu(ds).
\end{align*} 
In the last step we used that the characteristic function of a $\mathcal{N}(0,s)$-distributed random variable is given by $\phi(\xi)=\exp(-\frac{s\xi^2}{2})$ for $\xi \in \mathbb{R}$ and $s>0$. Further, we used the fact that $f_s'(x)=-x/sf_s(x)$ to see that 
\begin{align*}
\int_{-1}^1 xf_s(x)=-s(f_s(1)-f_s(-1))=0.
\end{align*}

\textit{Claim 2:} (See \cite[P. 53]{LevyProcessesAndStochasticCalculus}.)
For a Lévy subordinator $l$ with triplet $(\gamma, 0, \nu)$ it holds
\begin{align*}
\mathbb{E}(\exp(-\xi l(t))) = \exp(-t(\gamma \xi + \int_0^\infty (1-\exp(-\xi y))\nu(dy))),
\end{align*}
for $t\geq 0$ and $\xi>0$. 

With these two assertions at hand we can now prove the \lk \space formula. The case $\xi=0$ is trivial since both sides equal $1$ in this case.
Let $(x,y)\in [0,\mathbb{T}]_2$ and $0\neq \xi \in \mathbb{R}$ be fixed. Using Lemma \ref{LE:ExpValueIndepRVs} and Remark \ref{REM:LemmaAlsoComplexFunctions} with $g(\cdot)=\exp(i\xi \cdot)$ and $Z=(l^1(x), l^2(y))$ we calculate
\begin{align*}
\mathbb{E}(\exp(i\xi W(l_1(x),l_2(y)))) = \mathbb{E}(m(l_1(x),l_2(y)))
\end{align*}
where
\begin{align*}
m(x',y'):=\mathbb{E}(\exp(i\xi W(x',y'))) = \exp(-\frac{1}{2}\sigma^2\xi^2(x'+y')) \text{ for }(x',y')\in \mathbb{R}_+^2,
\end{align*}
where we used Assumption \ref{ASS:GRFCharFct}. Therefore, using the independence of the processes $l_1$ and $l_2$ together with Claim 2 we obtain
\begin{align*}
\phi_{L(x,y)}(\xi) &= \mathbb{E}(\exp(-\frac{1}{2}\sigma^2\xi^2(l_1(x) + l_2(y))))\\
&=\mathbb{E}(\exp(-\frac{1}{2}\sigma^2\xi^2l_1(x)))\, \mathbb{E}(\exp(-\frac{1}{2}\sigma^2\xi^2l_2(x)))\\
&=\exp(-x(\gamma_1 \frac{\sigma^2\xi^2}{2} + \int_0^\infty (1-\exp(-\frac{\sigma^2\xi^2}{2}y))\nu_1(dy))) \\
&~~~\cdot \exp(-y(\gamma_2 \frac{\sigma^2\xi^2}{2} + \int_0^\infty (1-\exp(-\frac{\sigma^2\xi^2}{2}y))\nu_2(dy)))\\
&=\exp(-x(\gamma_1 \frac{\sigma^2\xi^2}{2} + \int_0^\infty (1-\exp(-\frac{\xi^2}{2}y))\hat{\nu}_1(dy))) \\
&~~~\cdot \exp(-y(\gamma_2 \frac{\sigma^2\xi^2}{2} + \int_0^\infty (1-\exp(-\frac{\xi^2}{2}y))\hat{\nu}_2(dy))),
\end{align*}
where we define the (Lévy-)measures $\hat{\nu}_1$ and $\hat{\nu}_2$ by $\hat{\nu}_k([a,b])=\nu_k([a/\sigma^2,b/\sigma^2])$ for $a,~b\in \mathbb{R}_+$ and $k=1,2$. Now, using Claim 1 we calculate 
\begin{align*}
\phi_{L(x,y)}(\xi)&=\exp\big(-x(\gamma_1 \frac{\sigma^2\xi^2}{2} - \int_{\mathbb{R}\setminus \{0\}} \exp(i\xi x) - 1 - i\xi x \mathds{1}_{\{|x|\leq 1\}}(x)\hat{\nu}_1^\sharp(dx))) \\
&~~~~~~~~~-y(\gamma_2 \frac{\sigma^2\xi^2}{2} - \int_{\mathbb{R}\setminus \{0\}} \exp(i\xi x) - 1 - i\xi x \mathds{1}_{\{|x|\leq 1\}}(x)\hat{\nu}_2^\sharp(dx))\big),
\end{align*}
where the measures $\hat{\nu}_k^\sharp$ for $k=1,2$ are given by:
\begin{align*}
\hat{\nu}_k^\sharp([a,b])&=\int_0^\infty\int_a^b\frac{1}{\sqrt{2\pi t}}\exp(-\frac{x^2}{2t})dx\hat{\nu}_k(dt)\\
&=\int_0^\infty \int_a^b \frac{1}{\sqrt{2\pi\sigma^2 t}} \exp(-\frac{x^2}{2\sigma^2 t})dx \nu_k(dt),
\end{align*}
for $a,b\in \mathbb{R}$. This finishes the proof.
\end{proof}
 
By the convolution theorem we immediately obtain the following corollary (see \cite[Lemma 15.11 (iv)]{WTheorie}).
\begin{corollary}\label{COR:SubordGRFSumOfLevyProcesses}
Let Assumption \ref{ASS:GRFCharFct} hold. We assume $d$ independent Lévy subordinators $l_k=(l_k(x),~x\in [0,T_k])$ are given for $k=1,\dots,d$, which are independent of $W$  and the corresponding Lévy triplets are given by $(\gamma_k,0,\nu_k)$ for $k=1,\dots,d$. We consider the subordinated GRF $L:\Omega \times [0,\mathbb{T}]_d\rightarrow \mathbb{R}$ defined by $L(\underline{x}):=W(l_1(x_1),\dots,l_k(x_d)) \text{, for } \underline{x}=(x_1,\dots,x_d)\in [0,\mathbb{T}]_d$. Further, we assume that independent Lévy processes $\tilde{l}_k$ on $[0,T_k]$ are given with triplets $(0, \sigma^2\gamma_k/2, \nu_k^\#)$ for $k=1,\dots,d$ in the sense of the one-dimensional \lk \space formula, see Theorem \ref{TH:LevyKhinchinFormula1d}. Here, the Lévy measure $\nu_k^\#$ is defined by
\begin{align*}
\nu_k^\#([a,b]):=\int_0^\infty \int_a^b\frac{1}{\sqrt{2\pi\sigma^2t}}\exp\left(-\frac{x^2}{2\sigma^2t}\right)dx\,\nu_k(dt),
\end{align*}
 for $k =1,\dots, d$ and $a,b\in \mathbb{R}$. The pointwise marginal distribution of the subordinated GRF satisfies
\begin{align*}
L(\underline{x})\stackrel{\mathcal{D}}{=}\tilde{l}_1(x_1) + \dots + \tilde{l}_d(x_d),
\end{align*}
for every $\underline{x}=(x_1,\dots,x_d)\in [0,\mathbb{T}]_d$.
\end{corollary}

\begin{proof}
By Theorem \ref{TH:LevyKhinchinFormula} the characteristic function of $L$ at a fixed point $\underline{x}=(x_1,\dots,x_d)\in [0,\mathbb{T}]_d$ admits the representation

\begin{equation*}
\phi_{L(\underline{x})}(\xi) = \prod_{i=1}^d \exp\Big( 
-\frac{\sigma^2\xi^2}{2}x_i\gamma_i - x_i\int_{\mathbb{R}\setminus \{0\}} 1 - e^{i\xi z} + i\xi z \mathds{1}_{\{|z|\leq 1\}}(z) \nu_i^\# (dz) 
\Big), \text{ for }\xi \in \mathbb{R}.
\end{equation*}
The assertion then follows by the convolution theorem.
\end{proof}

We point out that the case of stationary GRFs is excluded by Assumption \ref{ASS:GRFCharFct}. Therefore, we consider this situation in the following remark where we again assume $d=2$ for notational simplicity.
 
\begin{rem}\label{REM:PointwiseDistStatGRF} Let $W$ be a stationary GRF with covariance function $q((x,y),(x',y')=\tilde{q}((|x-x'|,|y-y'|))$, for $(x,y),~(x',y')$ $\in \mathbb{R}_+^2$, and pointwise variance $\sigma^2:=\tilde{q}((0,0))>0$. Let $l_1$ and $l_2$ be independent Lévy subordinators, which are also independent of W. We obtain by Lemma \ref{LE:ExpValueIndepRVs} the following representation for the characteristic function of the subordinated random field defined by $L(x,y):=W(l_1(x),l_2(y))$, for $(x,y)\in[0,\mathbb{T}]_2$:
\begin{align*}
\phi_{L(x,y)}(\xi)=\mathbb{E}(\exp(i\xi W(l_1(x),l_2(y)))=\mathbb{E}(m(l_1(x),l_2(y))),
\end{align*}
where
\begin{align*}
m(x',y')=\mathbb{E}(\exp(i\xi W(x',y')))=\exp(-\frac{1}{2}\sigma^2\xi^2),
\end{align*}
which is a constant function in $(x',y')$. Therefore we obtain
\begin{align*}
\phi_{L(x,y)}(\xi)=\exp(-\frac{1}{2}\sigma^2\xi^2),
\end{align*}
for $(x,y)\in[0,\mathbb{T}]_2$. Hence, in case of a stationary GRF, the subordinated GRF is pointwise normally distributed with variance $\sigma^2$.
\end{rem} 

We conclude this subsection with a remark on the given \lk \space formula and its meanings.
 
\begin{rem}\label{REM:CharacterizationPropertyLevyKhinchin}
With the approach of subordinating GRFs on a higher-dimensional domain, we obtain a discontinuous Lévy-type random field and a \lk \space formula which allows access to the pointwise distribution of the random field. Further we obtain a similar parametrization of the class of subordinated random fields, as it is the case for Lévy processes on a one-dimensional parameter space: Under the assumptions of Theorem \ref{TH:LevyKhinchinFormula}, every subordinated GRF can be characterized by the tupel $(\sigma^2,\gamma_1,\dots,\gamma_d,\nu_{ext},q)$, where $q:\mathbb{R}_+^d\times\mathbb{R}_+^d\rightarrow \mathbb{R}$ is the covariance function of the GRF. Further, the class of subordinated GRFs is linear in the sense that for the sum of two independent subordinated GRFs one can construct a single subordinated GRF with the same pointwise characteristic function.
\end{rem} 
 

	\section{Covariance function}\label{SEC:CovFct}
	One advantage of the subordinated GRF is that the correlation between spatial points is accessible. The correlation structure is hereby determined by the covariance function of the underlying GRF and the specific choice of the subordinators. For statistical applications it is often important to image or enforce a specific correlation structure in view of fitting random fields to physical phenomena. In this context the question arises whether one can find analytically explicit formulas for the covariance function of a subordinated Gaussian random field. This will be explored in the following section.
	
	For notational simplicity we restrict the dimension to be $d=2$ in this section but we point out that analogous results apply for dimensions $d\geq 3$. A direct application of Lemma \ref{LE:ExpValueIndepRVs} yields the following corollary.
   \begin{corollary}\label{COR:CovFctGeneral}
   Let $W$ be a continuous, zero-mean GRF on $\mathbb{R}_+^2$. Further, let $l_1$ and $l_2$ be two independent Lévy subordinators which are independent of $W$. Then the subordinated GRF $L$ defined by $L(x,y):=W(l_1(x),l_2(y))$, for $(x,y)\in \mathbb{R}_+^2$, is zero-mean with covariance function
   \begin{align*}
   q_L((x,y),(x',y')):=\mathbb{E}(L(x,y)L(x',y'))=\mathbb{E}(q_W((l_1(x),l_2(y)),(l_1(x'),l_2(y')))),
   \end{align*}
   for $(x,y),~(x',y')\in \mathbb{R}_+^2$, where $q_W:\mathbb{R}_+^2\times \mathbb{R}_+^2\rightarrow\mathbb{R}$ denotes the covariance function of the GRF $W$.
   \end{corollary}	
	
\begin{proof}
 For $(x,y)\in [0,\mathbb{T}]_2$ it holds
	\begin{align*}
	\mathbb{E}(L(x,y))=\mathbb{E}(W(l_1(x),l_2(y)))=\mathbb{E}(m(l_1(x),l_2(y)))=0,
	\end{align*}
    since $	m(x',y')=\mathbb{E}(W(x',y'))=0$, for $(x',y')\in \mathbb{R}_+^2$. 
    
    Let $(x,y),~(x',y')\in [0,\mathbb{T}]_2$ be fixed. Another application of Lemma \ref{LE:ExpValueIndepRVs} with $\tilde{W}(x,y,x',y'):=W(x,y)\cdot W(x',y')$, $g=id_d$ and $Z:=(l_1(x),l_2(y),l_1(x'),l_2(y'))$ yields
    \begin{align*}
    q_L((x,y),(x',y'))=\mathbb{E}(L(x,y)L(x',y')) = \mathbb{E}(m(l_1(x),l_2(y),l_1(x'),l_2(y'))),
    \end{align*}
    where the function $m$ is given by
    \begin{align*}
    m(x_1,y_1,x_2,y_2):=\mathbb{E}(W(x_1,y_1)W(x_2,y_2))=q_W((x_1,y_1),(x_2,y_2)),
    \end{align*}
    for $x_1,y_1,x_2,y_2\in \mathbb{R}_+$, which finishes the proof.
\end{proof}	
\subsection{The stationary case}
We use Corollary \ref{COR:CovFctGeneral} to derive a semi-explicit formula for the covariance function of the subordinated GRF, where the underlying GRF is stationary.
\begin{lemma}\label{LE:CovFctStatCase}
Let $W:\mathbb{R}_+^2\rightarrow\mathbb{R}$ be a zero-mean, continuous and stationary GRF with covariance function $q_W((x,y),(x',y'))=\tilde{q}_W(|x-x'|,|y-y'|)$. Further, suppose that $l_1$ and $l_2$ are independent Lévy subordinators on $[0,T_1]$ (resp. $[0,T_2]$) with density functions $f_1$ and $f_2$, i.e. $f_1^x(\cdot)$ (resp. $l_2^y(\cdot)$) is the density function of $l_1(x)$ (resp. $l_2(y)$) for $(x,y)\in [0,\mathbb{T}]_2$. The covariance function of the subordinated GRF $L$ with $L(x,y):=W(l_1(x),l_2(y))$, for $(x,y)\in [0,\mathbb{T}]_2$, admits the representation
\begin{align*}
q_L((x,y),(x',y'))=\int_{\mathbb{R}_+}\int_{\mathbb{R}_+} \tilde{q}_W(s,t)f_1^{|x-x'|}(s)f_2^{|y-y'|}(t)dsdt,
\end{align*}
 for $(x,y),~(x',y')\in [0,\mathbb{T}]_2$ with $x\neq x'$ and $y\neq y'$.
\end{lemma}
For $x=x'$ and $y\neq y'$ it holds 
\begin{align*}
q_L((x,y),(x,y'))=\int_{\mathbb{R}_+}\tilde{q}_W(0,t)f_2^{|y-y'|}(t)dt,
\end{align*}
for $x\neq x'$ and $y=y'$ one obtains 
\begin{align*}
q_L((x,y),(x',y))=\int_{\mathbb{R}_+}\tilde{q}_W(s,0)f_1^{|x-x'|}(s)ds,
\end{align*}
and for $(x,y)=(x',y')$ the pointwise variance is given by 
\begin{equation*}
\var(L(x,y))=q_L((x,y),(x,y))=\tilde{q}(0,0).
\end{equation*}
\begin{proof}
The assertion follows immediately by Corollary \ref{COR:CovFctGeneral} together with the independence of the processes $l_1$ and $l_2$ and the fact that $|l^k(x)-l^k(x')|\stackrel{\mathcal{D}}{=}l^k(|x-x'|)$ for $x,~x'\in[0,T_k]$ and $k=1,2$ by the definition of a Lévy process.
\end{proof}

\subsection{The non-stationary case}
In this subsection, we derive a formula for the covariance function of the subordinated GRF for the case that the underlying GRF is non stationary. The following lemma will be useful in the proof of the covariance representation.
\begin{lemma}\label{LE:JointDensityLevyProcess}
Let $l=(l(x),~x\in[0,T])$ be a general Lévy process with density function $f:[0,T]\times \mathbb{R}\rightarrow\mathbb{R}$, i.e. the probability density function of the random variable $l(x)$ is given by $f^x(\cdot)$, for $x\in[0,T]$. In this case, the joint  probability density function of $Z:=(l(x),l(x'))$ with $x\neq ~x'\in[0,T]$ is given by $f_Z(s,t)=f^{\min(x,x')}(s)\cdot f^{|x'-x|}(t-s)$ for $t,s\in\mathbb{R}$.
\end{lemma}

\begin{proof}Let $x,~x'\in[0,T]$ with $x< x'$ and $x_1,x_2\in \mathbb{R}$ be fixed. The increment $l(x')-l(x)$ is stochastically independent of the random variable $l(x)$, which yields
\begin{align*}
\mathbb{P}(l(x)\leq x_1 \wedge l(x')\leq x_2) &= \mathbb{E}(\mathds{1}_{\{l(x)\leq x_1\}}\mathds{1}_{\{l(x')\leq x_2\}})\\
&=\mathbb{E}(\mathds{1}_{\{l(x)\leq x_1\}}\mathds{1}_{\{l(x')-l(x)\leq x_2-l(x)\}})\\
&=\int_\mathbb{R}\int_\mathbb{R}\mathds{1}_{\{s\leq x_1\}}\mathds{1}_{\{t\leq x_2-s\}}f^x(s)f^{x'-x}(t)dtds\\
&=\int_{-\infty}^{x_1}\int_{-\infty}^{x_2-s}f^x(s)f^{x'-x}(t)dtds\\
&=\int_{-\infty}^{x_1}\int_{-\infty}^{x_2}f^x(s)f^{x'-x}(t-s)dtds.
\end{align*}
For the case that $x'< x$ the same argument yields
\begin{align*}
\mathbb{P}(l(x)\leq x_1 \wedge l(x')\leq x_2)=\int_{-\infty}^{x_1}\int_{-\infty}^{x_2}f^{x'}(s)f^{x-x'}(t-s)dtds,
\end{align*}
which finishes the proof.
\end{proof}

\begin{rem}\label{REM:JointDistributionLevyProcess}
Note that Lemma \ref{LE:JointDensityLevyProcess} immediately implies that the joint density $f_Z(s,t)$ of the two-dimensional random vector $Z=(l(x),l(x'))$ for a Lévy subordinator $l$ on $[0,T]$ and $x\neq x'\in [0,T]$ is given by
\begin{align*}
f_Z(s,t)= f^{\min(x,x')}(s)\cdot f^{|x'-x|}(t), \text{ for } s,t\in\mathbb{R}_+,
\end{align*}
and the joint probability admits the form
\begin{align*}
\mathbb{P}(l(x)\leq x_1 \wedge l(x')\leq x_2)&=\int_{-\infty}^{x_1}\int_{-\infty}^{x_2}f^{\min (x,x')}(s)f^{|x-x'|}(t-s)dtds\\
& = \int_{0}^{x_1}\int_{0}^{x_2}f^{\min (x,x')}(s)f^{|x-x'|}(t-s)d
tds,
\end{align*}
for $x_1,x_2\geq 0$.
\end{rem}

With this lemma at hand we are able to derive a formula for the covariance function of the subordinated (non-stationary) GRF.
\begin{lemma}\label{LE:CovFctNonStatCase}
Let $W:\mathbb{R}_+^2\rightarrow\mathbb{R}$ be a zero-mean, continuous and non-stationary GRF with covariance function $q_W$. Further, suppose that $l_1$ and $l_2$ are independent Lévy subordinators on $[0,T_1]$ (resp. $[0,T_2]$) with density functions $f_1$ and $f_2$, i.e. $f_1^x(\cdot)$ (resp. $l_2^y(\cdot)$) is the density function of $l_1(x)$ (resp. $l_2(y)$) for $(x,y)\in [0,\mathbb{T}]_2$. The covariance function of the subordinared GRF $L$ with $L(x,y):=W(l_1(x),l_2(y))$, for $(x,y)\in [0,\mathbb{T}]_2$, admits the representation
\begin{align*}
q_L((x,y),(x',y'))&=\int_{\mathbb{R}_+}\int_{\mathbb{R}_+}\int_{\mathbb{R}_+} \int_{\mathbb{R}_+}  q_W((x_1,x_2),(x_3,x_4))f_1^{\min(x,x')}(x_1)f_2^{\min(y,y')}(x_2)\\
&\hphantom{{}=\int_{\mathbb{R}_+}\int_{\mathbb{R}_+}\int_{\mathbb{R}_+} \int_{\mathbb{R}_+}}\times f_1^{|x'-x|}(x_3-x_1)f_2^{|y-y'|}(x_4-x_2)dx_1\,dx_2\,dx_3\,dx_4,
\end{align*}
 for $(x,y),~(x',y')\in [0,T]^2$ with $x\neq x'$ and $y\neq y'$. 
 
 For $x=x'$ and $y\neq y'$, it holds
 \begin{align*}
q_L((x,y),(x,y'))&=\int_{\mathbb{R}_+}\int_{\mathbb{R}_+}\int_{\mathbb{R}_+}  q_W((x_1,x_2),(x_1,x_4))f_1^{x}(x_1)f_2^{\min(y,y')}(x_2)\\
&\hphantom{{}=\int_{\mathbb{R}_+}\int_{\mathbb{R}_+} \int_{\mathbb{R}_+}}\times f_2^{|y-y'|}(x_4-x_2)dx_1\,dx_2\,dx_4,
\end{align*}
and for $x\neq x'$ and $y=y'$ it holds
 \begin{align*}
q_L((x,y),(x',y))&=\int_{\mathbb{R}_+}\int_{\mathbb{R}_+}\int_{\mathbb{R}_+}  q_W((x_1,x_2),(x_3,x_2))f_1^{\min(x,x')}(x_1)f_2^{y}(x_2)\\
&\hphantom{{}=\int_{\mathbb{R}_+}\int_{\mathbb{R}_+} \int_{\mathbb{R}_+}}\times f_1^{|x'-x|}(x_3-x_1)dx_1\,dx_2\,dx_3.
\end{align*}
For $(x,y)=(x',y')$ one obtains for the pointwise variance of the field 
\begin{align*}
\var(L(x,y))=q_L((x,y),(x,y))=\int_{\mathbb{R}_+} \int_{\mathbb{R}_+} q_W(x_1,x_2,x_1,x_2)f_1^x(x_1)f_2^y(x_2)dx_1dx_2.
\end{align*}
\end{lemma}

\begin{proof}
Using Corollary \ref{COR:CovFctGeneral}, the independence of the processes $l_1$ and $l_2$ together with Lemma \ref{LE:JointDensityLevyProcess} and Remark \ref{REM:JointDistributionLevyProcess} we calculate for $(x,y),~(x',y')\in [0,\mathbb{T}]_2$ with $x\neq x'$ and $y\neq y'$:
\begin{align*}
q_L((x,y),(x',y'))&=\mathbb{E}(q_W((l_1(x),l_2(y)),(l_1(x'),l_2(y'))))\\
&=\int_{\mathbb{R}_+^4}q_W((x_1,x_2),(x_3,x_4))d\mathbb{P}_{(l_1(x),l_2(y),l_1(x'),l_2(y'))}(x_1,x_2,x_3,x_4)\\
&=\int_{\mathbb{R}_+^2}\int_{\mathbb{R}_+^2} q_W((x_1,x_2),(x_3,x_4))d\mathbb{P}_{(l_1(x),l_1(x'))}(x_1,x_3)d\mathbb{P}_{(l_2(y),l_2(y'))}(x_2,x_4)\\
&=\int_{\mathbb{R}_+}\int_{\mathbb{R}_+}\int_{\mathbb{R}_+} \int_{\mathbb{R}_+}  q_W((x_1,x_2),(x_3,x_4))f_1^{\min(x,x')}(x_1)f_2^{\min(y,y')}(x_2)\\
&\hphantom{{}=\int_{\mathbb{R}_+}\int_{\mathbb{R}_+}\int_{\mathbb{R}_+} \int_{\mathbb{R}_+}}\times f_1^{|x'-x|}(x_3-x_1)f_2^{|y-y'|} (x_4-x_2)dx_1\,dx_2\,dx_3\,dx_4.
\end{align*}
The remaining cases follow by the same argument.
\end{proof}

\subsection{Statistical fitting of the covariance function}
The parametrization property of the subordinated GRF (see Remark \ref{REM:CharacterizationPropertyLevyKhinchin}) motivates a direct approach of covariance fitting: For a natural number $N\in \mathbb{N}$, we assume that discrete points $\{(x_i,y_i),~i=1,\dots,N\}$ are given with corresponding empirical covariance function data $C^{emp}=\{C_{i,j}^{emp},~i,j=1,\dots,N\}$, where $C_{i,j}$ represents the empirical covariance of the field evaluated at the points $(x_i,y_i)$ and $(x_j,y_j)$. We search for the solution to the problem
\begin{align*}
argmin\Big\{\|\tilde{q}_L - C^{emp}\|_\ast  ~\Big|~ \text{ admissible tuples } (\sigma^2,\gamma_1,\gamma_2,\nu_{ext},C)\Big\}
\end{align*}
where we use the notation $\tilde{q}_L:=\{q_L(x_i,y_i),~i,j=1,\dots,N\}$ and $\|\cdot\|_\ast$ is an appropriate norm on $\mathbb{R}^N$, e.g. the euclidian norm. In order to solve this problem, the formulas for the covariance function given by Lemma \ref{LE:CovFctStatCase} and Lemma \ref{LE:CovFctNonStatCase} can be used, but still the solution cannot be found easily due to the complexity of the set of admissible parameters.


		\section{Stochastic regularity - pointwise moments}
		\label{SEC:StochReg}
		In this section we consider pointwise moments of a subordinated GRF $L$. In particular, we  derive conditions which ensure the existence of pointwise $p$-th moments of the subordinated GRF $L$ defined by $L(\underline{x}):=W(l_1(x_1),\dots,l_d(x_d))$, for $\underline{x} = (x_1,\dots,x_d)\in [0,\mathbb{T}]_d$.

Obviously, in order to guarantee the existence of moments of the random variable $L(\underline{x})$, we have to impose conditions on the GRF $W$ \textit{and} the subordinators $l_1,\dots,l_d$. The following theorem gives a better insight into the interaction between the underlying GRF and the stochastic regularity of the subordinators and presents coupled regularity conditions on the tail behaviour of both components of the random field.

\begin{theorem}\label{TH:PointwiseMoments}
We assume that $W$ is a centered and continuous GRF on $\mathbb{R}_+^d$ with covariance function  $q_W:\mathbb{R}_+^d\times \mathbb{R}_+^d\rightarrow \mathbb{R}$. 
Further, we assume that there exist a positive number $N\in \mathbb{N}$, coefficients $\{c_j,~j=1,\dots,N\}\subset [0,+\infty)$ and $d$-dimensional exponents $\{\underline{\alpha}^{(j)},~j=1,\dots,N\}\subset \mathbb{R}_0^d$ such that the pointwise variance function $\sigma_W^2$ of $W$ satisfies
\begin{align}\label{EQ:TailEstGRFVar}
\sigma_W(\underline{x}) =q_W(\underline{x},\underline{x})^{1/2}\leq \sum_{j=1}^N c_j \underline{x}^{\underline{\alpha}^{(j)}}, \text{ for } x_1,\dots,x_d\geq 0.
\end{align}
Here, we use the notation $\underline{x}^{\underline{\alpha}} = x_1^{\alpha_1}\cdot \dots \cdot x_d^{\alpha_d}$ for $\underline{x}=(x_1,\dots,x_d)\in \mathbb{R}_+^d$ and $\underline{\alpha}=(\alpha_1,\dots,\alpha_d)\in \mathbb{R}_0^d$.
We consider a fixed point $\underline{x}\in [0,\mathbb{T}]_d$ and assume that the densities $f_1^{x_1},\dots,f_d^{x_d}$ of the evaluated processes $l_1(x_1),\dots,l_d(x_d)$ fulfill
\begin{align}\label{EQ:TailEstSubord}
f_i^{x_i}(z) \leq C |z|^{-\eta_i},\text{ for } z\geq K,
\end{align}
with decay rates $\{\eta_i,~i=1,\dots,d\}$ and a finite constant $K>0$. We define the number
\begin{align*}
a := \min\{ {(\eta_i-1)}/{\alpha_i^{(j)}}~\big |~ i=1,\dots,d,~j=1,\dots,N, ~\alpha_i^{(j)}\neq 0\}.
\end{align*} Then, the random variable $L(\underline{x})$ admits a $p$-th moment for $p\in[1,a)$, i.e. $L(\underline{x})\in \mathcal{L}^p(\Omega;\mathbb{R})$ for $p\in [1,a)$.
\end{theorem}

\begin{proof}
Let $Z\sim \mathcal{N}(0,\sigma^2)$ be a real-valued, centered, normally distributed random variable with variance $\sigma^2>0$. It follows by Equation (18) in \cite{MomentsAndAbsoluteMomentsOfTheNormalDistribution} that the $p$-th absolute moment of $Z$ admits the formula
\begin{align}\label{EQ:pMomNormalDist}
\mathbb{E}(|Z|^p) = 2^\frac{p}{2} \frac{\Gamma(\frac{p + 1}{2})}{\sqrt{\pi}}\sigma^p =:C_p\sigma^p,
\end{align} 
for all $p>-1$.
Let $p\geq 1$ be a fixed number. We use Lemma \ref{LE:ExpValueIndepRVs} to calculate for the $p$-th moment of $L(\underline{x})$:
\begin{align*}
\mathbb{E}(|L(\underline{x})|^p) &= \mathbb{E}(|W(l_1(x_1),\dots,l_d(x_d))|^p) =\mathbb{E}(m(l_1(x_1),\dots,l_d(x_d))),
\end{align*}
where 
\begin{align*}
m(x_1',\dots,x_d'):=\mathbb{E}(|W(x_1',\dots,x_d')|^p) = C_p\sigma_W^p(x_1',\dots,x_d'),
\end{align*}
for $(x_1',\dots,x_d') \in \mathbb{R}_+^d$ where we used Equation \eqref{EQ:pMomNormalDist}. Hence, we obtain 
\begin{align*}
\mathbb{E}(|L(\underline{x})|^p) = C_p\mathbb{E}(\sigma_W^p(l_1(x_1),\dots,l_d(x_d))).
\end{align*}
Next, we use the tail estmations \eqref{EQ:TailEstGRFVar} and \eqref{EQ:TailEstSubord}, Hölder's inequality and the independence of the subordinators to calculate

\begin{align*}
\mathbb{E}(|L(\underline{x})|^p) & = C_p\mathbb{E}(\sigma_W^p(l_1(x_1),\dots,l_d(x_d))\\
& \leq C_p  \int_{\mathbb{R}_+^d}(\sum_{j=1}^N c_j \underline{z}^{{\underline{\alpha}}^{(j)}})^p f_1^{x_1}(z_1) \dots  f_d^{x_d}(z_d)d(z_1,\dots,z_d)\\
&\leq C(N,p) \sum_{j=1}^N c_j^p  \prod_{i=1}^d \underbrace{\int_0^{+\infty}  z_i^{p\alpha_i^{(j)}} f_i^{x_i}(z_i)dz_i}_{=:I_i^j}.
\end{align*}
It remains to show that all the integrals $I_i^j$ are finite. For $i\in\{1,\dots,d\}$ and $j\in\{1,\dots,N\}$ with $\alpha_i^{(j)}=0$ we have $I_i^j=1$. If $\alpha_i^{(j)}\neq 0$ it holds
\begin{align*}
I_i^j &= \Big(\int_0^{K} + \int_K^{+\infty}\Big)  z_i^{p\alpha_i^{(j)}} f_i^{x_i}(z_i)dz_i\\
&\leq K^{p\alpha_i^{(j)}} + C\int_K^{+\infty}  z_i^{p\alpha_i^{(j)} - \eta_i}dz_i<+\infty,
\end{align*}
where the integral in the last step is finite since $p\alpha_i^{(j)}-\eta_i < -1$ for all $i\in\{1,\dots,d\}$ and $j\in\{1,\dots,N\}$ with $\alpha_i^{(j)}\neq 0$.
\end{proof}
We close this section with three remarks on the assumptions and possible extensions of Theorem \ref{TH:PointwiseMoments}.
\begin{rem}\label{REM:PointwiseStochRegMatern}
The assumption given by Equation \eqref{EQ:TailEstGRFVar} is, for example, fulfilled for the $d$-dimensional Brownian sheet with $N=1$, $c_1=1$ and $\alpha^{(1)}=(1/2,\dots,1/2)\in \mathbb{R}_+^d$. Condition \eqref{EQ:TailEstGRFVar} also accomodates the GRFs we considered in the Lévy-Khinchin formula (see Theorem \ref{TH:LevyKhinchinFormula} and Assumption \ref{ASS:GRFCharFct}) with $N=d$, $c_1=\dots=c_d=1$ and $\underline{\alpha}^{(j)}=1/2\cdot \hat{e}_j$ for $j=1,\dots,d$, where $\hat{e}_j$ is the $j$-th unit vector in $\mathbb{R}^d$. Further, this assumption is fulfilled for any stationary GRF $W$. Indeed, in case of a stationary GRF the assumption is satisfied for $\alpha^{(1)}=(\varepsilon,0,\dots,0)$ for any $\varepsilon>0$ and, hence, Theorem \ref{TH:PointwiseMoments} yields that every moment of the corresponding evaluated subordinated GRF exists, independently of the specific choice of the subordinators. This is consistent with Remark \ref{REM:PointwiseDistStatGRF}. The assumption on the Lévy subordinators in Equation \eqref{EQ:TailEstSubord} is natural and can be verified easily in many cases, see also \cite[Assumption 3.7 and Remark 3.8]{ApproximationAndSimulation}.
\end{rem}

\begin{rem}\label{REM:PointwiseStochRegDiscrDist}
We point out that the statement of Theorem \ref{TH:PointwiseMoments} remains valid if we consider Lévy distributions with discrete probability distribution which satisfy a discrete version of \eqref{EQ:TailEstSubord}: If the GRF $W$ satisfies  \eqref{EQ:TailEstGRFVar} and the evaluated (discrete) subordinators $l_1(x_1),\dots,l_d(x_d)$ satisfy 
\begin{align}\label{EQ:TailEstSubordDiscrete}
f_i^{x_i}(k) = \mathbb{P}(l_i(x_i)=k)\leq C|k|^{-\eta_i}, \text{ for } k\geq K \text{ and } i\in\{1,\dots,d\},
\end{align}
then we obtain that $\mathbb{E}(|L(x_1,\dots,x_d)|^p)<\infty$ for $p\in [1,a)$ with the real number $a$ defined in Theorem \ref{TH:PointwiseMoments}.
\end{rem}
 
 \begin{rem}\label{REM:PointwiseStochRegGeneralLavyPr}
For the pointwise existence of moments given by Theorem \ref{TH:PointwiseMoments}, it is not necessary to restrict the subordinating processes to the class of Lévy subordinators. More generally, one could consider a GRF $W$ satisfying \eqref{EQ:TailEstGRFVar} and \textit{general} Lévy processes $l_1,\dots,l_d$ satisfying \eqref{EQ:TailEstSubord} for $|z|\geq K$. In this case, Theorem \ref{TH:PointwiseMoments} still holds for the random field $L$ defined by $L(\underline{x}) := W(|l_1(x_1)|,\dots,|l_d(x_d)|)$, for $\underline{x} = (x_1,\dots,x_d)\in [0,\mathbb{T}]_d$.
 \end{rem}

	
\section{Numerical Examples}\label{SEC:NumEx}
In this section we present numerical experiments about the theoretical results given in this paper. We verify the \lk \space formula which allows access to the pointwise distribution of a subordinated GRF and validate the formulas for the covariance functions given in Section \ref{SEC:CovFct}. Further, we introduce an approach of marginal-distribution-fitting for the subordinated GRF and present numerical experiments addressing the pointwise stochastic regularity of these fields (see Section \ref{SEC:StochReg}). 

\subsection{Verification of the \lk \space formula}
In this subsection we investigate the pointwise distribution of subordinated GRFs in order to verify the \lk \space formula given by Theorem \ref{TH:LevyKhinchinFormula}. To be more precise, we use Corollary \ref{COR:SubordGRFSumOfLevyProcesses} to obtain a pointwise distributional representation of a subordinated GRF as the sum of one-dimensional Lévy processes with transformed Lévy triplets.

Assume $L=(W(l_1(x),l_2(y)),~(x,y)\in [0,1]^2)$ is a subordinated GRF where the GRF $W$ satisfies Assumption \ref{ASS:GRFCharFct} and the two subordinators $l_1$ and $l_2$ are characterized by the Lévy triplets $(\gamma_k,0,\nu_k)$ for $k=1,2$. It follows by Corollary \ref{COR:SubordGRFSumOfLevyProcesses} that $L$ admits the distributional representation
\begin{align}\label{EQ:NumExPWDist}
L(x,y)\stackrel{\mathcal{D}}{=}\tilde{l}_1(x) +\tilde{l}_2(y),
\end{align}
for $(x,y)\in [0,1]^2$. Here, the processes $\tilde{l}_k$ on $[0,1]$ are independent Lévy processes with triplets $(0, \sigma^2\gamma_k/2, \nu_k^\#)$ for $k=1,2$ in the sense of the one-dimensional \lk \space formula  (see Theorem \ref{TH:LevyKhinchinFormula1d}) and the Lévy measure $\nu_k^\#$ is defined by
\begin{align*}
\nu_k^\#([a,b]):=\int_0^\infty \int_a^b\frac{1}{\sqrt{2\pi\sigma^2t}}\exp\left(-\frac{x^2}{2\sigma^2t}\right)dx\,\nu_k(dt),
\end{align*}
$a,b\in \mathbb{R}$ and $k=1,2$.
In order to validate Equation \eqref{EQ:NumExPWDist} we choose specific spatial points and sample the subordinated GRF $L$ to compare it with the distribution given by the right hand side of \eqref{EQ:NumExPWDist}. We use two different methods to approximate the distribution of the Lévy processes on the right hand side of \eqref{EQ:NumExPWDist}: the compound Poisson approximation (CPA) (see \cite[Section 8.2.1]{LevyProcessesInFinance}) and the Fourier inversion method for Lévy processes (see \cite{NoteOnTheInversionTheorem} and \cite{ApproximationAndSimulation}) which allows for a direct approximation of the density of the right hand side of \eqref{EQ:NumExPWDist}.

\subsubsection{First approach: CPA}\label{SUBSUBSEC:CPA}\label{SUBSUBSEC:NumExCPA}
We recall that a $Gamma(a_G,b_G)$ process $l_G$ has independent $Gamma$-distributed increments and $l_G(t)$ follows a $Gamma(a_G\cdot t,b_G)$ distribution. In our first example we use $Gamma(4,12)$ processes to subordinate the GRF $W$ which is defined by $W(x,y)=\sqrt{x+y}\,\tilde{W}(x,y)$, for $(x,y)\in [0,1]^2$, where $\tilde{W}$ is a Matérn-1.5-GRF with pointwise standard deviation $\sigma=2$ (see Remark \ref{REM:StatFieldsExtensionLKFormula}). We fix the evaluation point $(x,y)=(1,1)$ and use the CPA method to obtain samples of the Lévy process on the right hand side of~\eqref{EQ:NumExPWDist} which can then be compared with samples of the subordinated GRF. Figure~\ref{FIG:CPACompareHist} (left and middle) shows the corresponding histograms for $10.000$ samples of each distribution.

	\begin{figure}[ht]
	\centering
	\subfigure{\includegraphics[scale=0.32]{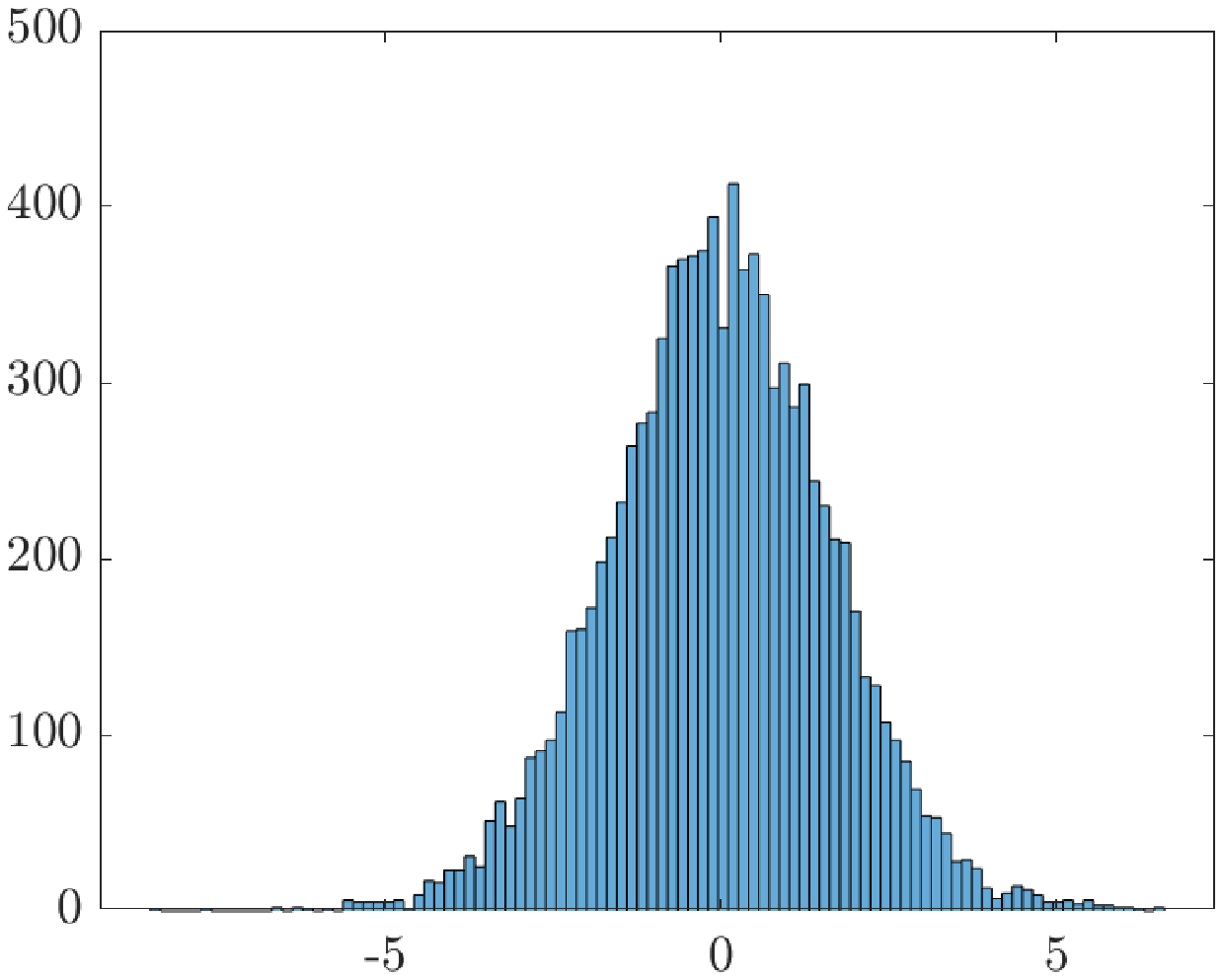}}
	\subfigure{\includegraphics[scale=0.32]{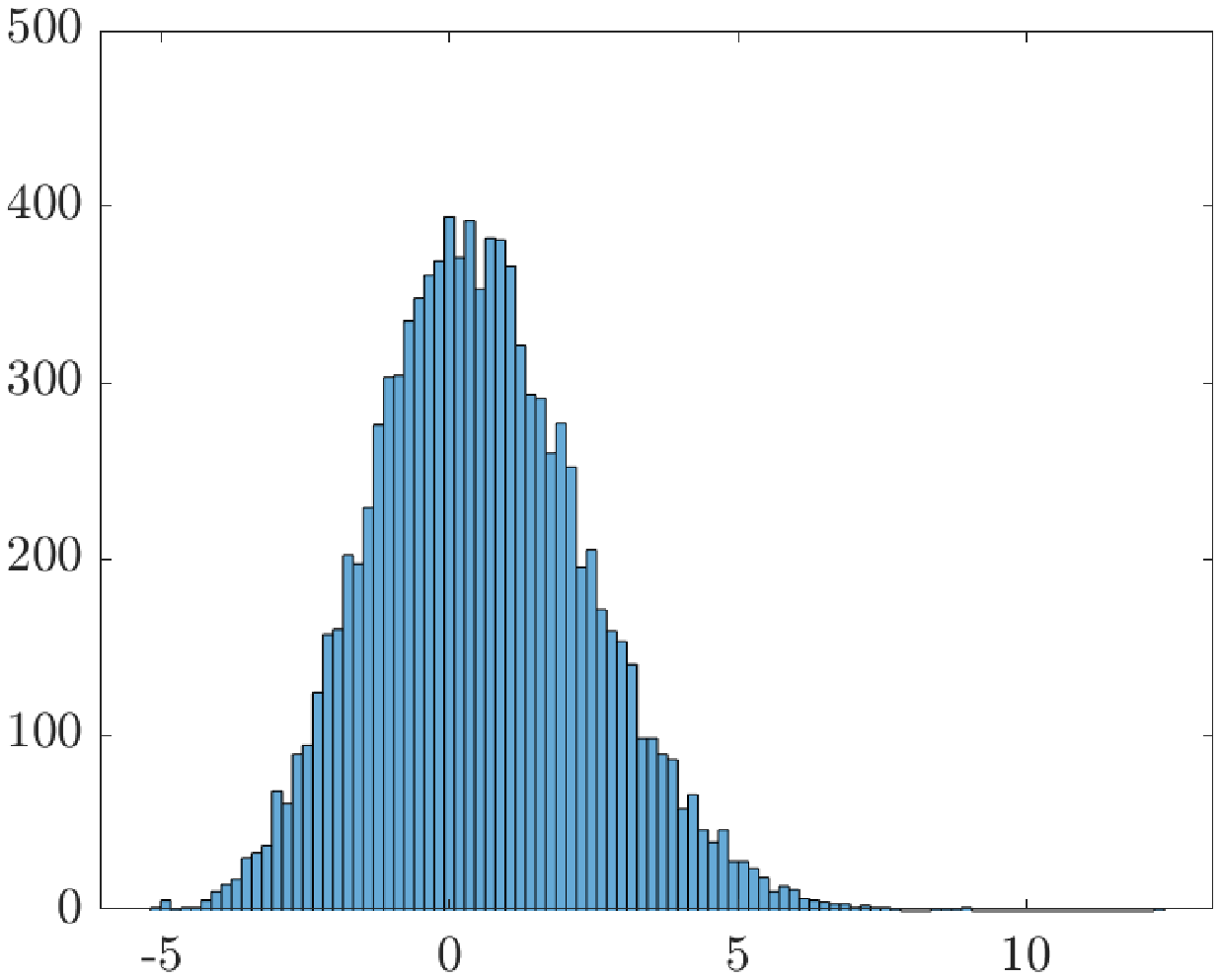}}
    \subfigure{\includegraphics[scale=0.32]{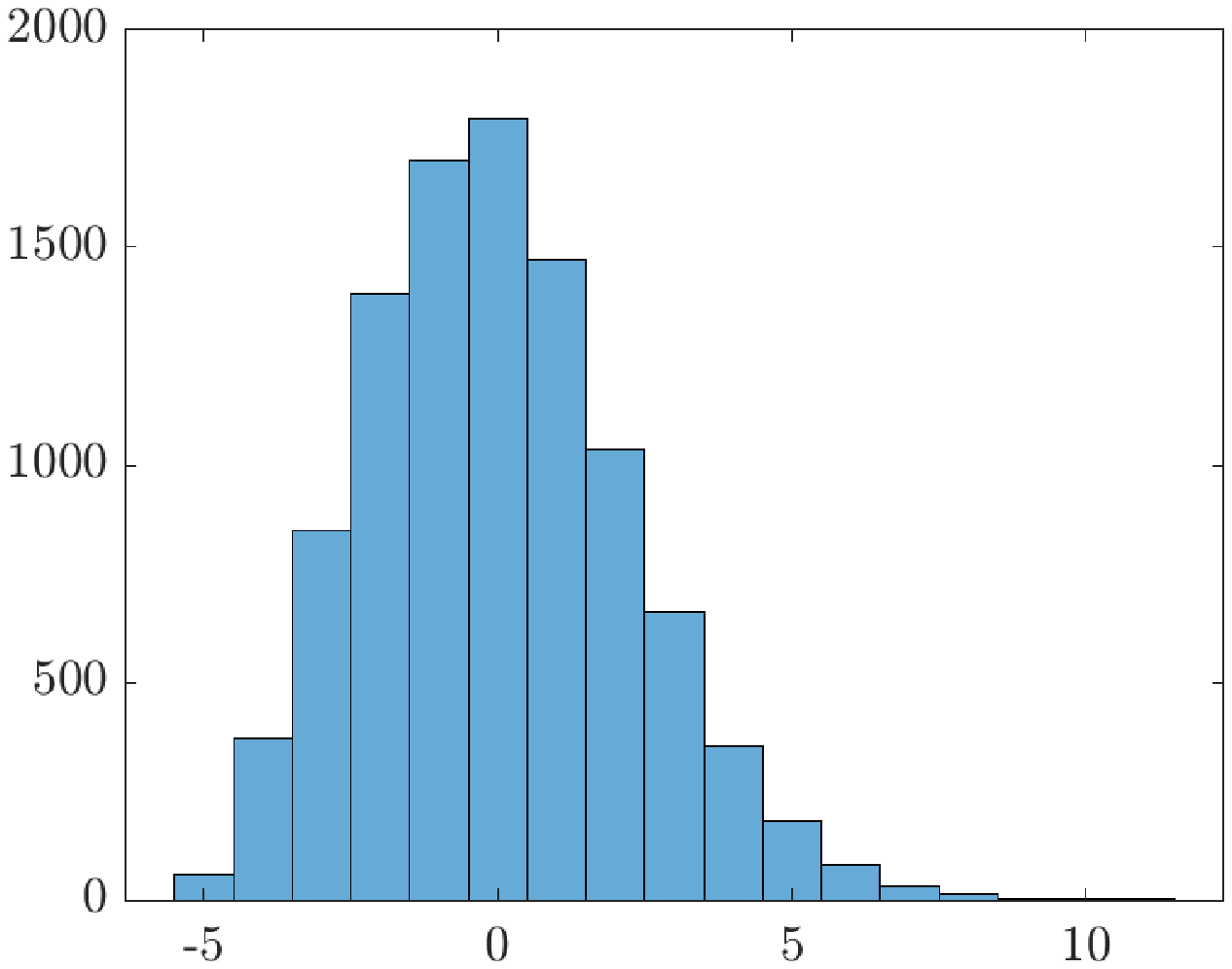}}
    \caption{Samples of the subordinated GRF $W(l_1(1),l_2(1))$ (left), the sum of the corresponding tranformed Lévy processes $\tilde{l}_1(1) + \tilde{l}_2(1)$ using the CPA method (middle) and $10.000$ Samples of $Poiss(5)-5$ (right).}
    \label{FIG:CPACompareHist}
    \end{figure}
Visually, we obtain a relatively accurate fit of the distributions since the histograms have similar characteristics. However, in contrast to the left histogram, the distribution corresponding to the middle histogram in Figure \ref{FIG:CPACompareHist} is not symmetric. This is not due to the specific choice of the CPA parameters but due to the method itself since in CPA a Lévy process is approximated essentially by the sum of compensated Poisson processes which is never symmetric as one can see in the right histogram of Figure~\ref{FIG:CPACompareHist}.


We conclude that, even if the histograms indicate that the underlying distributions match, the CPA is not perfectly suitable to approximate the pointwise marginal distribution of the Lévy processes given on the right hand side of Equation \eqref{EQ:NumExPWDist} since the CPA cannot image the symmetry properties of this distribution.

\subsubsection{Second approach: Fourier inversion method}\label{SUBSUBSEC:NumExPWDistFI}

The second approach is to approximate the density function of the right hand side of \eqref{EQ:NumExPWDist} by the Fourier inversion (FI) method (see \cite{NoteOnTheInversionTheorem} and \cite{ApproximationAndSimulation}) and compare it with samples of the subordinated GRF. Figure \ref{FIG:FIGamma412Hist} illustrates the results for this approach where we used the evaluation point $(x,y)=(1,1)$, the same GRF as in Subsection \ref{SUBSUBSEC:CPA}, $Gamma(4,12)$ subordinators and $100.000$  samples of the subordinated GRF.

\begin{figure}[ht]
	\centering
	\subfigure{\includegraphics[scale=0.4]{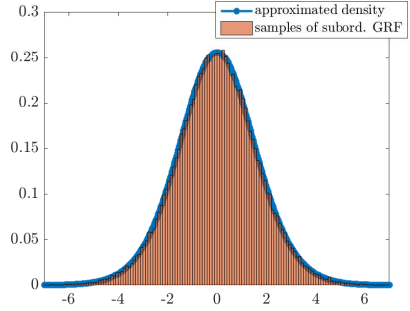}}
    \caption{Samples of $Gamma(4,12)$-subordinated GRF and approximated density (FI).}
    \label{FIG:FIGamma412Hist}
    \end{figure}

As one can see in Figure \ref{FIG:FIGamma412Hist}, the pointwise distribution of the subordinated GRF perfectly matches the approximated density of the right hand side of \eqref{EQ:NumExPWDist}. We want to confirm this by a Kolmogorov-Smirnov-Test (see for example \cite[Section VII.4]{PestmanMathematicalStatistics}). Figure \ref{FIG:KSFIGam412} illustrates how the empirical CDF, obtained by sampling the subordinated GRF, converges to the target CDF which is computed by the Fourier inversion method using Equation \eqref{EQ:NumExPWDist}. A Kolmogorov-Smirnov-test with $10.000$ samples and a level of significance of $5\%$ is passed.

\begin{figure}[ht]
	\centering
	\subfigure{\includegraphics[scale=0.33]{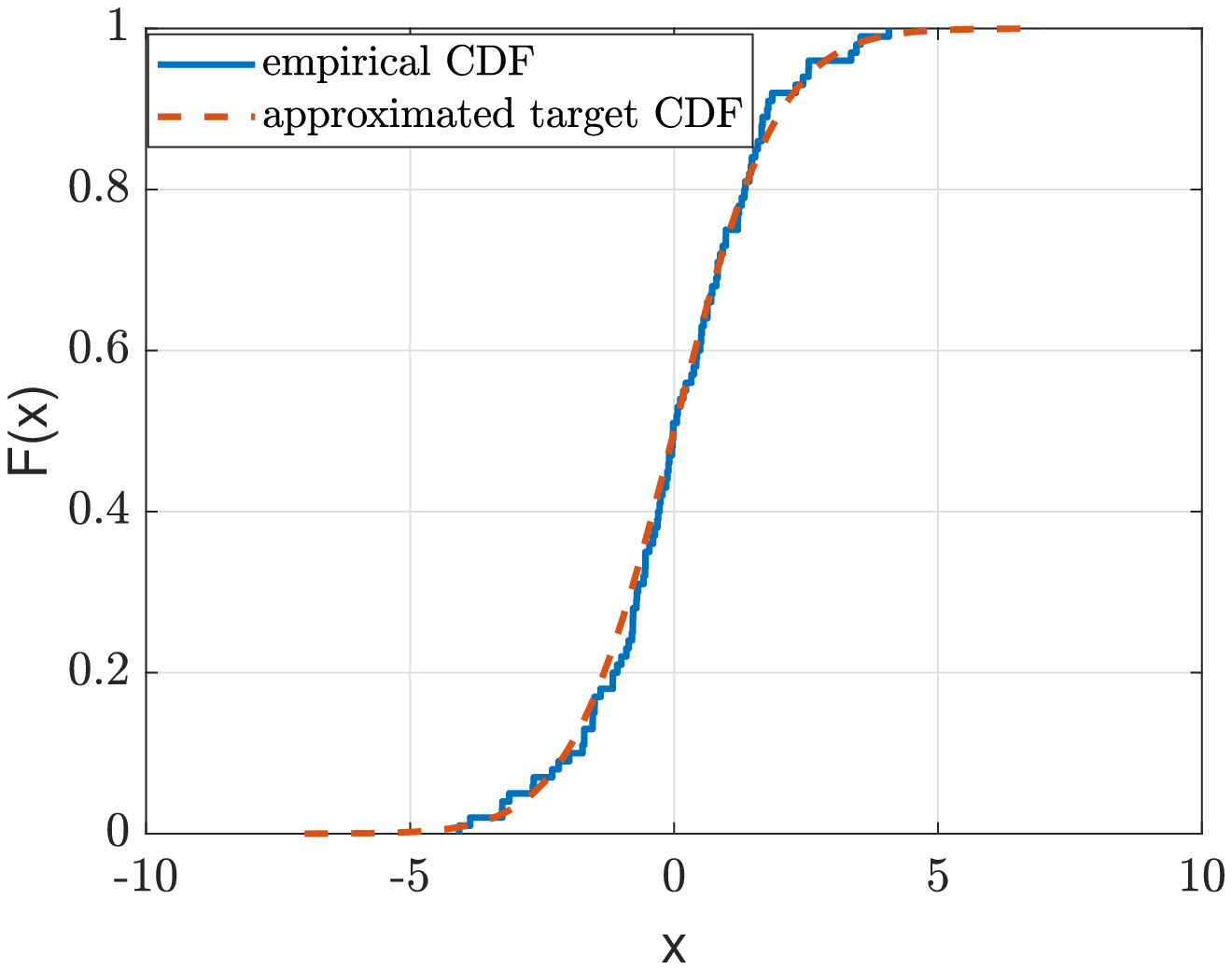}}
	\subfigure{\includegraphics[scale=0.33]{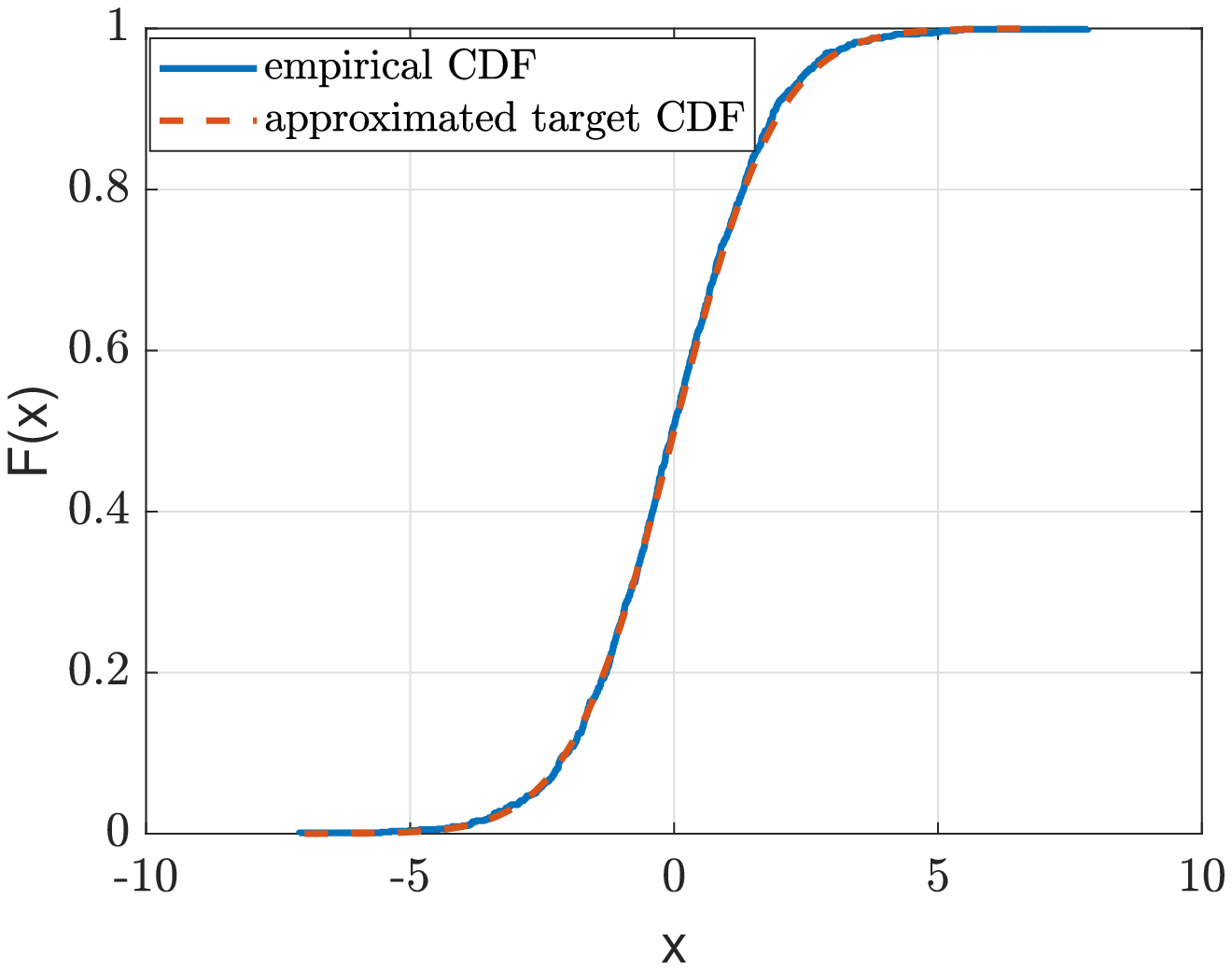}}
	\subfigure{\includegraphics[scale=0.33]{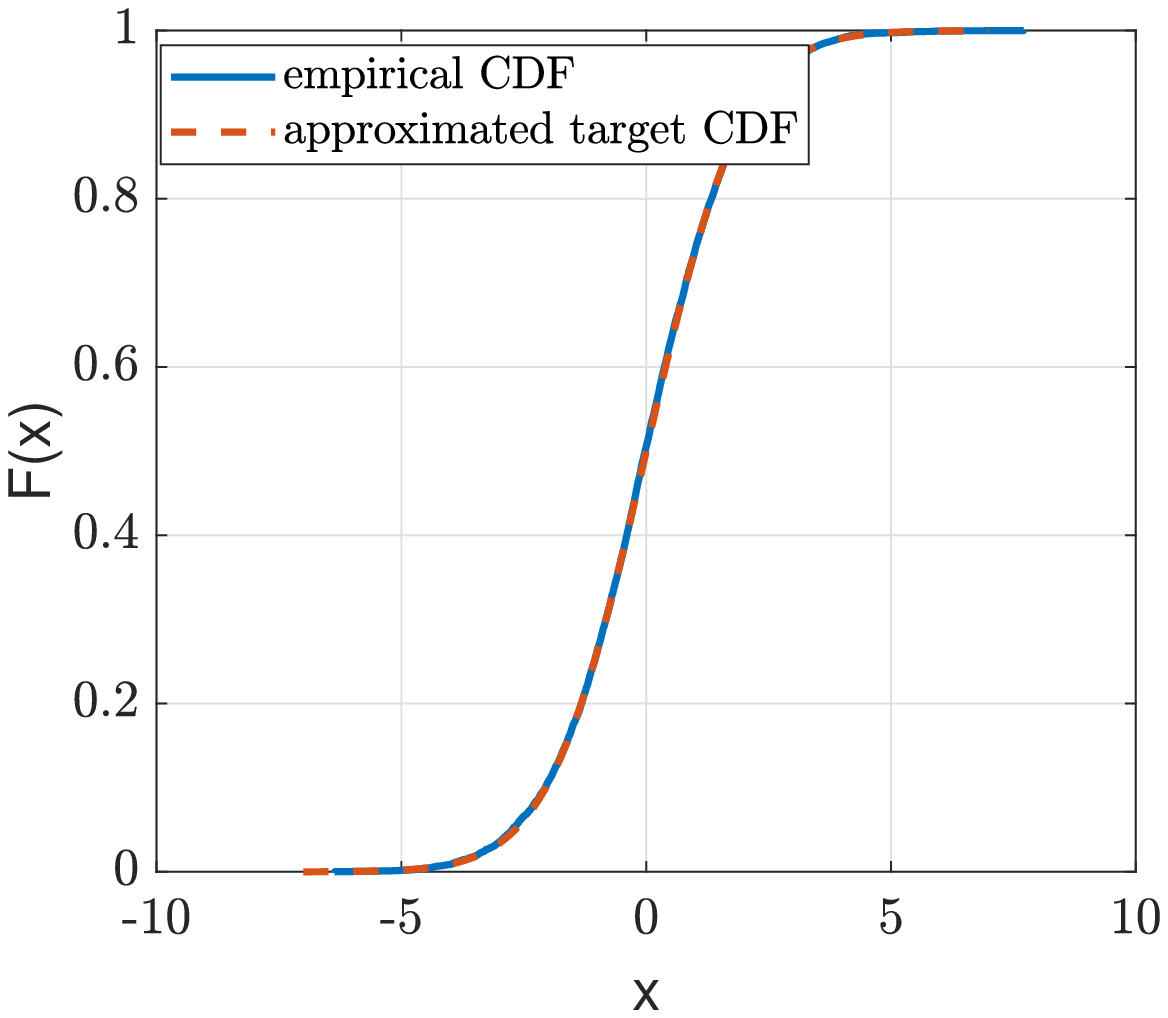}}
    \caption{Approximated target CDF (FI) vs. empirical CDF using $100$ (left), $1.000$ (middle) and $10.000$ (right) samples of the subordinated GRF with $Gamma(4,12)$ subordinators.}
    \label{FIG:KSFIGam412}
    \end{figure}

In the next experiment we use a modified subordinator, which results in a less smooth pointwise density of the subordinated GRF. We repeat the above experiment with $Gamma(0.5,10)$ subordinators where the GRF, the evaluation point and the sample size remain unchanged. Figure \ref{FIG:FIG0510Hist} shows $100.000$ samples of the subordinated GRF and the density of the process given by the right hand side of Equation \eqref{EQ:NumExPWDist} computed via the Fourier Inversion method.

\begin{figure}[ht]
	\centering
   \subfigure{\includegraphics[scale=0.4]{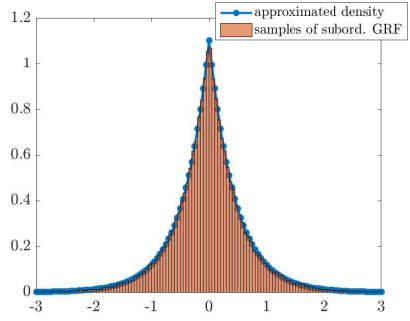}}
    \caption{Samples of $Gamma(0.5,10)$-subordinated GRF and approximated density (FI).}
    \label{FIG:FIG0510Hist}
    \end{figure}

As in the first experiment, the results given by Figure \ref{FIG:FIG0510Hist} indicate that the pointwise distribution of the subordinated GRF matches the approximated density of the right hand side of \eqref{EQ:NumExPWDist}. Figure \ref{FIG:KSFIGam0510} illustrates how the empirical CDF, obtained by sampling of the subordinated GRF, converges to the target CDF of the right hand side of Equation \eqref{EQ:NumExPWDist}, which is computed by the Fourier inversion method. A Kolmogorov-Smirnov-test with a level of significance of $5\%$ is passed.

\begin{figure}[ht]
	\centering
	\subfigure{\includegraphics[scale=0.3]{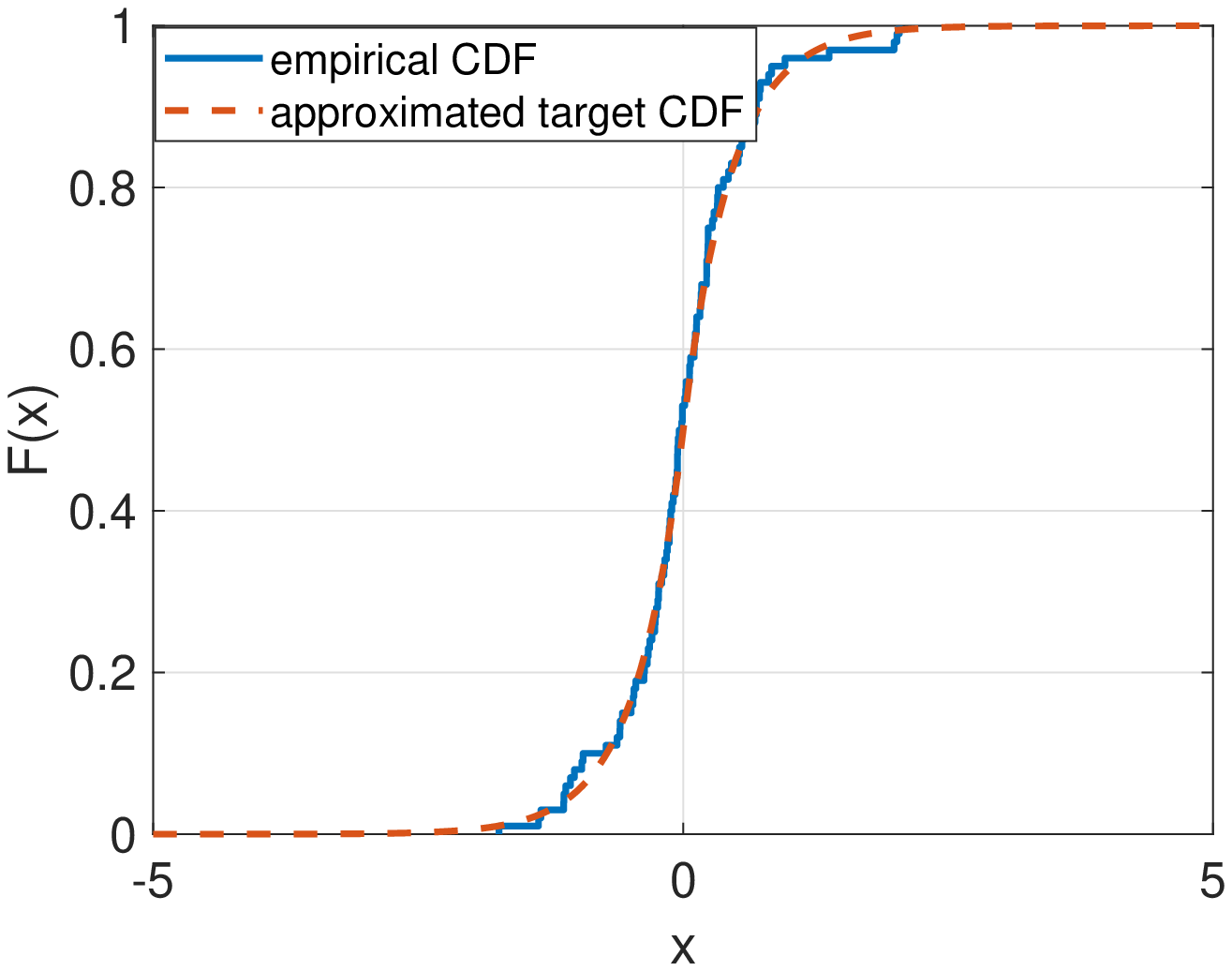}}
   \subfigure{\includegraphics[scale=0.3]{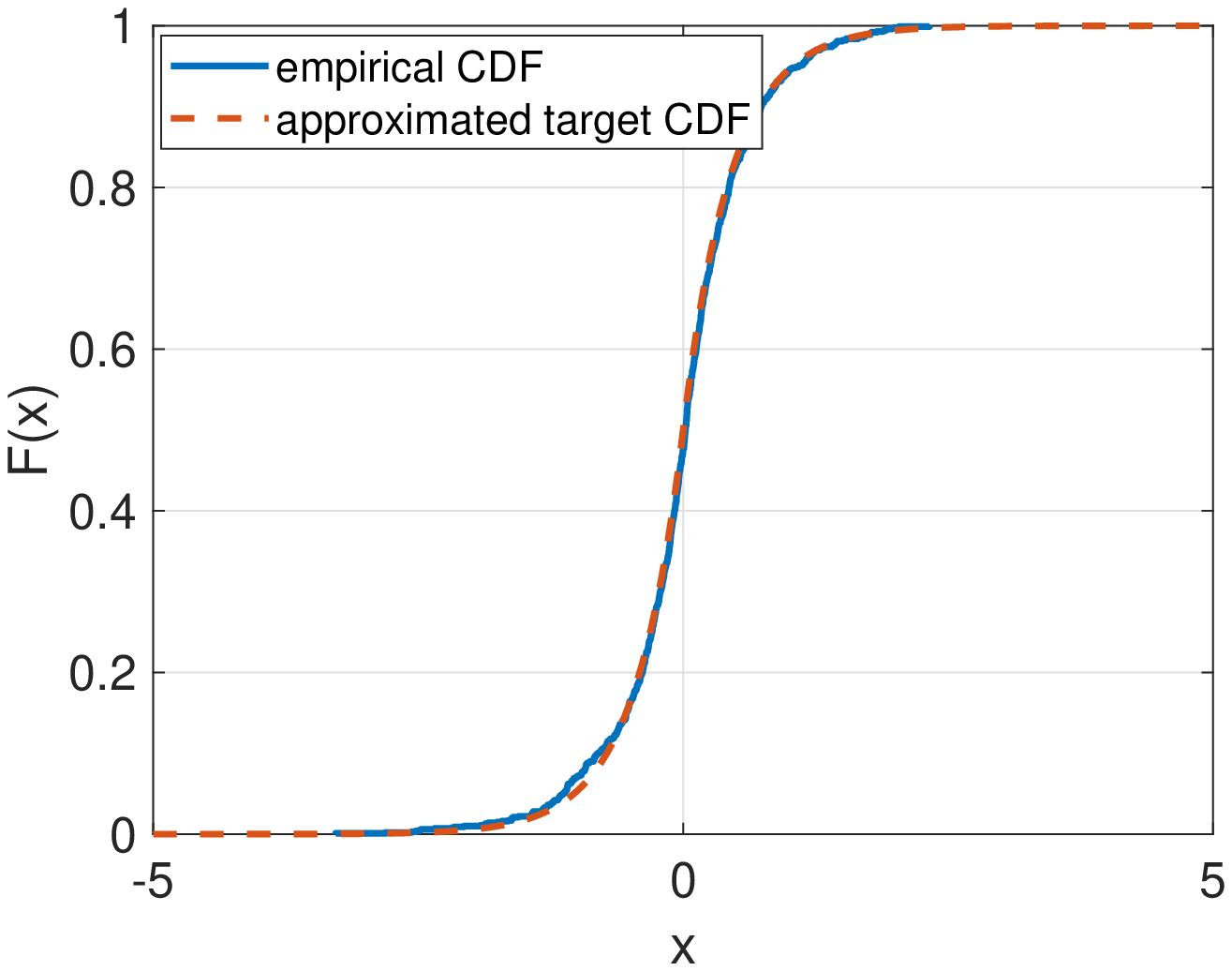}}
	\subfigure{\includegraphics[scale=0.3]{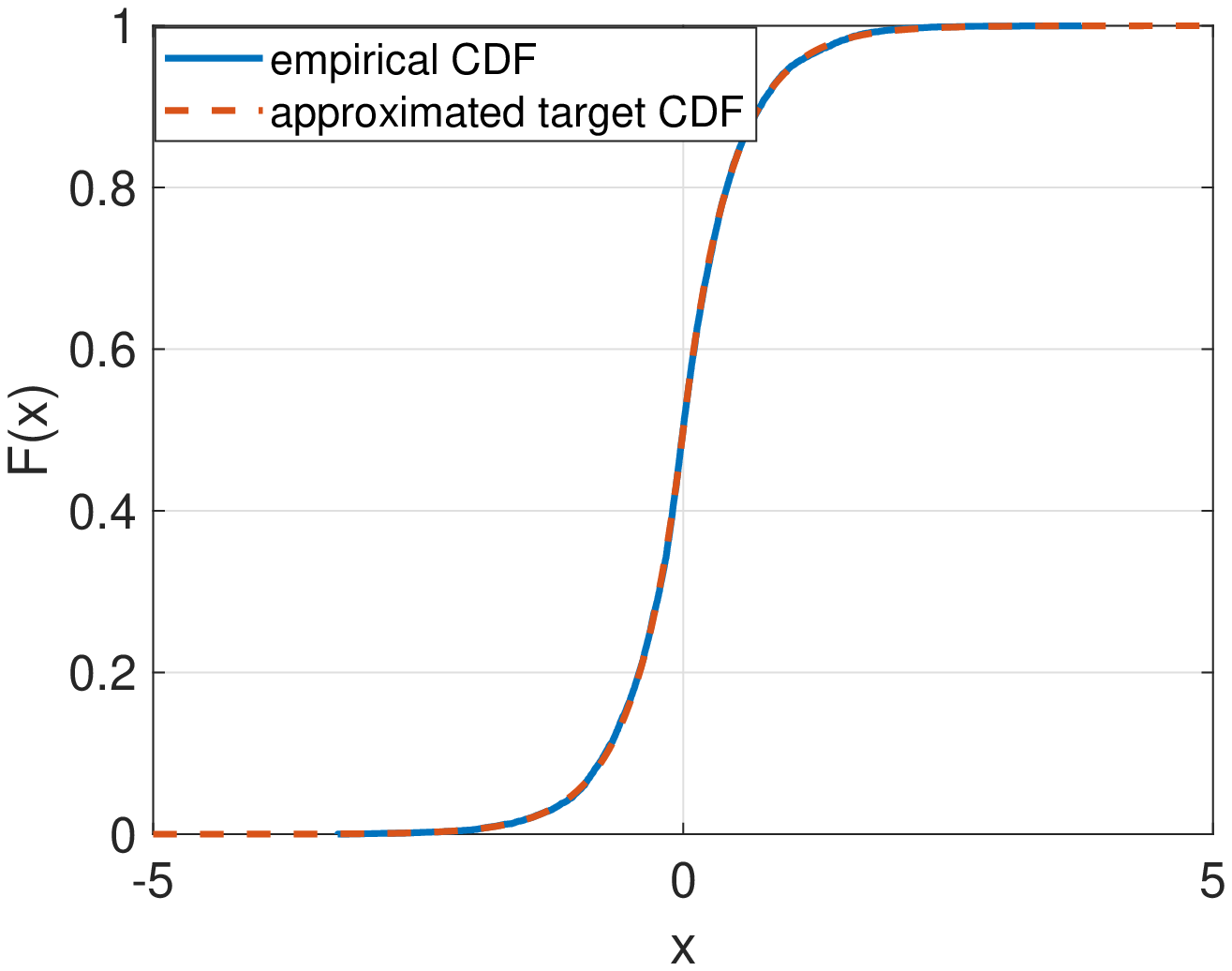}}
    \caption{Approximated target CDF (FI) vs. empirical CDF using $100$ (left), $1.000$ (middle) and $10.000$ (right) samples of the subordinated GRF with $Gamma(0.5,10)$ subordinators.}
    \label{FIG:KSFIGam0510}
    \end{figure}

\subsection{Verification of Covariance formulas}\label{SUBSEC:NumExCovFormula}
In Section \ref{SEC:CovFct} we derived semi-explicit formulas for the covariance function of centered subordinated GRFs. In the following subsection we validate these fomulas by comparing them with numerically estimated covariances using specific GRFs and subordinators. We use both, stationary and non-stationary underlying GRFs.

In our numerical experiments, we evaluate the integrals appearing in the derived formulas for the covariance function via the trapezoidal rule. Further, we estimate the empirical covariance of the subordinated GRF by a standard Monte Carlo (MC) estimation: for spatial points $(x,y),~(x',y')\in [0,\mathbb{T}]_2$ and a sample number $M\in \mathbb{N}$ we estimate
\begin{align*}
q_L((x,y),(x',y'))\approx \frac{1}{M}\sum_{i=1}^M W(l_1(x),l_2(y))^{(i)}\cdot W(l_1(x'),l_2(y'))^{(i)},
\end{align*}
where $(W(l_1(x),l_2(y))^{(i)},~i=1,\dots,M)$ (resp. $(W(l_1(x'),l_2(y'))^{(i)},~i=1,\dots,M)$ are i.i.d. copies of the subordinated GRF evaluated at $(x,y)$ (resp. $(x',y')$). In order to interpret our results we state the following standard result.

\begin{lemma}\label{LE:MCConv}
Let $M\in \mathbb{N}$ and $Z\in \mathcal{L}^2(\Omega)$ be a random variable with finite second moment. The standard Monte Carlo estimator $E_M(Z):=1/M\sum_{i=1}^M Z^{(i)}$, with i.i.d. copies $(Z^{(i)},~i=1,\dots,M)$ of $Z$, satisfies
\begin{align*}
\|\mathbb{E}(Z) - E_M(Z)\|_{\mathcal{L}^2(\Omega)} \leq \frac{\|Z\|_{\mathcal{L}^2(\Omega)}}{\sqrt{M}}.
\end{align*} 
\end{lemma}

\subsubsection{The stationary case}\label{SUBSUBSEC:NumExCovStat}

In order to verify Lemma \ref{LE:CovFctStatCase}, we compute a Monte Carlo estimation for the covariance of a subordinated (stationary) GRF and compare it with the right hand side of the formula given by Lemma \ref{LE:CovFctStatCase}. We use a Matérn-1.5-GRF with parameters $\sigma=r=0.5$. Further, $l_j$ is a $Gamma(a_G^{(j)},b_G^{(j)})$-subordinator for $j=1,2$, where we vary the parameters $\underline{a}_G=(a_G^{(1)},a_G^{(2)})$ and $\underline{b}_G=(b_G^{(1)},b_G^{(2)})$. We consider the covariance of the field at different points to cover different situations: points with small distance, points, which are equal in one coordinate, and points with large distance. 


We use the 500 independent MC runs to approximate the RMSE:
\begin{align*}
RMSE^2&:=\|q_L((x,y),(x',y')) - E_M\big(W(l_1(x),l_2(y))\cdot W(l_1(x'),l_2(y'))\big)\|_{\mathcal{L}^2(\Omega)}^2\\
   &\approx \frac{1}{500}\sum_{i=1}^{500} \Big(q_L((x,y),(x',y')) - E_M^{(i)}\big(W(l_1(x),l_2(y))\cdot W(l_1(x'),l_2(y'))\big)\Big)^2 ,
\end{align*} 
where $(E_M^{(i)}(W(l_1(x),l_2(y)) W(l_1(x'),l_2(y'))),~i=1,\dots,500)$ denote the 500 independent MC estimations.
The results are presented in Figure \ref{FIG:CovEstStatMCConv}, which shows the expected Monte Carlo convergence of order $\mathcal{O}(1/\sqrt{M})$ in all considered situations (see Lemma \ref{LE:MCConv}).

\begin{figure}[ht]
	\centering
	\subfigure{\includegraphics[scale=0.3]{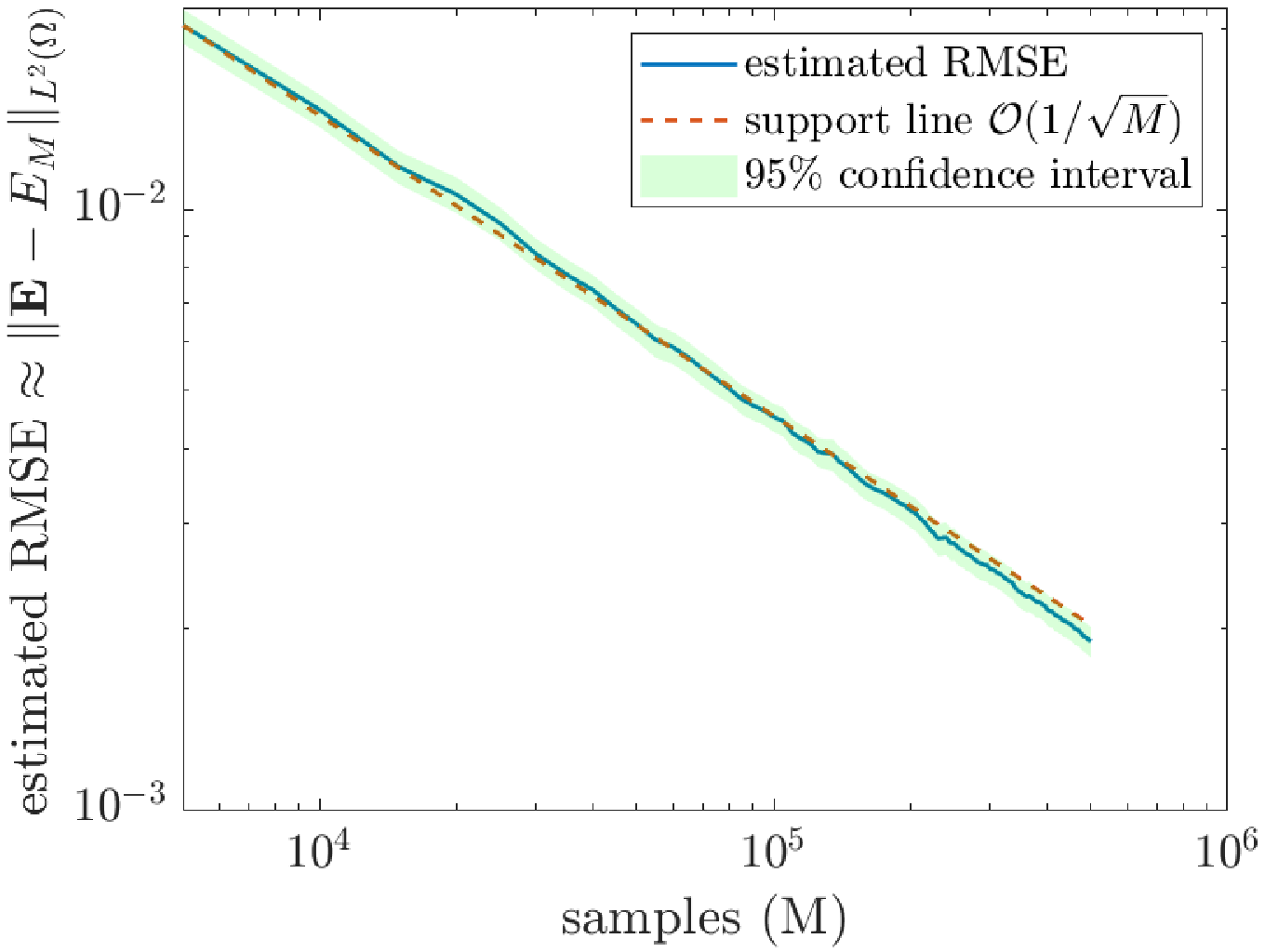}}
\subfigure{\includegraphics[scale=0.3]{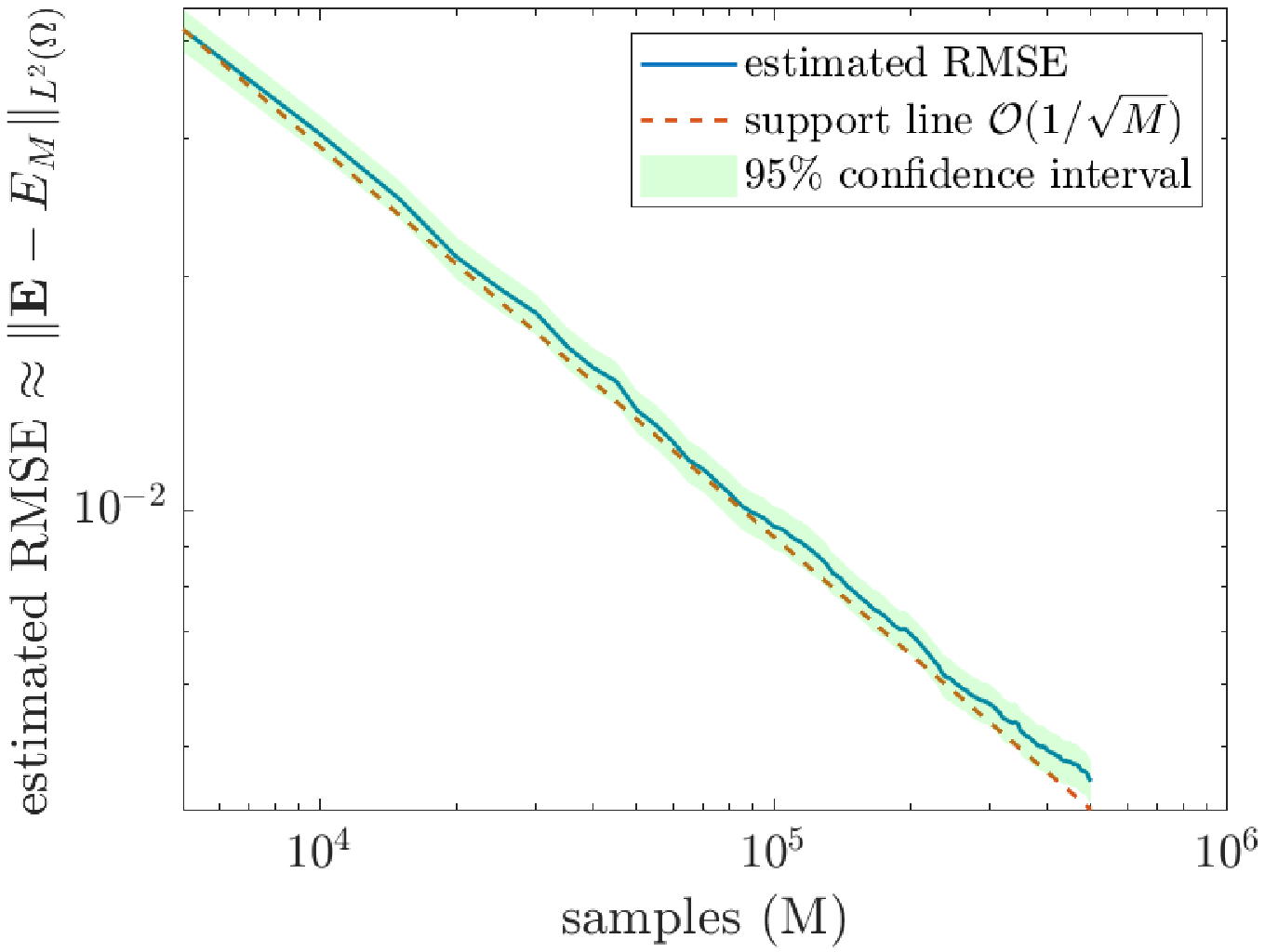}}
\subfigure{\includegraphics[scale=0.3]{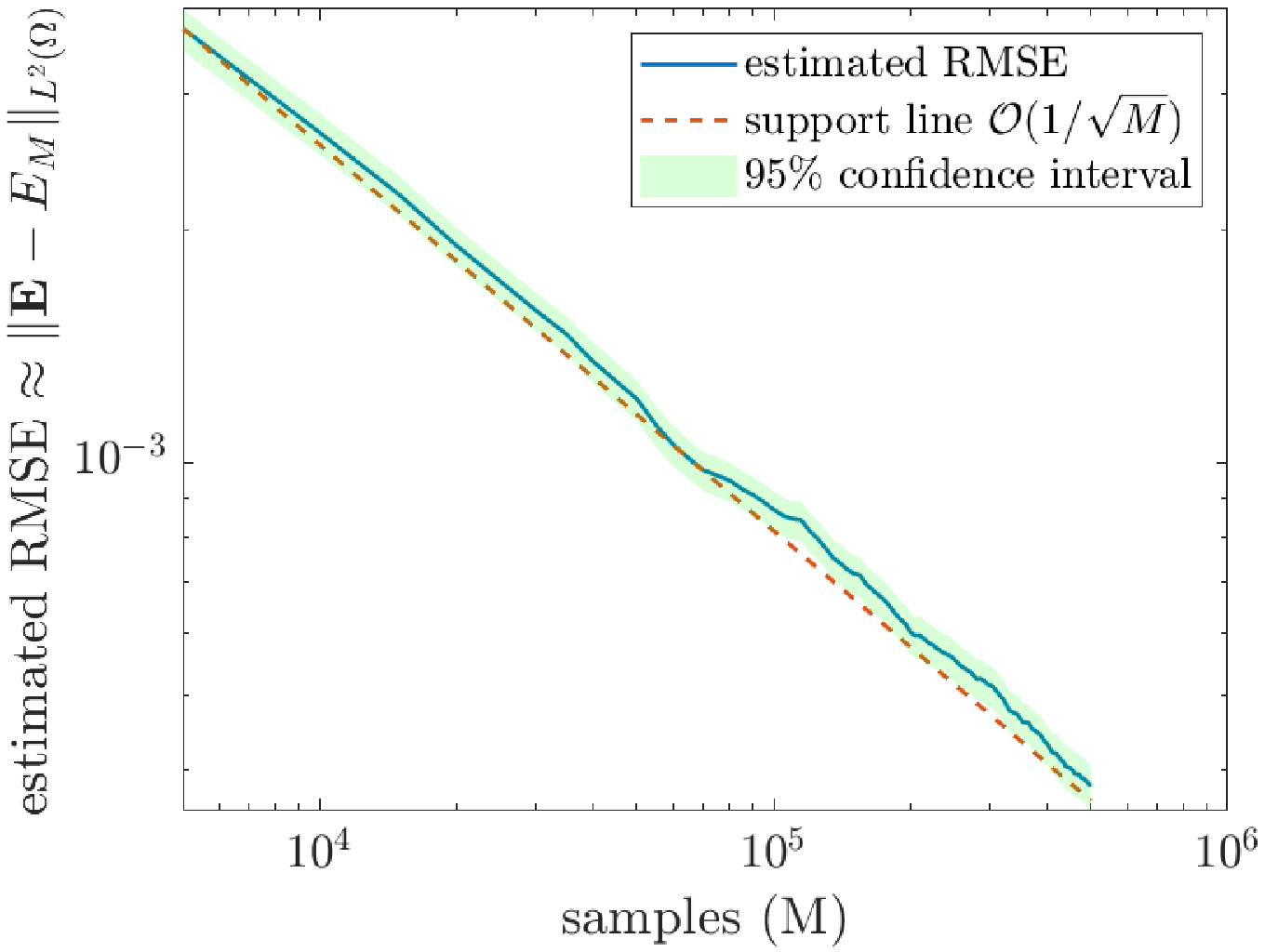}}
\caption{Convergence of the MC estimation of the covariance for different points and parameters: $q_L((1,1),(2,2))$ with $\underline{a}_G=(4,4),~\underline{b}_G=(12,12),~\sigma=1,~r=4$ (left), $q_L((0.5,1),(1.5,1))$ with $\underline{a}_G=(5,5),~\underline{b}_G=(10,10)~\sigma=1.5,~r=2$ (middle), $q_L((2,3),(6,4))$ with $\underline{a}_G=(1,2),~\underline{b}_G=(7,5)~\sigma=0.5,~r=0.5$ (right).}
\label{FIG:CovEstStatMCConv}
\end{figure}

\subsubsection{The non-stationary case}
In order to validate the covariance formula for a non-stationary underlying GRF, we choose the GRF to be a Brownian sheet on two dimensions, i.e. a centered GRF $W$ with $\mathbb{E}(W(x,y)W(x',y'))=\min(x,x')\cdot \min(y,y')$. Note that the Brownian sheet has continuous paths (see e.g. \cite[Theorem 1.4.4]{RandomFieldsAndGeometry}). As in Subsection \ref{SUBSUBSEC:NumExCovStat}, we assume the Lévy processes $l_j$ to be $Gamma(a_G^{(j)},b_G^{(j)})$ subordinators, for $j=1,2$, with different parameters $\underline{a}_G=(a_G^{(1)},a_G^{(2)})$ and $\underline{b}_G=(b_G^{(1)},b_G^{(2)})$. We consider the covariance of the field at different points: points with smaller and larger distance, points, which are equal in one coordinate and points, which are equal in both coordinates (pointwise variance). 



As in the first example, we use the 500 independent MC runs to estimate the RMSE (see Subsection \ref{SUBSUBSEC:NumExCovStat}). Figure \ref{FIG:CovEstNonStatMCConv} shows the expected convergence of order $\mathcal{O}(1/\sqrt{M})$ in all the considered situations (see Lemma \ref{LE:MCConv}).

\begin{figure}[ht]
	\centering
	\subfigure{\includegraphics[scale=0.24]{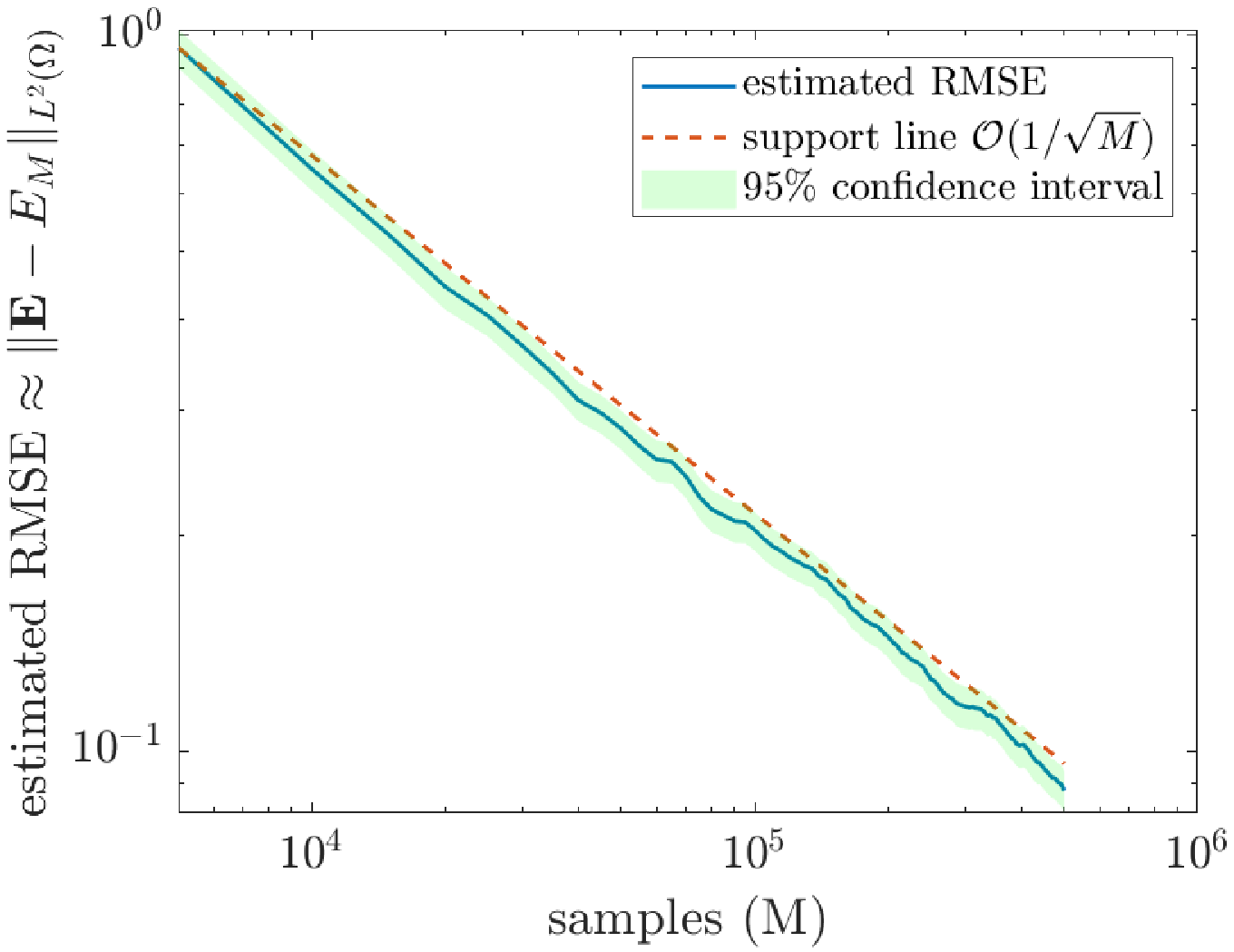}}
\subfigure{\includegraphics[scale=0.24]{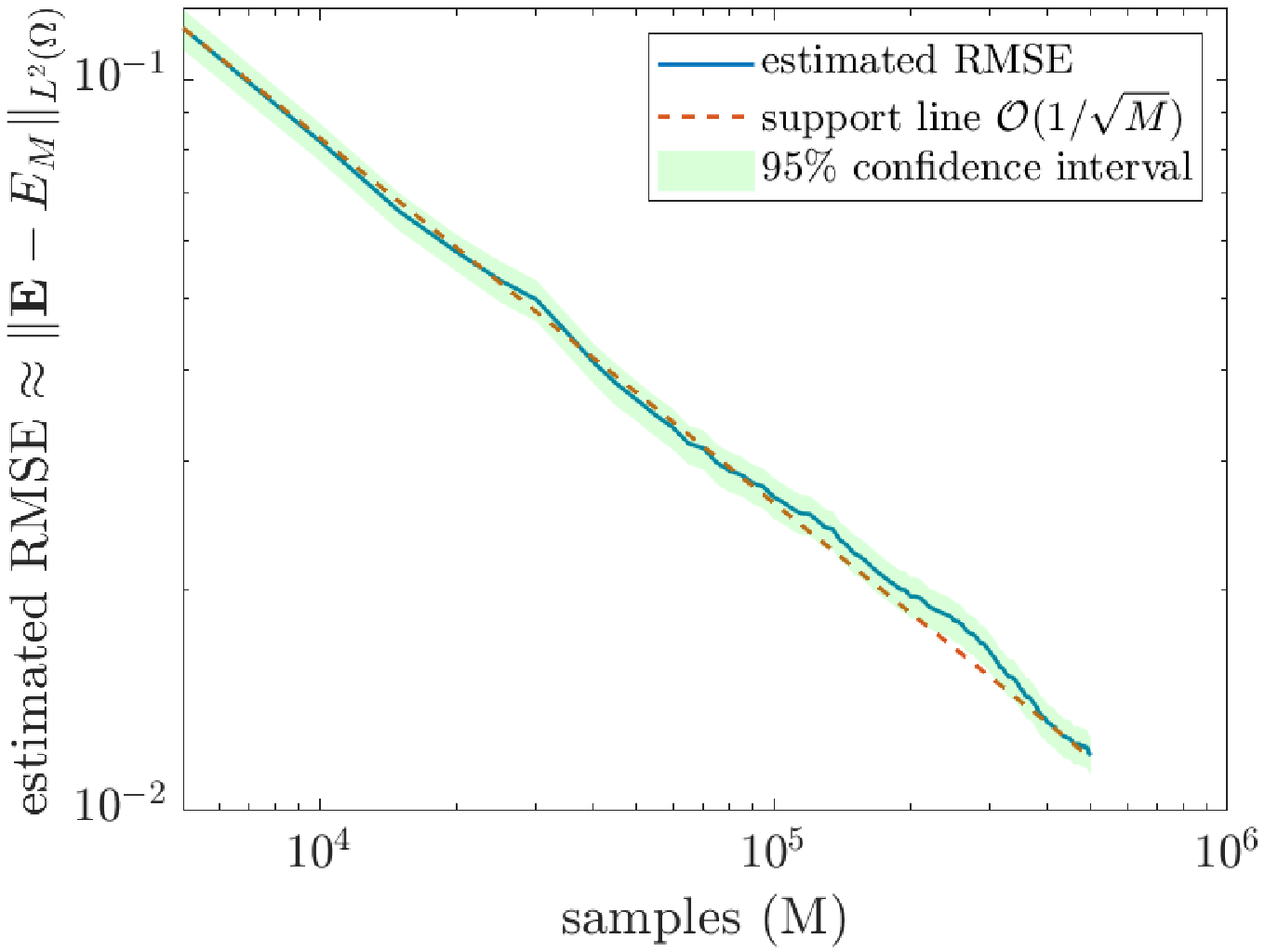}}
\subfigure{\includegraphics[scale=0.24]{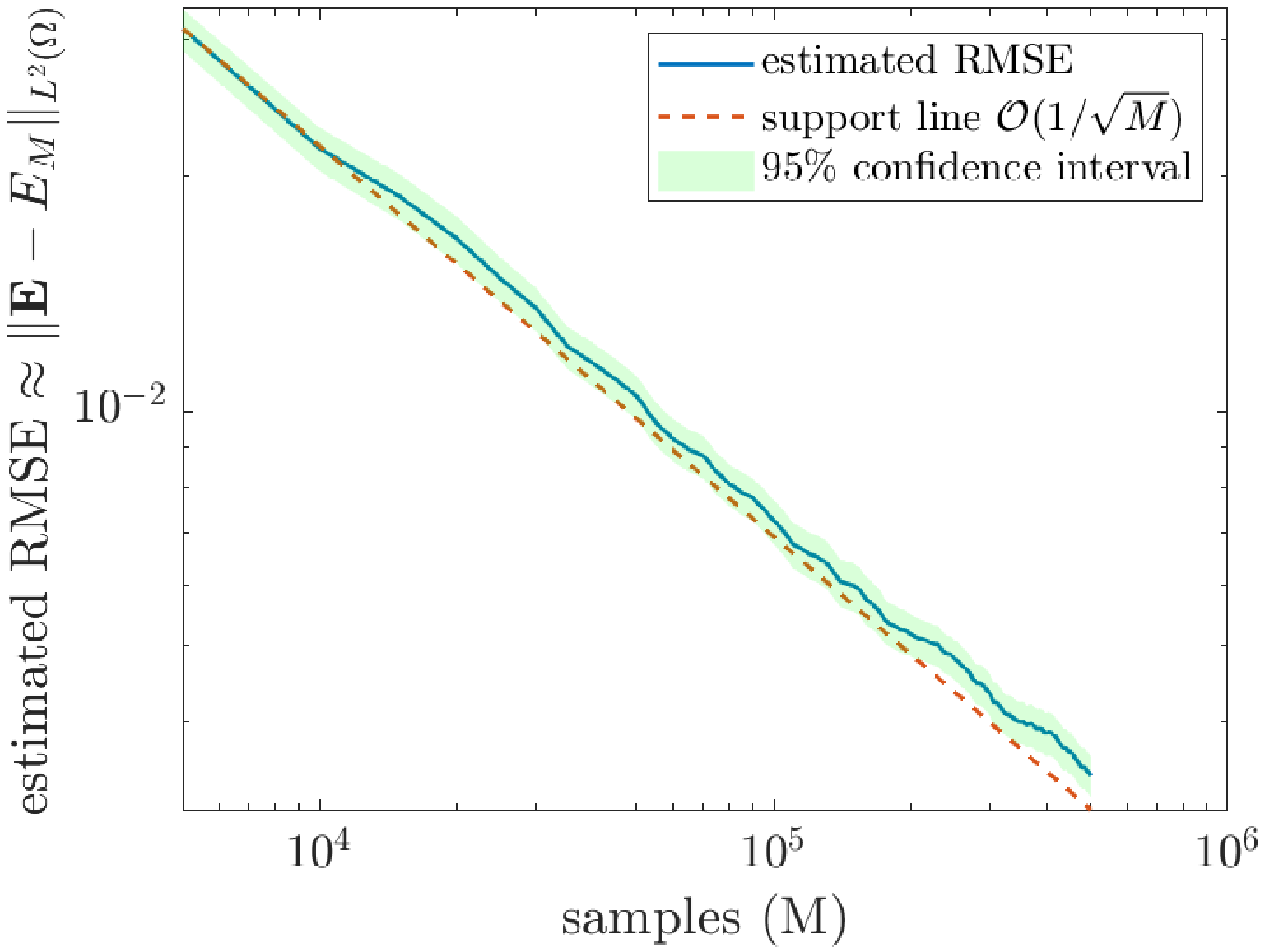}}
\subfigure{\includegraphics[scale=0.24]{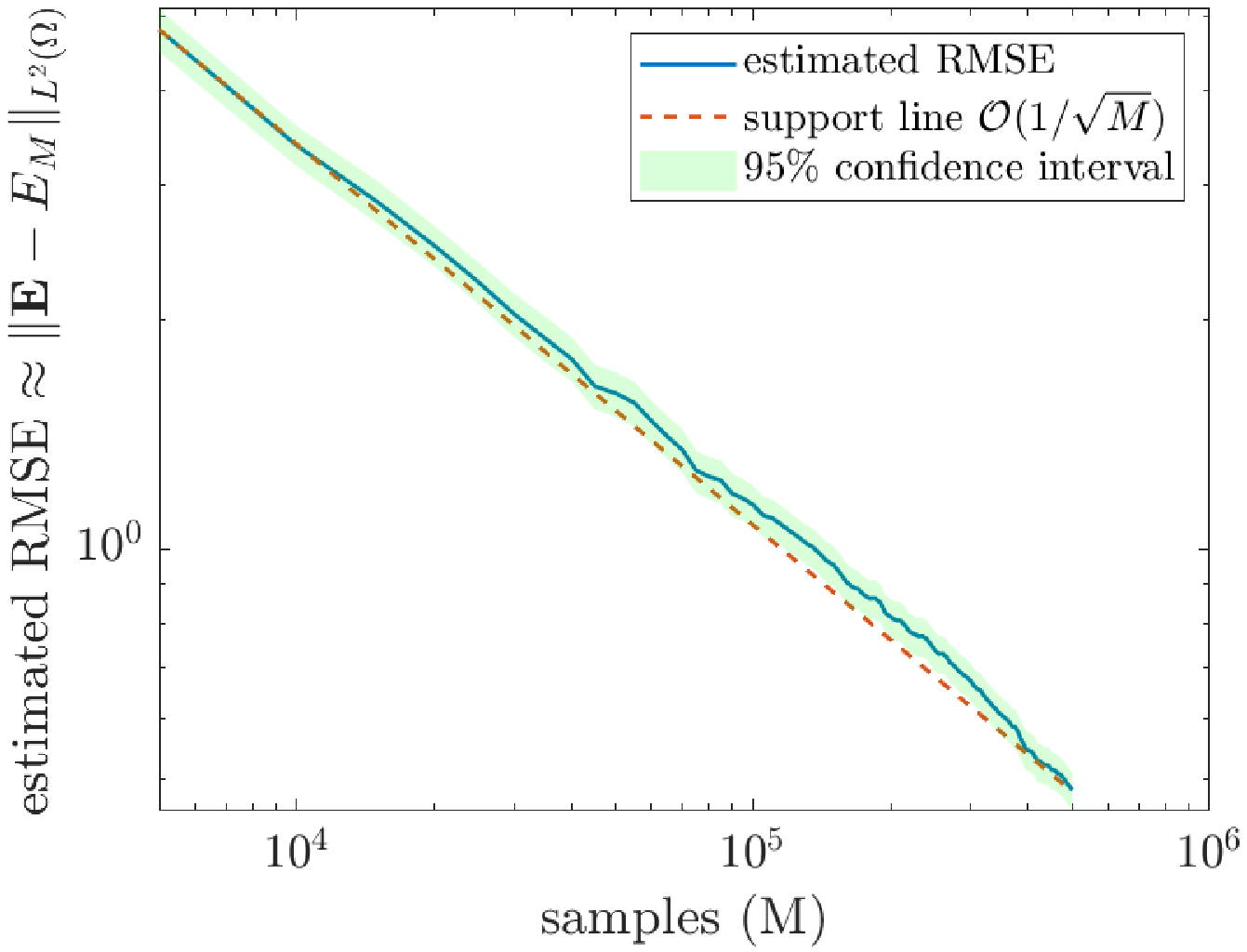}}
\caption{Convergence of the MC estimation of the covariance for different points and parameters: $q_L((1,1),(2,2))$ with $\underline{a}_G=(5,5),~\underline{b}_G=(1,1)$ (top left), $q_L((1,3),(4,3))$ with $\underline{a}_G=(2,2),~\underline{b}_G=(2,2)$ (top right), $q_L((0.9,2),(4,3))$ with $\underline{a}_G=(3,4),~\underline{b}_G=(6,5)$ (bottom left),  $q_L((3,3),(3,3))$ with $\underline{a}_G=(5,5),~\underline{b}_G=(1,1)$ (bottom right).}
\label{FIG:CovEstNonStatMCConv}
\end{figure}

\subsection{Stochastic fitting: pointwise distribution}\label{SUBSEC:NumExStochFit}

In this section, we present an approach to fit the parameters of a subordinated GRF in order to obtain a specific pointwise distribution. We work under Assumption \ref{ASS:GRFCharFct} and assume that a subordinated GRF with an unknown set of parameters $(\gamma_1,\gamma_2,\sigma, \nu_1,\nu_2)$ is given (see Remark \ref{REM:CharacterizationPropertyLevyKhinchin}). We further assume that the distribution of the random field at $n\in \mathbb{N}$ fixed spatial points is known and we want to estimate parameters of the field in order to obtain a fitted subordinated GRF which admits approximately the same pointwise distribution at the given $n$ points.

In order to fit the parameters of the subordinated GRF we follow two different approaches: fitting based on pointwise densities and fitting based on pointwise characteristic functions. The procedure is as follows: Theorem \ref{TH:LevyKhinchinFormula} allows access to the pointwise characteristic functions of a subordinated GRF with a given set of parameters. The pointwise characteristic functions in turn allow to calculate the pointwise density functions  (see Subsection \ref{SUBSUBSEC:NumExPWDistFI} and \cite{NoteOnTheInversionTheorem}). Hence, Theorem \ref{TH:LevyKhinchinFormula} gives access to the pointwise characteristic functions \textit{and} the pointwise density functions of a subordinated GRF. Having either the pointwise densities or the pointwise characteristic functions, a fitting-error can be computed by comparing those with the desired density functions (resp. characteristic functions) at the specific points. Therefore, starting with an initial set of parameters (\textit{initial guess}), we compute the $\mathcal{L}^2$-fitting-error of the densities (resp. characteristic functions) for all $n$ points and optimize the parameters in order to minimize the fitting-error using the Levenberg-Marquardt (LM) algorithm (see for example \cite{MoreLMAlg}). We will call these two approaches \textit{density-approach} and \textit{chararecteristic-function-approach} to fit the parameters of the subordinated GRF. Figure \ref{FIG:NumExTikxComp} illustrates the workflow of the two approaches.

\begin{figure}[ht]
	\centering
	\subfigure{\includegraphics[scale=0.7]{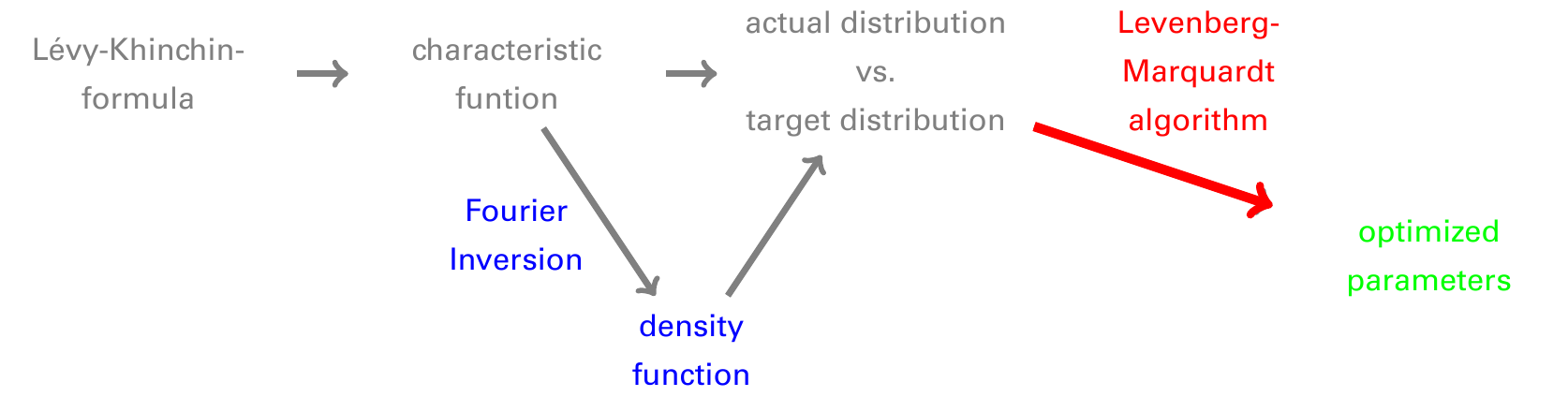}}
\caption{Schematic illustration of the density-approach (blue) and the characteristic-function-approach for parameter optimization of a subordinated GRF.}
\label{FIG:NumExTikxComp}
\end{figure}

In our experiments we consider the domain $[0,\mathbb{T}]_2=[0,1]^2$ and the four points $P_1=(0.1,0.1),$ $P_2=(0.1,0.8),~P_3=(0.7,0.2)$ and $P_4=(1,1)$ (see Figure \ref{FIG:PointsFitting}).
\begin{figure}[ht]
	\centering
	\subfigure{\includegraphics[scale=0.4]{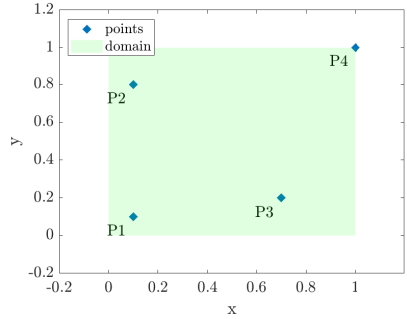}}
	\caption{Visualization of the spatial domain and the considered points for the stochastic fitting problem.}
	\label{FIG:PointsFitting}
	\end{figure}
Further, $l_j$ is a $Gamma(\tilde{a}_G^{(j)},\tilde{b}_G^{(j)})$-subordinator for $j=1,2$, where we vary the parameters $\underline{\tilde{a}}_G=(\tilde{a}_G^{(1)},\tilde{a}_G^{(2)})$ and $\underline{\tilde{b}}_G=(\tilde{b}_G^{(1)},\tilde{b}_G^{(2)})$. The GRF $W$ is defined by $W(x,y)=\sqrt{x+y}\,\tilde{W}(x,y)$, for $(x,y)\in [0,1]^2$, where $\tilde{W}$ is a Matérn-1.5-GRF with varying pointwise standard deviation $\tilde{\sigma}>0$ (see Remark \ref{REM:StatFieldsExtensionLKFormula} and Subsection \ref{SUBSUBSEC:CPA}).

\subsubsection{Density fitting}\label{SUBSUBSEC:NumExFittingDensity}
In this section we follow the density-approach to fit the pointwise distribution of a subordinated GRF at the points $P_1,~P_2,~P_3$ and $P_4$ described by Figure \ref{FIG:PointsFitting}.

We consider a Gamma-subordinated GRF with parameters $\tilde{a}_G^{(1)}=3,~\tilde{b}_G^{(1)}=10,~\tilde{a}_G^{(2)}=3,~\tilde{b}_G^{(2)}=10,~\tilde{\sigma}=2$. 
In the following, we perform two experiments, varying the initial guess for the parameter optimization. Figure \ref{FIG:NumExFittingDensity1} shows the results after $5$ iterations of the LM-algorithm using the following initial guess parameters: $a_G^{(1)}=2,~b_G^{(1)}=12,~a_G^{(2)}=4,~b_G^{(2)}=9,~\sigma=1$.

\begin{figure}[ht]
	\centering
	\subfigure{\includegraphics[scale=0.24]{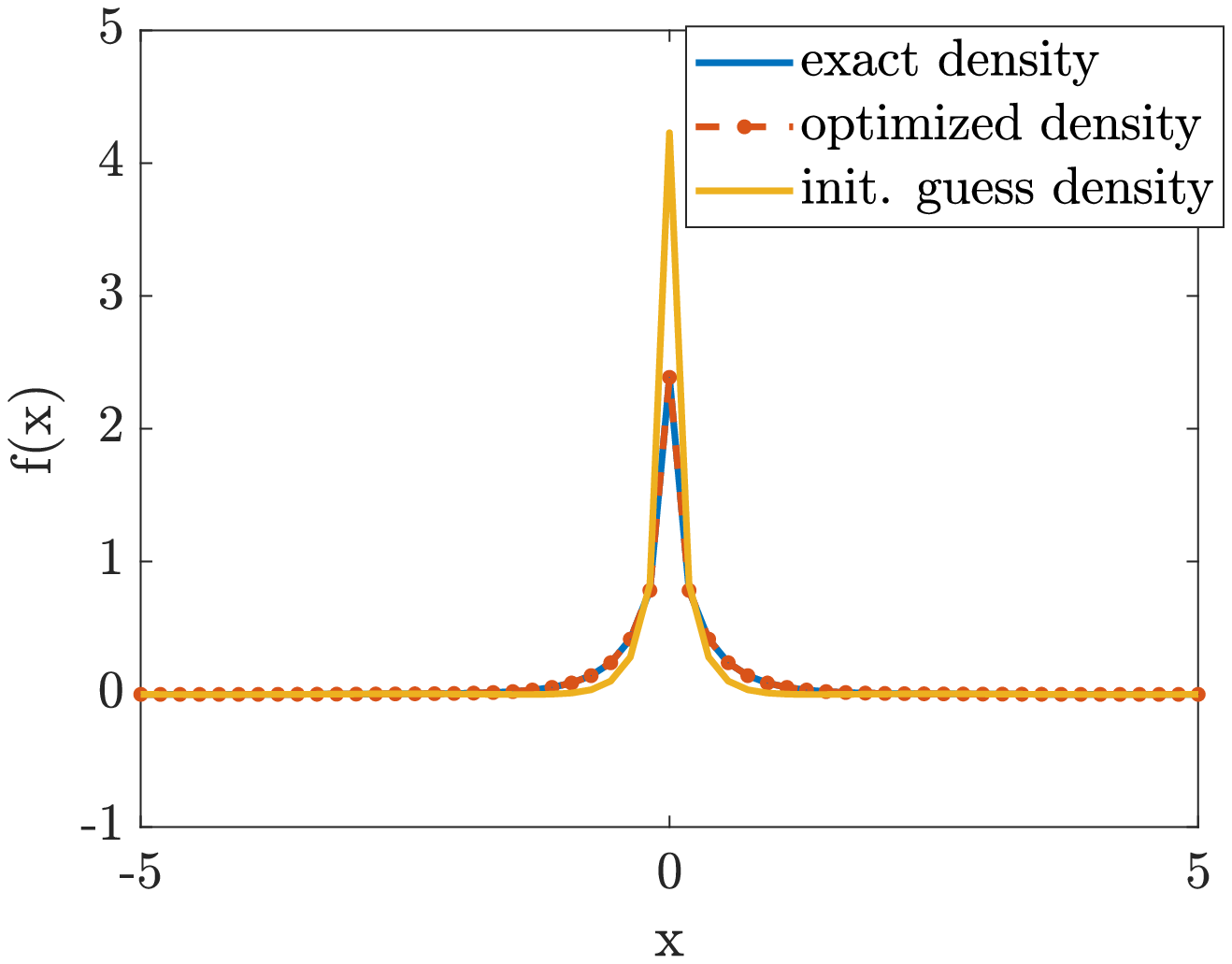}}
\subfigure{\includegraphics[scale=0.24]{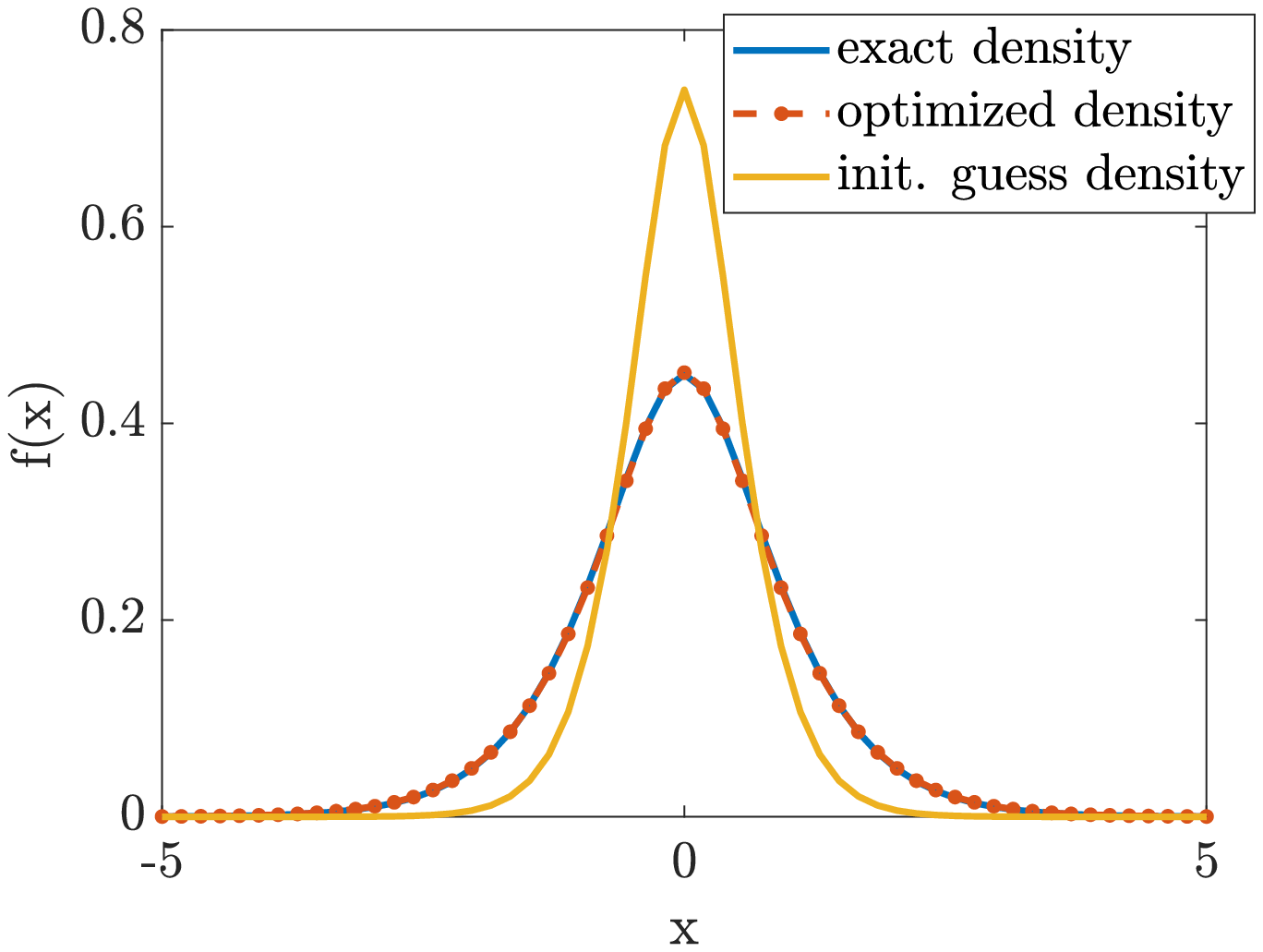}}
\subfigure{\includegraphics[scale=0.24]{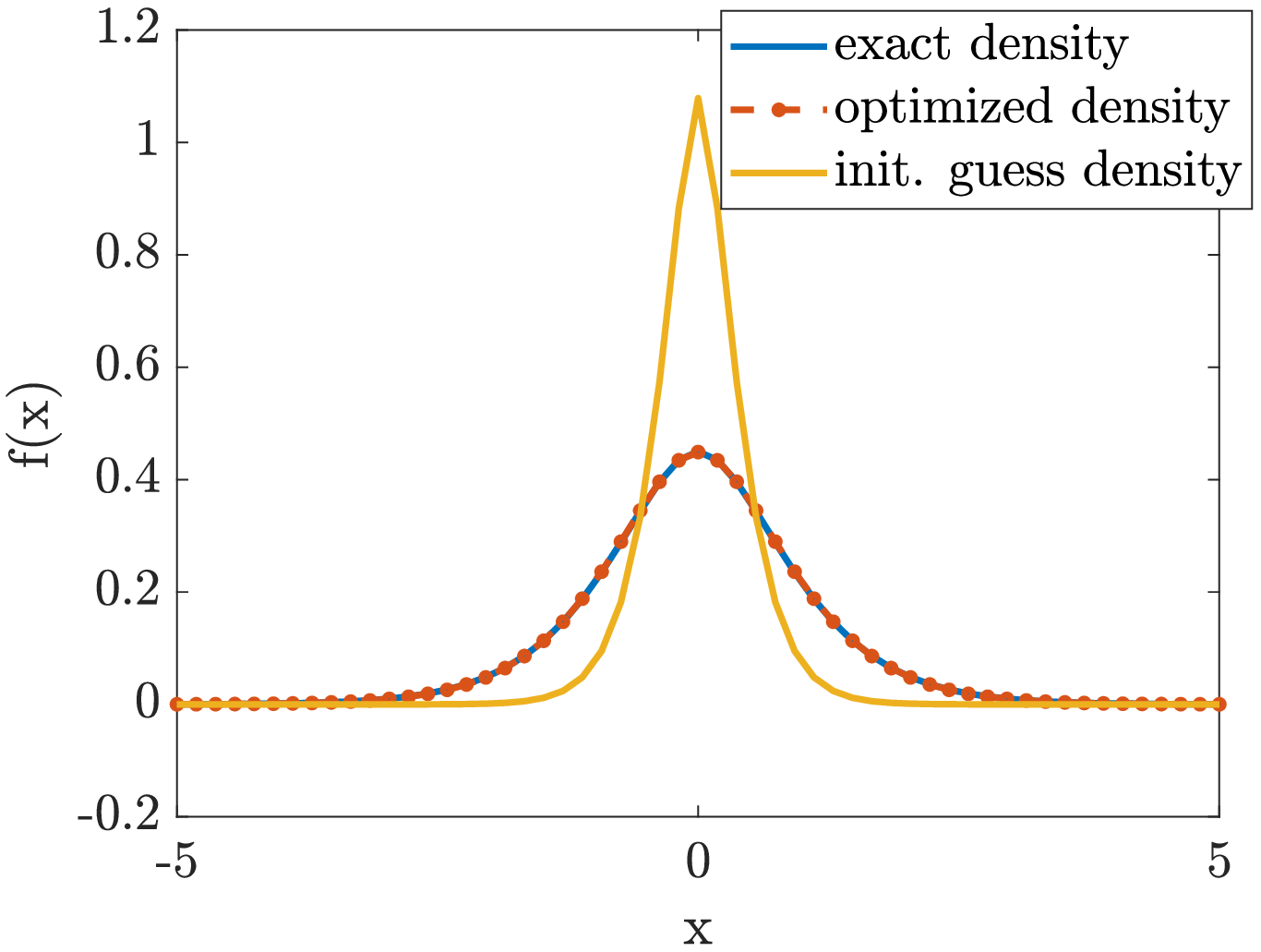}}
\subfigure{\includegraphics[scale=0.24]{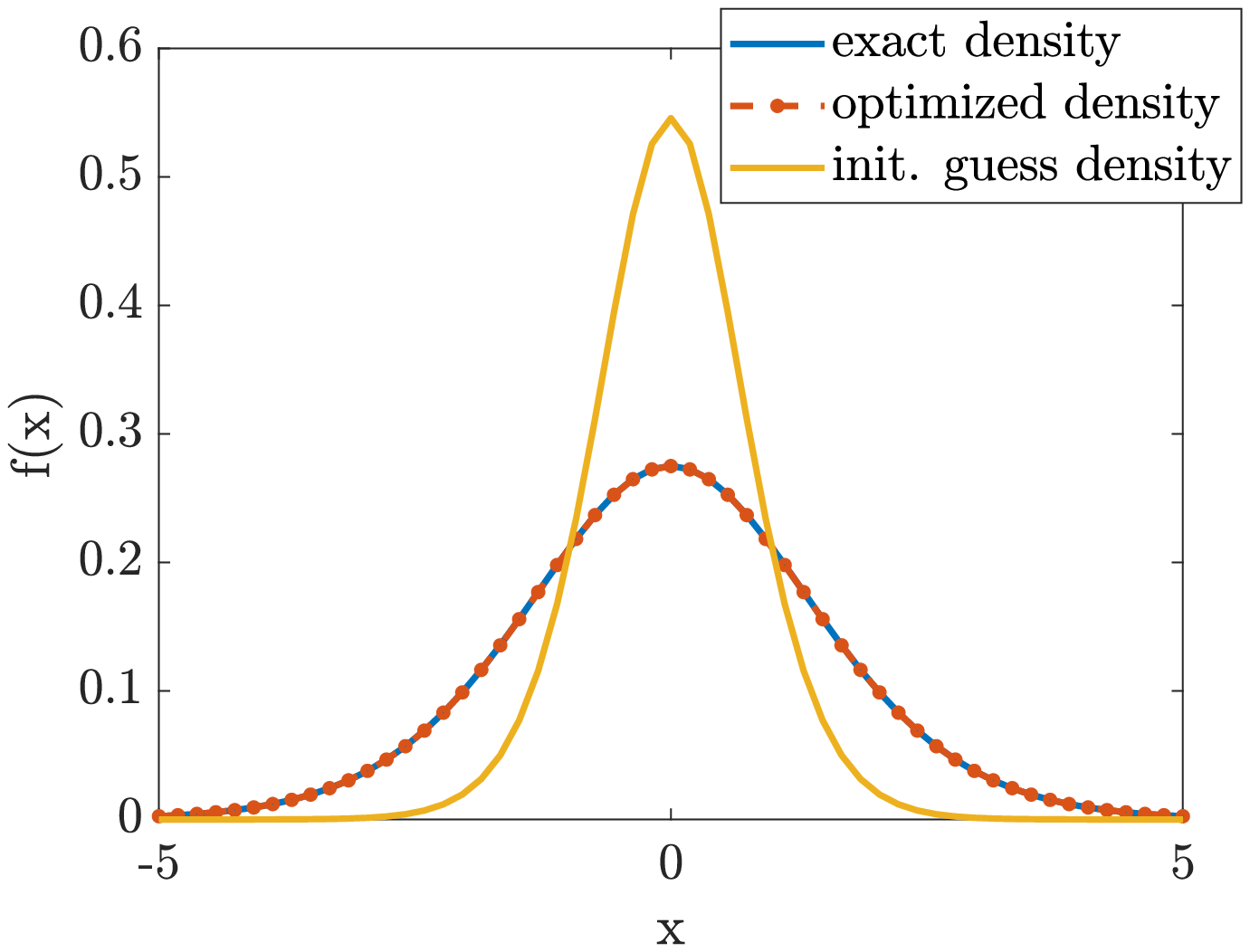}}
\caption{Density functions corresponding to the initial parameters $a_G^{(1)}=2,~b_G^{(1)}=12,~a_G^{(2)}=4,~b_G^{(2)}=9,~\sigma=1$, the optimized parameters and the exact parameters of the Gamma-subordinated GRF at different points: $P_1=(0.1,0.1)$ (left), $P_2=(0.1,0.8)$ (second from left), $P_3=(0.7,0.2)$ (second from right), $P_4=(1,1)$ (right).}
\label{FIG:NumExFittingDensity1}
\end{figure}
As we see in Figure \ref{FIG:NumExFittingDensity1}, the density functions at the 4 considered points match the desired densities corresponding to the true parameter set. Since this might also be caused by the fact that the initial guess parameters are quite close to the true parameters, in the second experiment we choose the initial guess far away from the true values.  
%
Figure \ref{FIG:NumExFittingDensity3} shows that the set of initial guess parameters leading to a good fit of the densities corresponding to the optimized parameters is limited: if we consider the initial parameters  $a_G^{(1)}=1,~b_G^{(1)}=24,~a_G^{(2)}=28,~b_G^{(2)}=1,~\sigma=1$, we see that densities do not match the desired ones anymore. This is also true for a higher number of iterations of the LM-algorithm.

\begin{figure}[ht]
	\centering
	\subfigure{\includegraphics[scale=0.24]{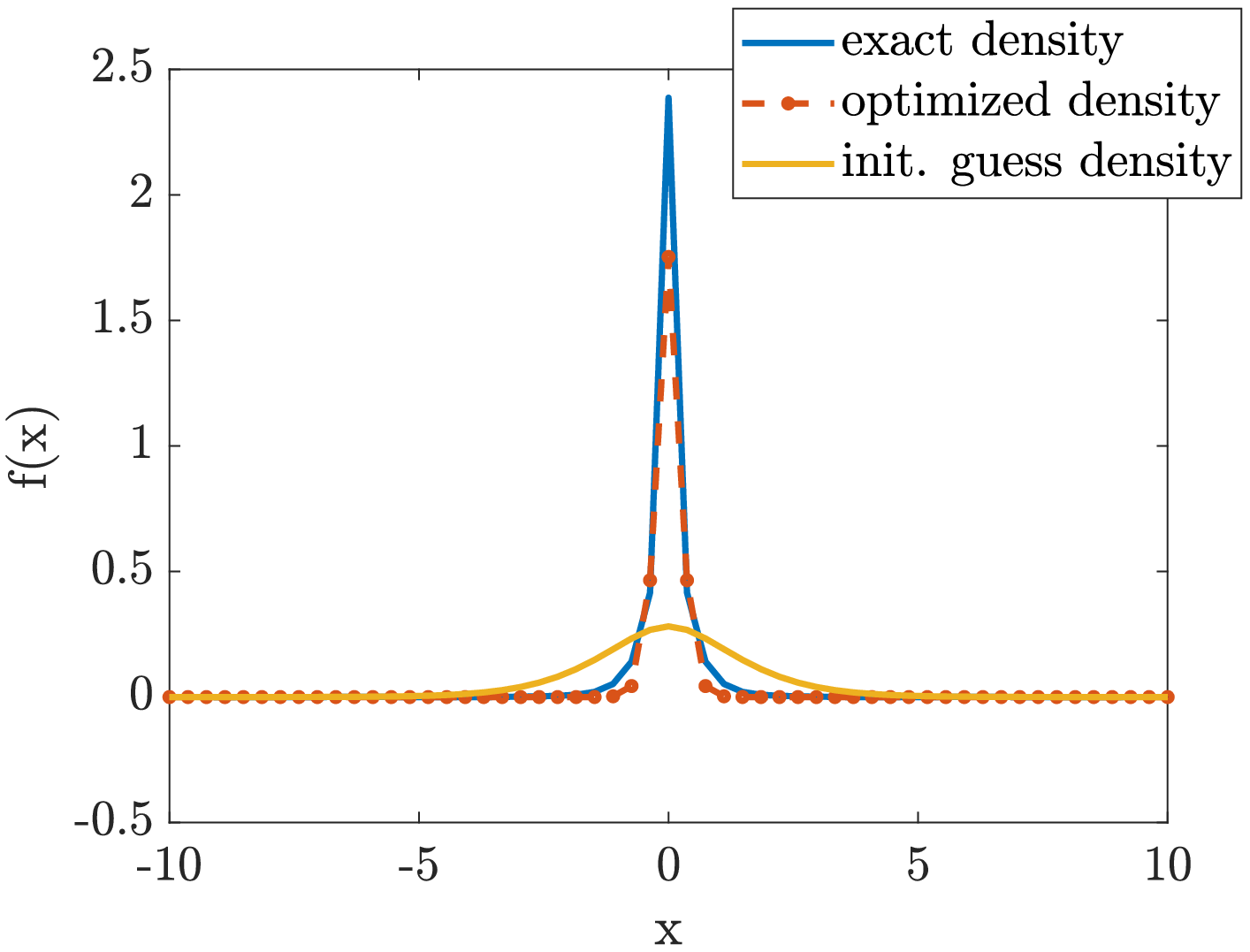}}
\subfigure{\includegraphics[scale=0.24]{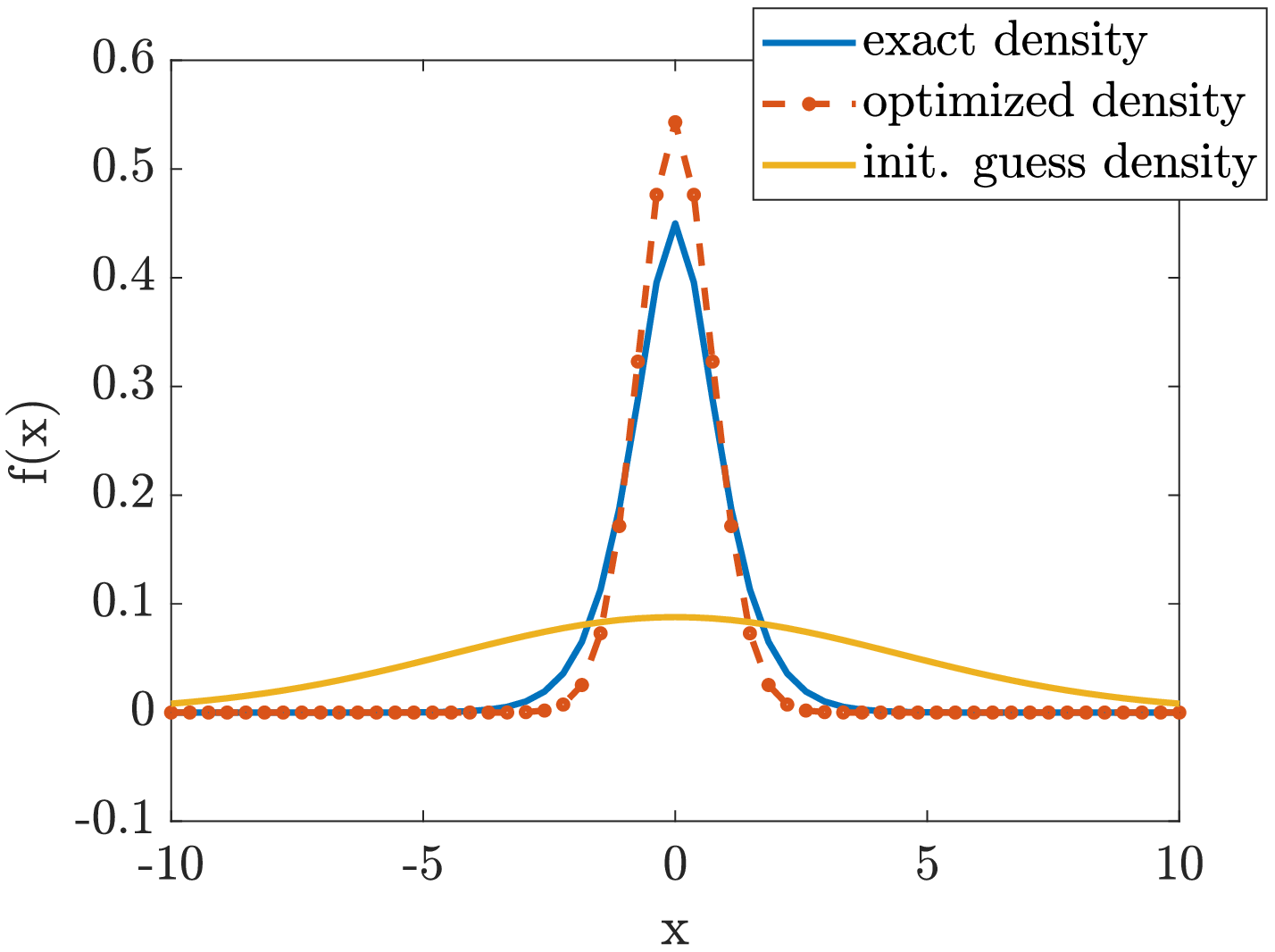}}
\subfigure{\includegraphics[scale=0.24]{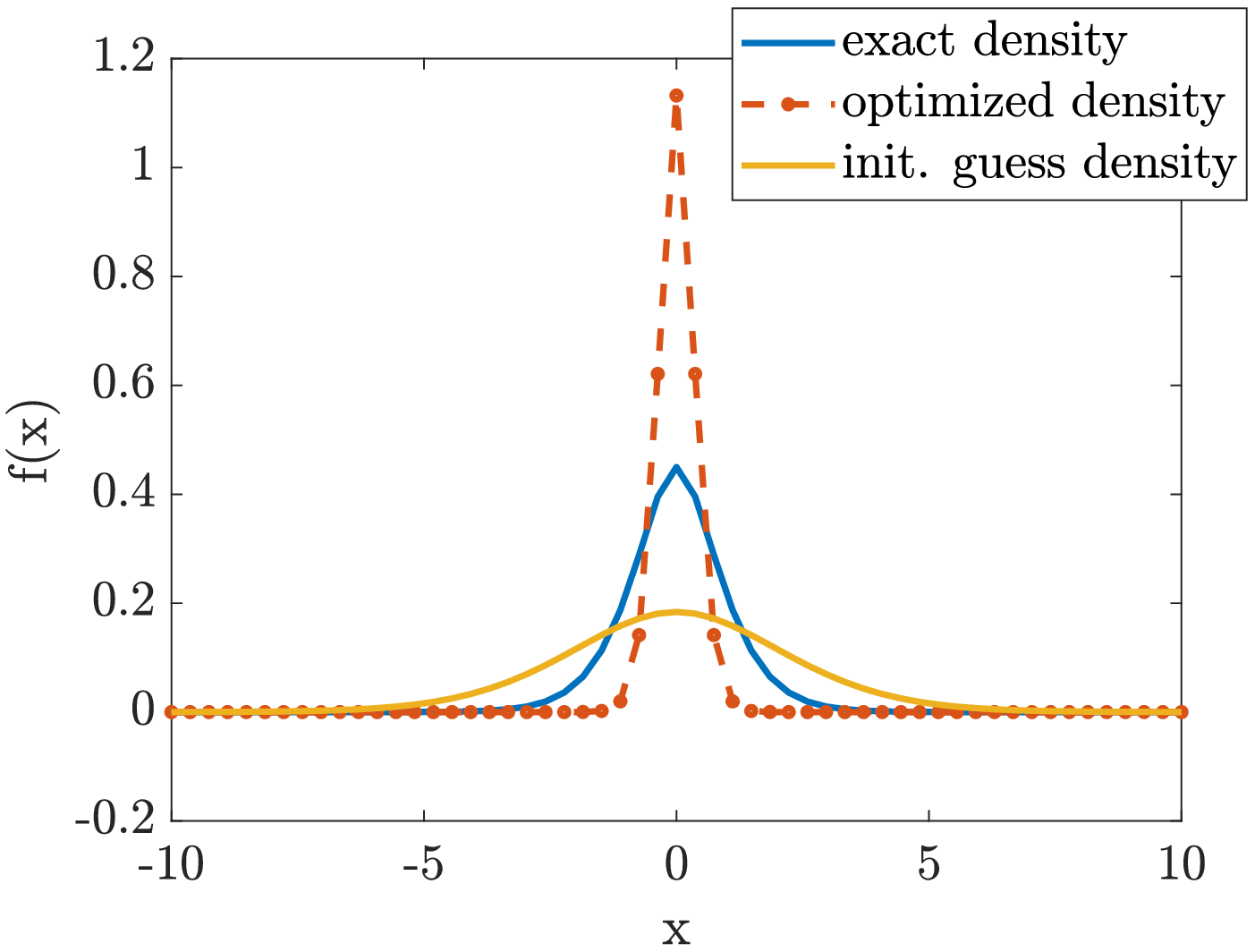}}
\subfigure{\includegraphics[scale=0.24]{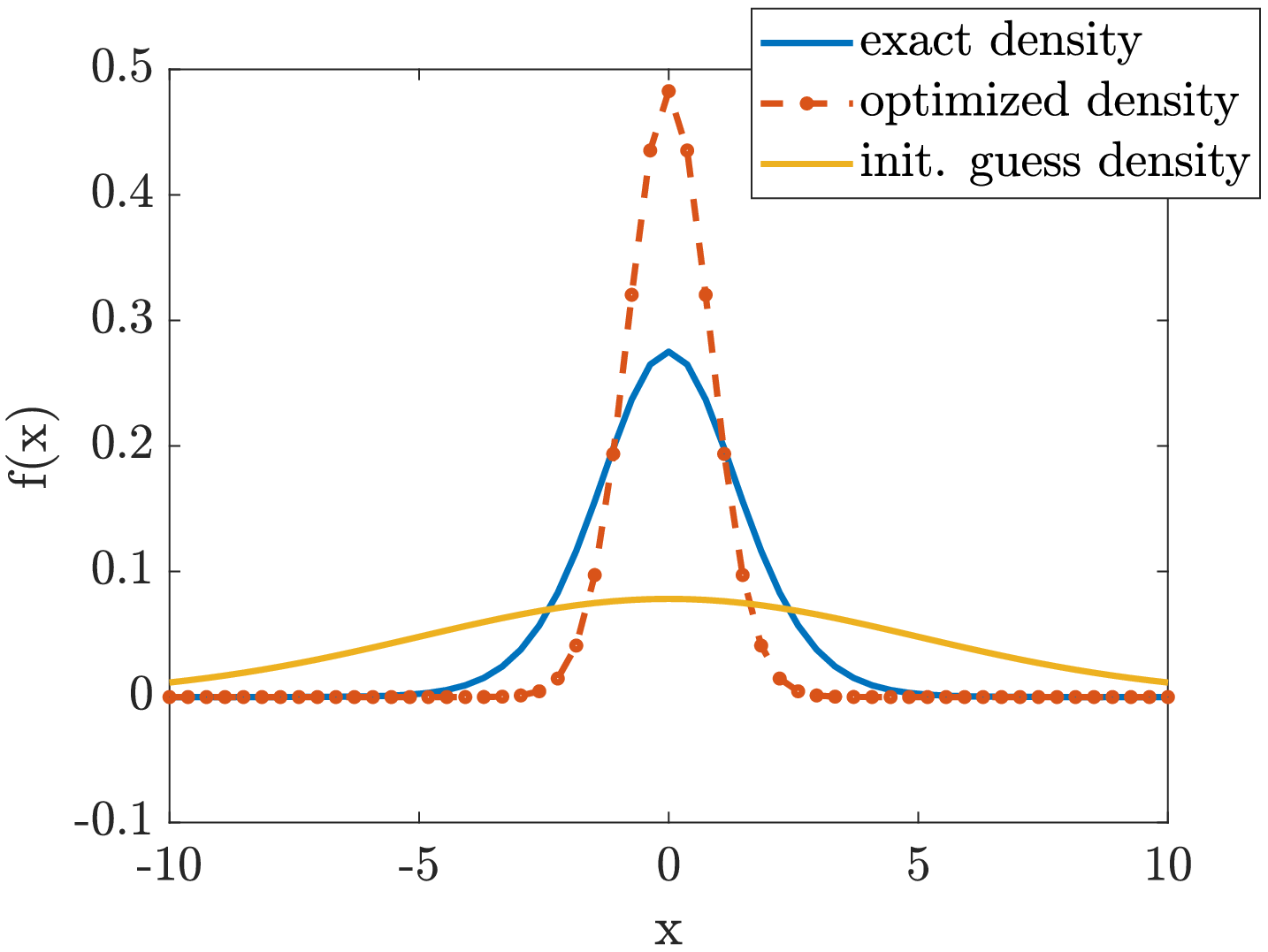}}
\caption{Density functions corresponding to the initial parameters $a_G^{(1)}=1,~b_G^{(1)}=24,~a_G^{(2)}=28,~b_G^{(2)}=1,~\sigma=1$, the optimized parameters and the exact parameters of the Gamma-subordinated GRF at different points: $P_1=(0.1,0.1)$ (left), $P_2=(0.1,0.8)$ (second from left), $P_3=(0.7,0.2)$ (second from right), $P_4=(1,1)$ (right).}
\label{FIG:NumExFittingDensity3}
\end{figure}

The results show that the density-approach brings two problems: the first one is that we need evaluations of the approximated pointwise density functions which is very expensive as iterated integrals that have to be computed. This is also the reason why we only perform $5$ iterations of the LM-algorithm. Furthermore, the method is quite sensitive to the parameters for the approximation of the integrals which have to be computed for the error. This is also a reason why the method fails if we choose an initial guess parameter set far away from the true parameters. 
\subsubsection{Characteristic function fitting}

In the next examples we use pointwise characteristic functions to optimize the parameters in order to obtain the desired pointwise distributions.

We again consider a Gamma-subordinated GRF with $\tilde{a}_G^{(1)}=3,~\tilde{b}_G^{(1)}=10,~\tilde{a}_G^{(2)}=3,~\tilde{b}_G^{(2)}=10,~\tilde{\sigma}=2$ and we perform $3$ experiments, varying the initial guess. Note that for the characteristic-function-approach, we always compare the real part of the characteristic functions since the imaginary parts are numerically zero in all considered examples. Figure \ref{FIG:NumExFittingChar1} shows the results after $50$ iterations of the LM-algorithm using the initial guess parameters $a_G^{(1)}=2,~b_G^{(1)}=12,~a_G^{(2)}=4,~b_G^{(2)}=9,~\sigma=1$. 
\begin{figure}[ht]
	\centering
	\subfigure{\includegraphics[scale=0.23]{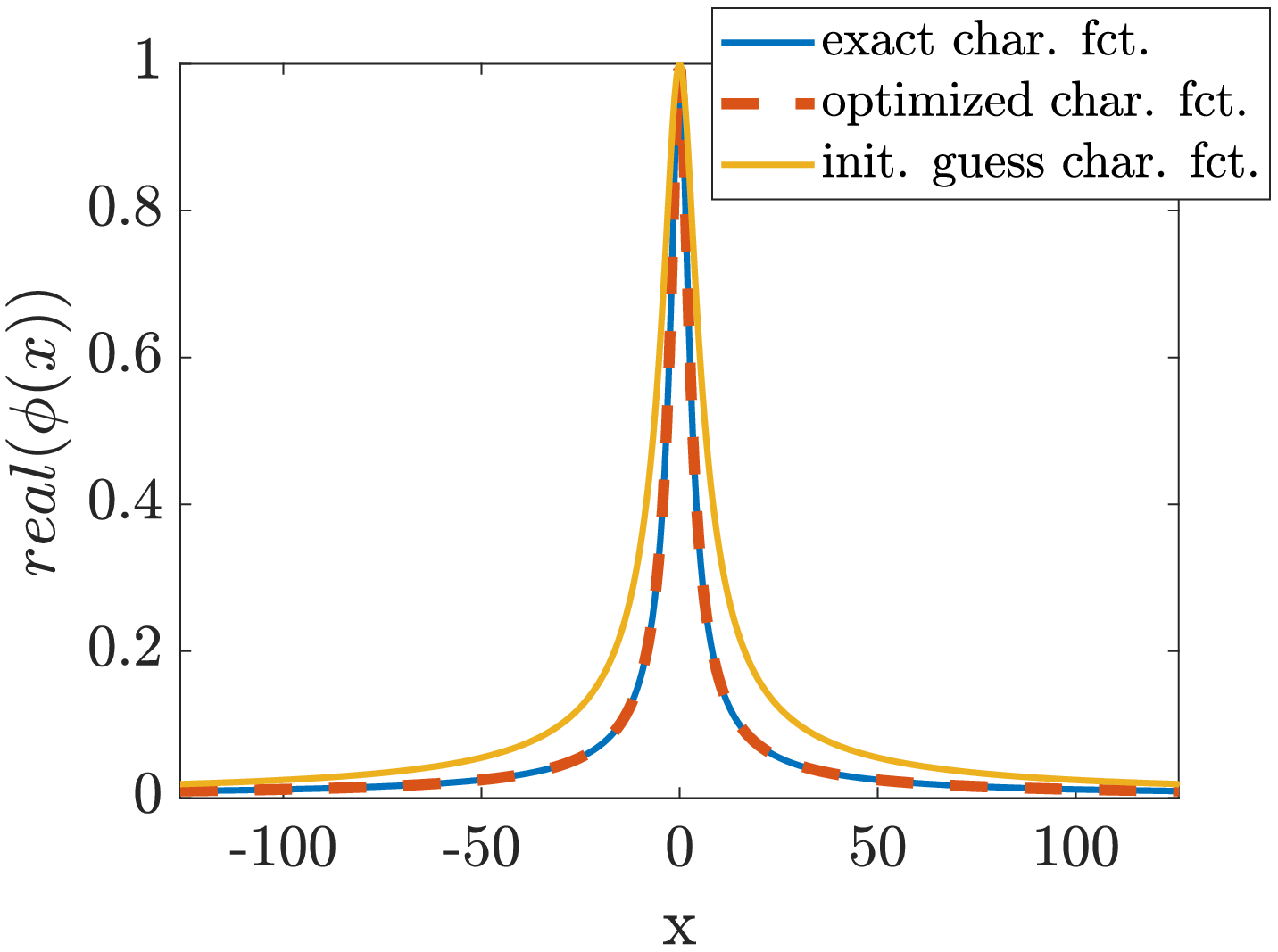}}
\subfigure{\includegraphics[scale=0.23]{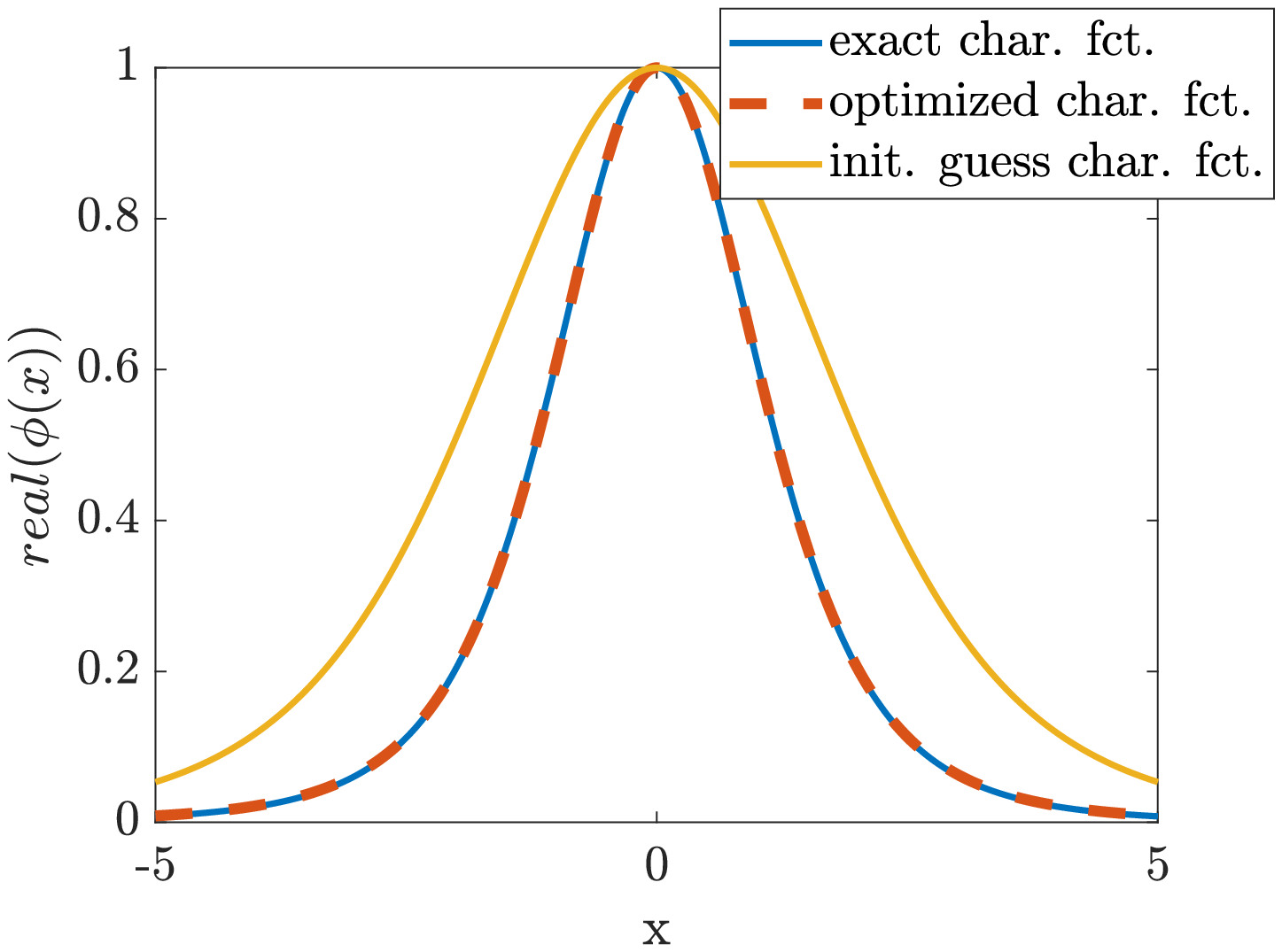}}
\subfigure{\includegraphics[scale=0.23]{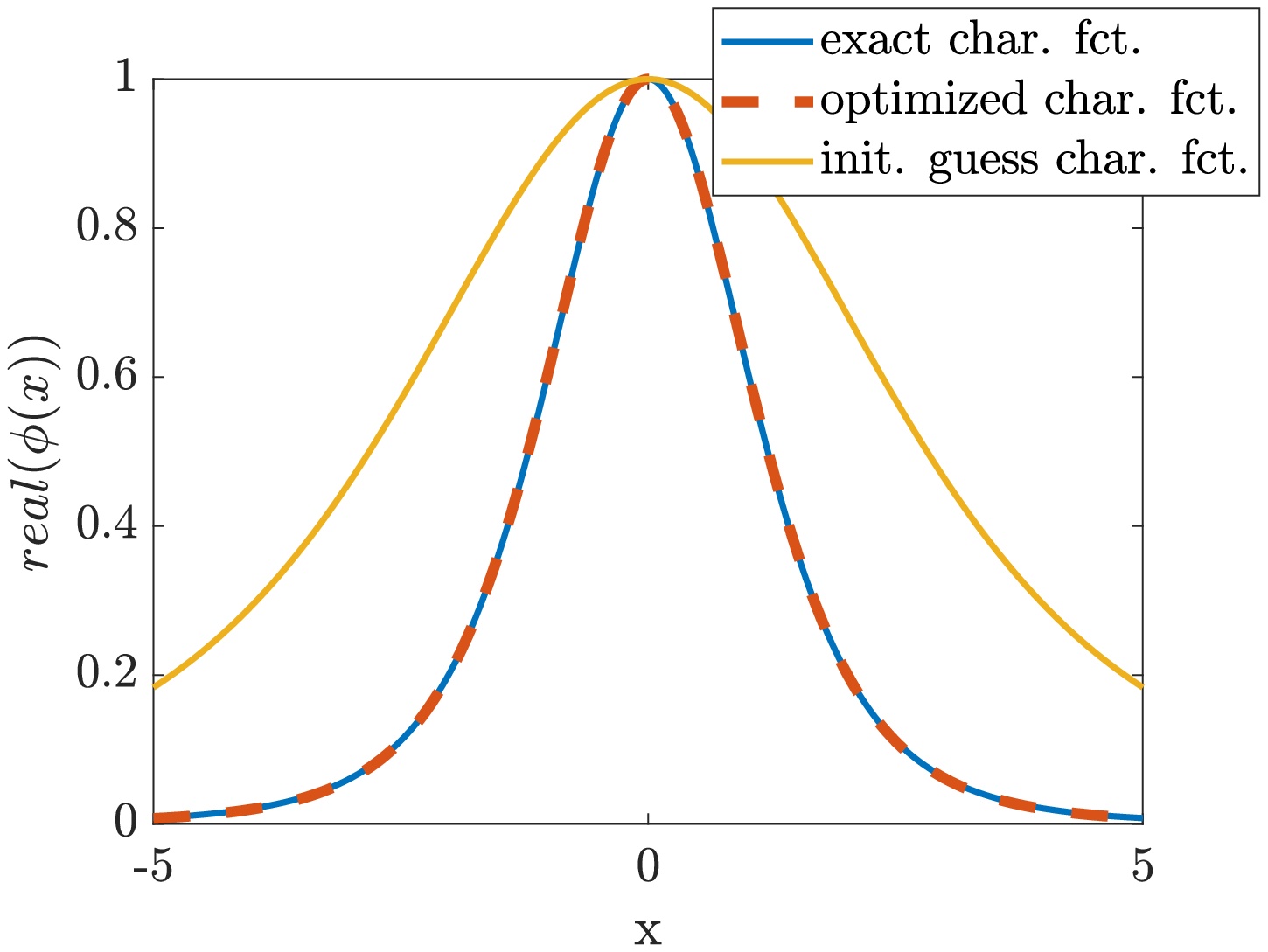}}
\subfigure{\includegraphics[scale=0.23]{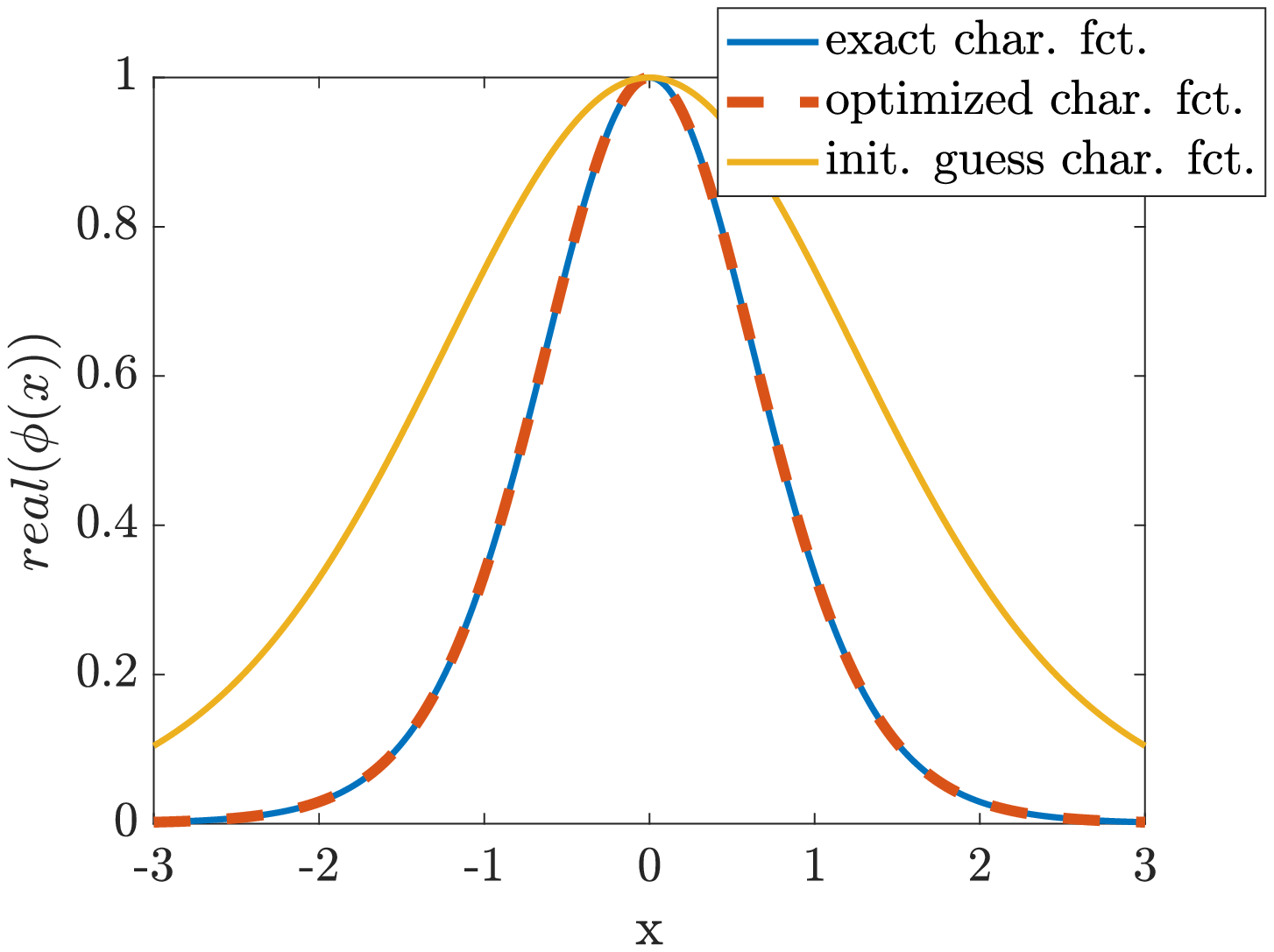}}
\caption{Real part of the characteristic functions corresponding to the initial parameters $a_G^{(1)}=2,~b_G^{(1)}=12,~a_G^{(2)}=4,~b_G^{(2)}=9,~\sigma=1$, the optimized parameters and the exact parameters of the Gamma-subordinated GRF at different points: $P_1=(0.1,0.1)$ (left), $P_2=(0.1,0.8)$ (second from left), $P_3=(0.7,0.2)$ (second from right), $P_4=(1,1)$ (right).}
\label{FIG:NumExFittingChar1}
\end{figure}
We see that the pointwise characteristic functions corresponding to the optimized parameters match the desired characteristic functions in this example. 


As described in Subsection \ref{SUBSUBSEC:NumExFittingDensity}, moving the initial parameters further away uncovers the limitations of the density-approach. Therefore, the results for the initial parameters given by $a_G^{(1)}=1,~b_G^{(1)}=24,~a_G^{(2)}=28,~b_G^{(2)}=1,~\sigma=1$ (Figure \ref{FIG:NumExFittingChar3}) are important for the comparison of the two approaches. We see that the characteristic-function-approach still works for this set of initial parameters: the characteristic functions of the subordinated GRF corresponding to the optimized parameters at the four specified point match the desired distributions.

\begin{figure}[ht]
	\centering
	\subfigure{\includegraphics[scale=0.23]{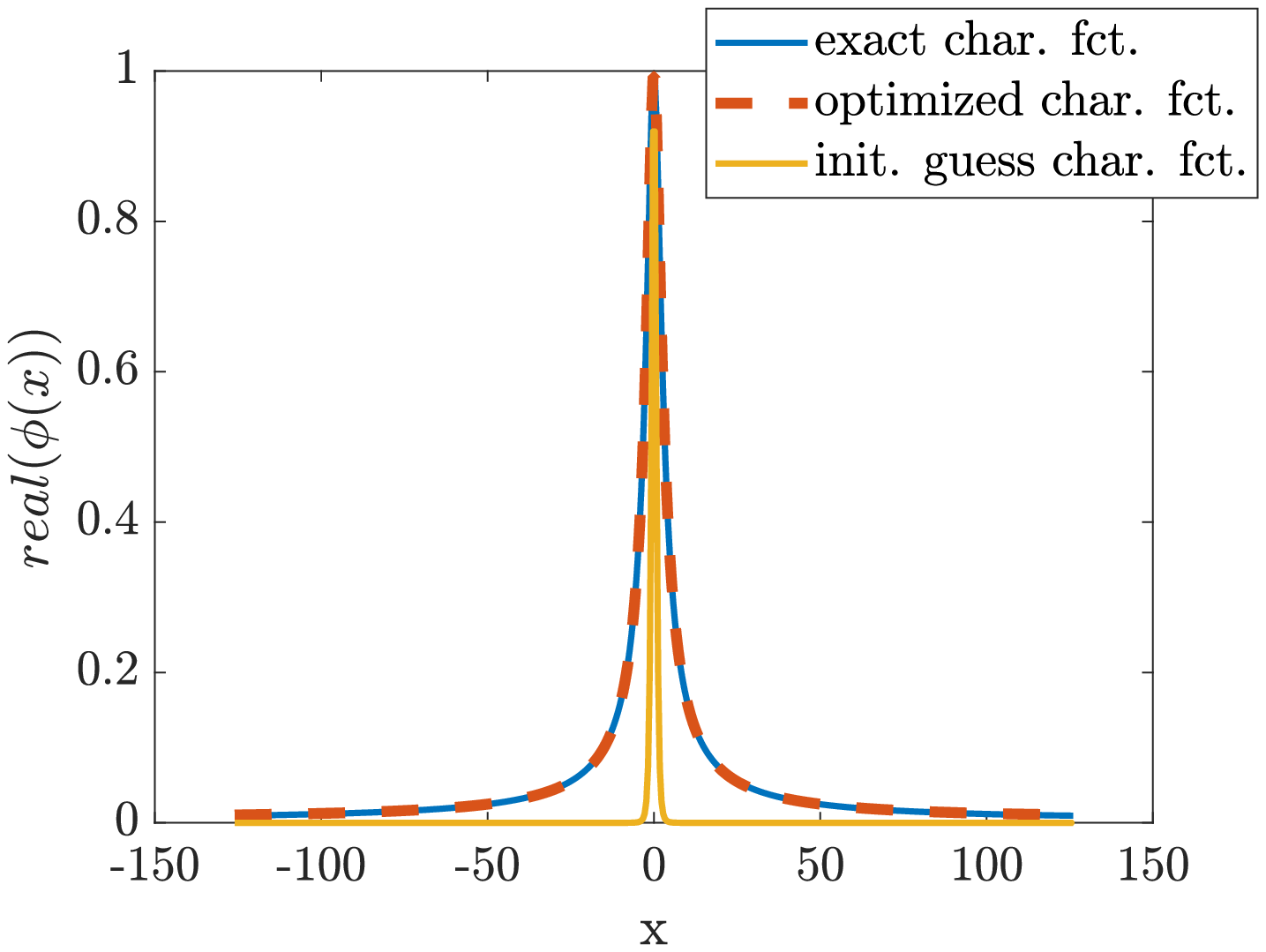}}
\subfigure{\includegraphics[scale=0.23]{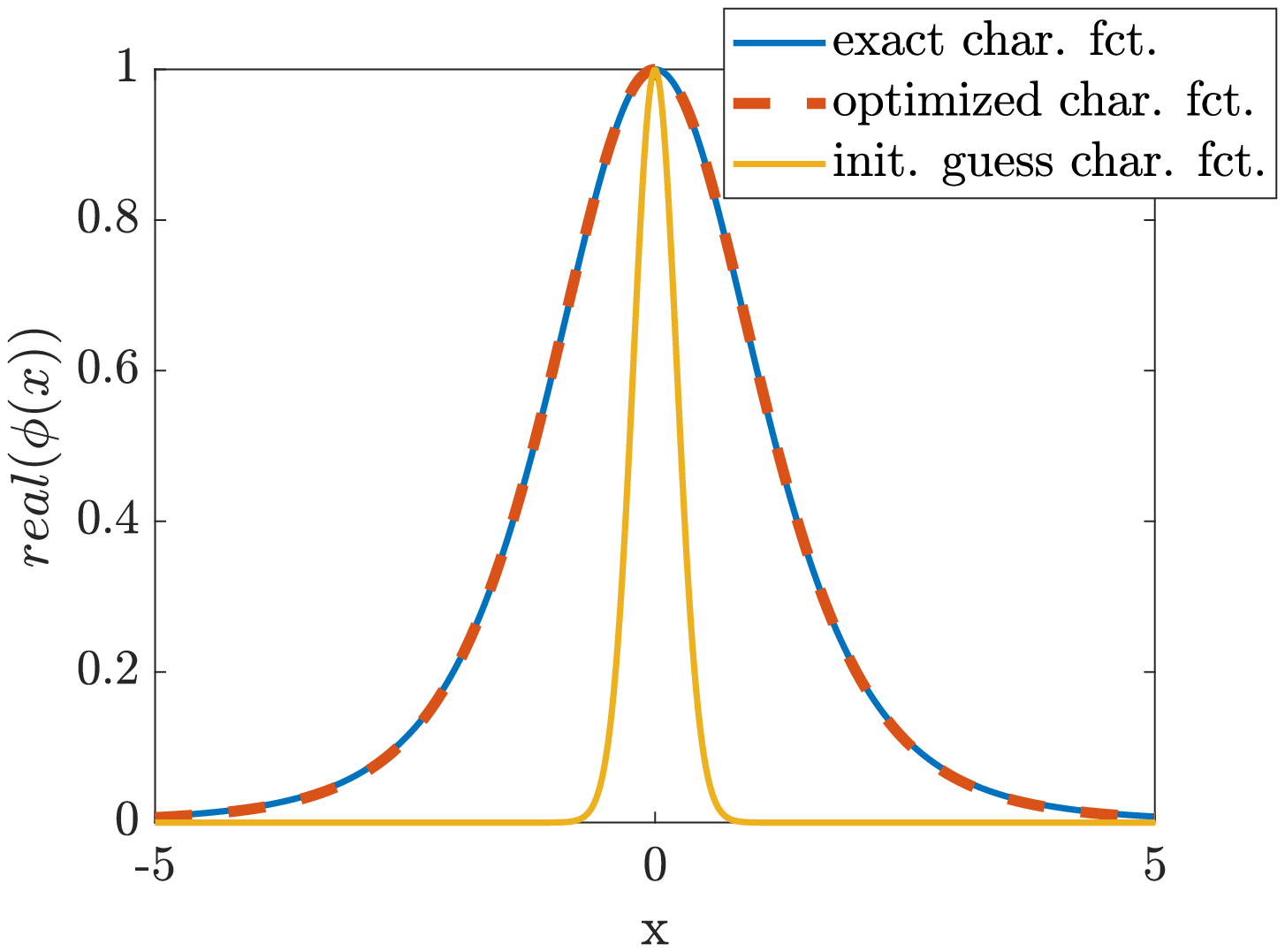}}
\subfigure{\includegraphics[scale=0.23]{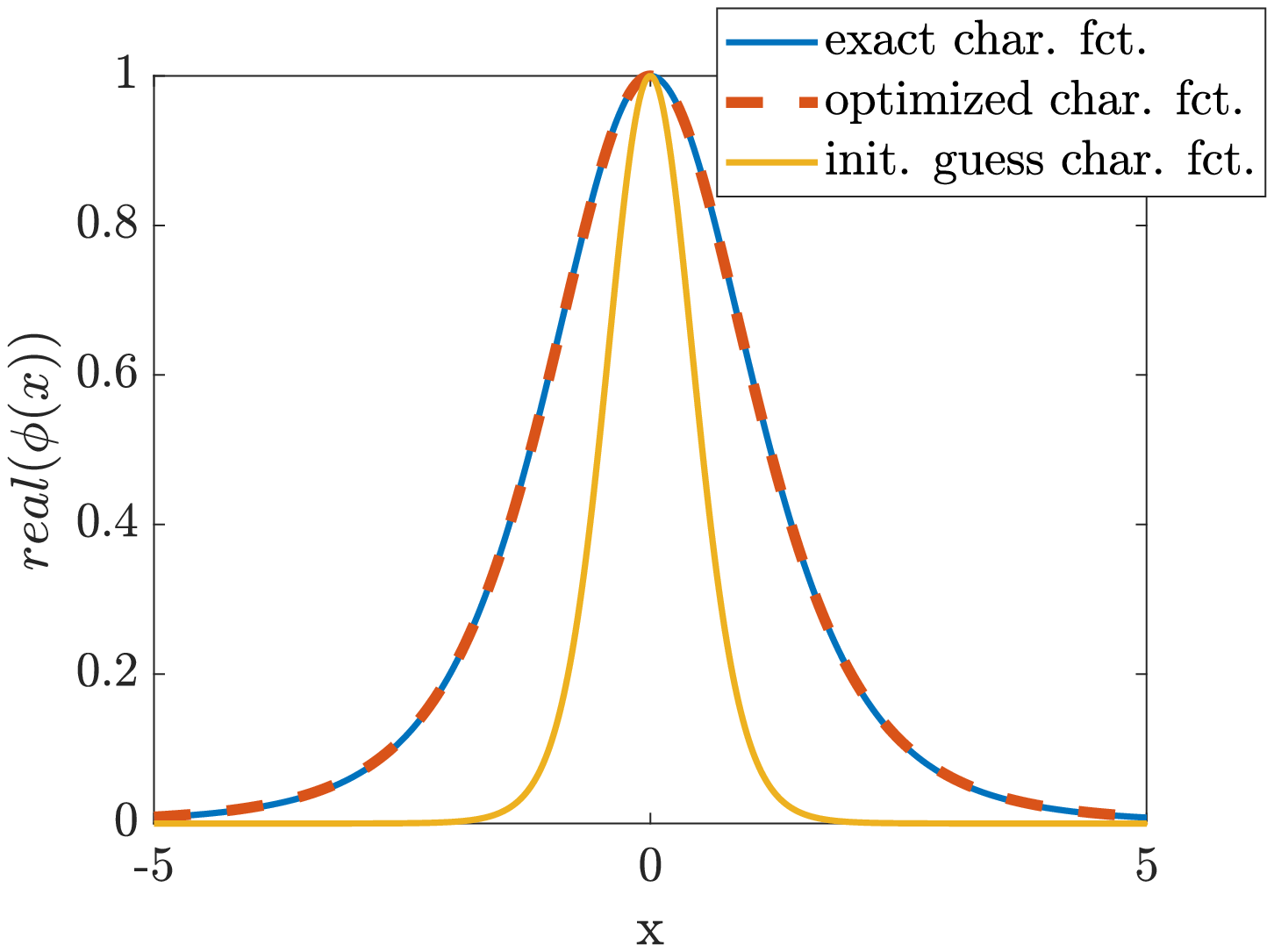}}
\subfigure{\includegraphics[scale=0.23]{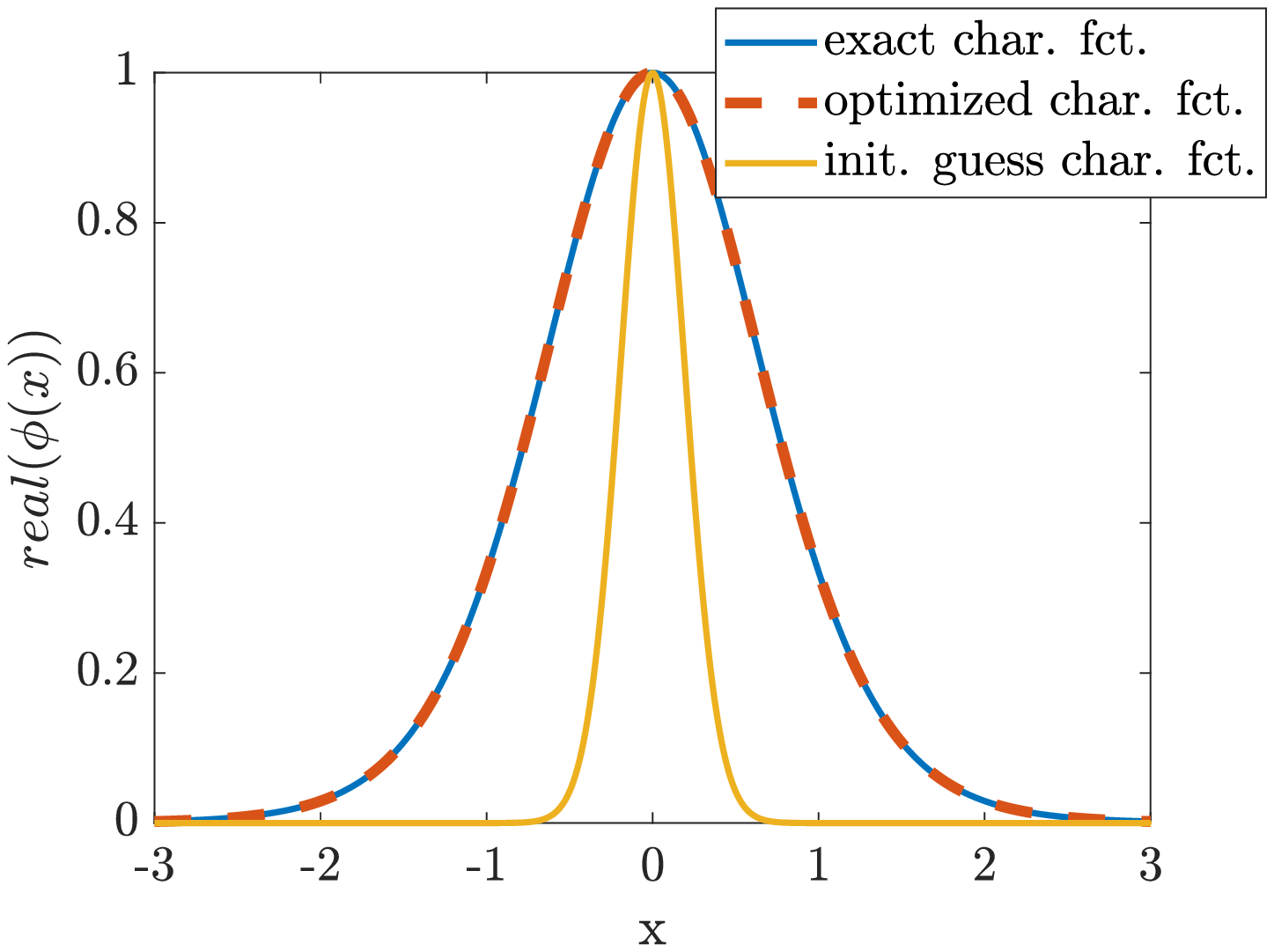}}
\caption{Real part of the characteristic functions corresponding to the initial parameters $a_G^{(1)}=1,~b_G^{(1)}=24,~a_G^{(2)}=28,~b_G^{(2)}=1,~\sigma=1$, the optimized parameters and the exact parameters of the Gamma-subordinated GRF at different points: $P_1=(0.1,0.1)$ (left), $P_2=(0.1,0.8)$ (second from left), $P_3=(0.7,0.2)$ (second from right), $P_4=(1,1)$ (right).}
\label{FIG:NumExFittingChar3}
\end{figure}
A further impressive result is shown in Figure  \ref{FIG:NumExFittingChar4}: even if we set the initial parameters to be $a_G^{(1)}=b_G^{(1)}=a_G^{(2)}=b_G^{(2)}=0,~\sigma=1$, which corresponds to the situation that the initial guess distribution is deterministic and zero, the optimized parameters lead to matching characteristic functions at the four points for the characteristic-function-approach.
\begin{figure}[ht]
	\centering
	\subfigure{\includegraphics[scale=0.23]{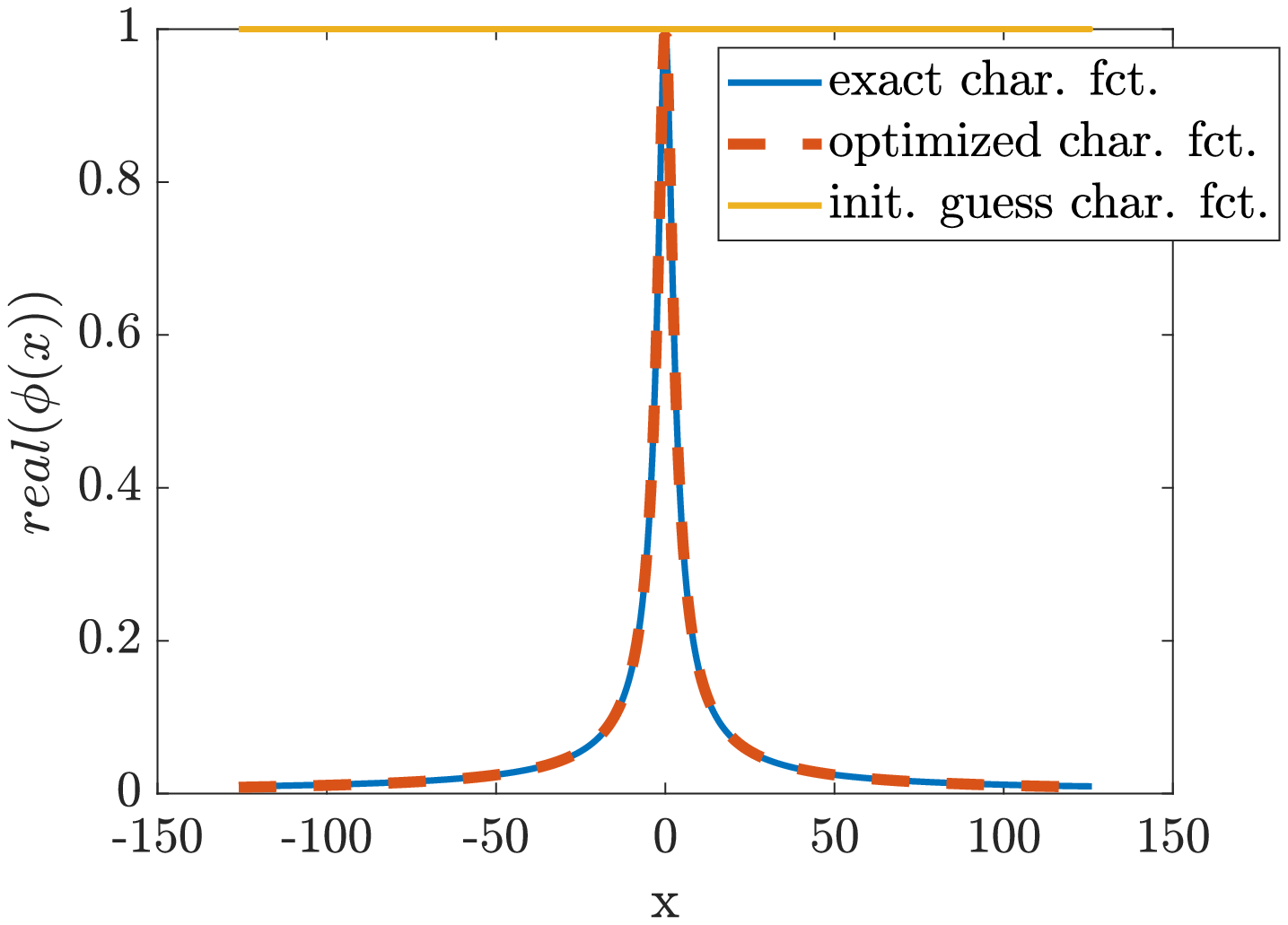}}
\subfigure{\includegraphics[scale=0.23]{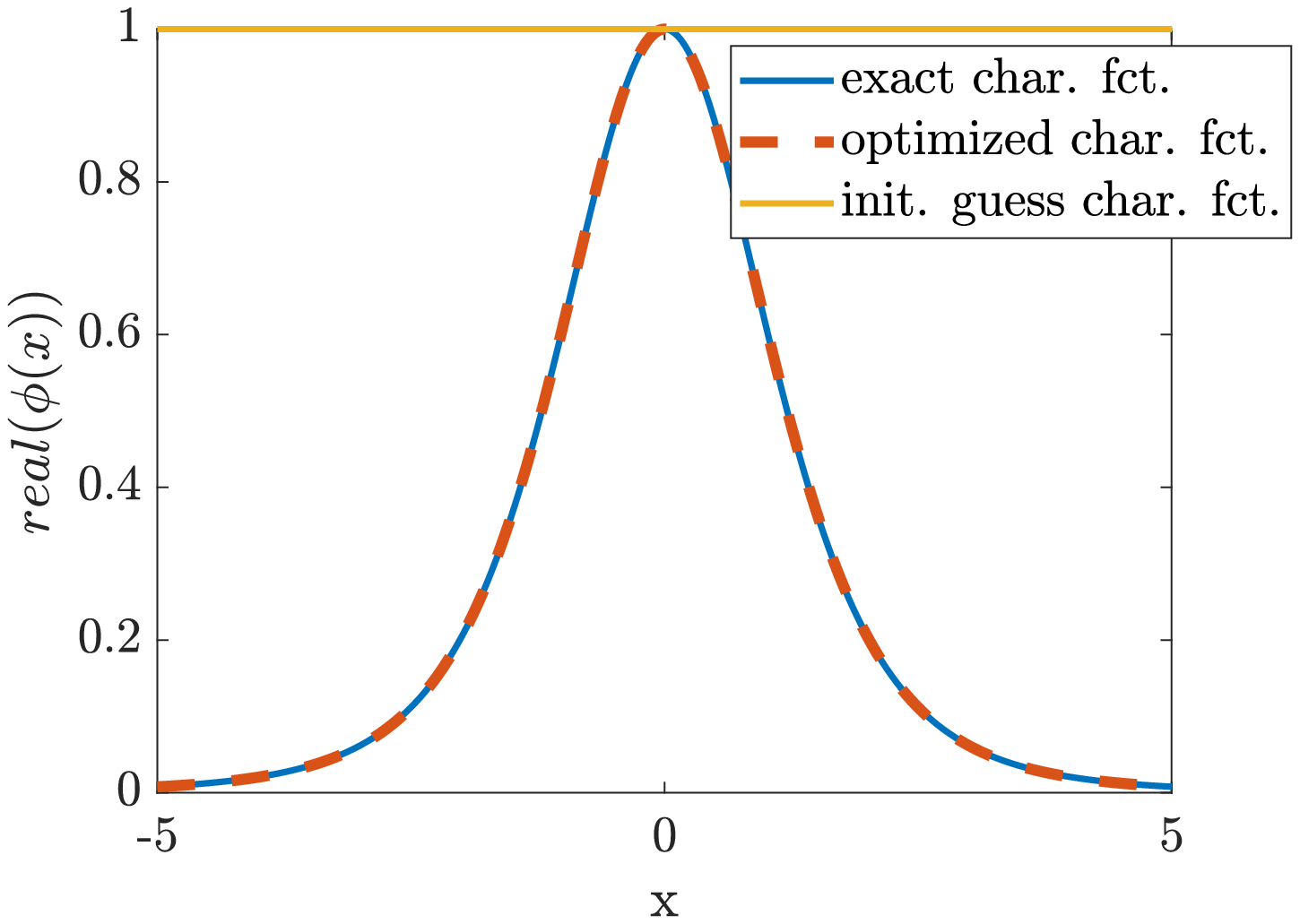}}
\subfigure{\includegraphics[scale=0.23]{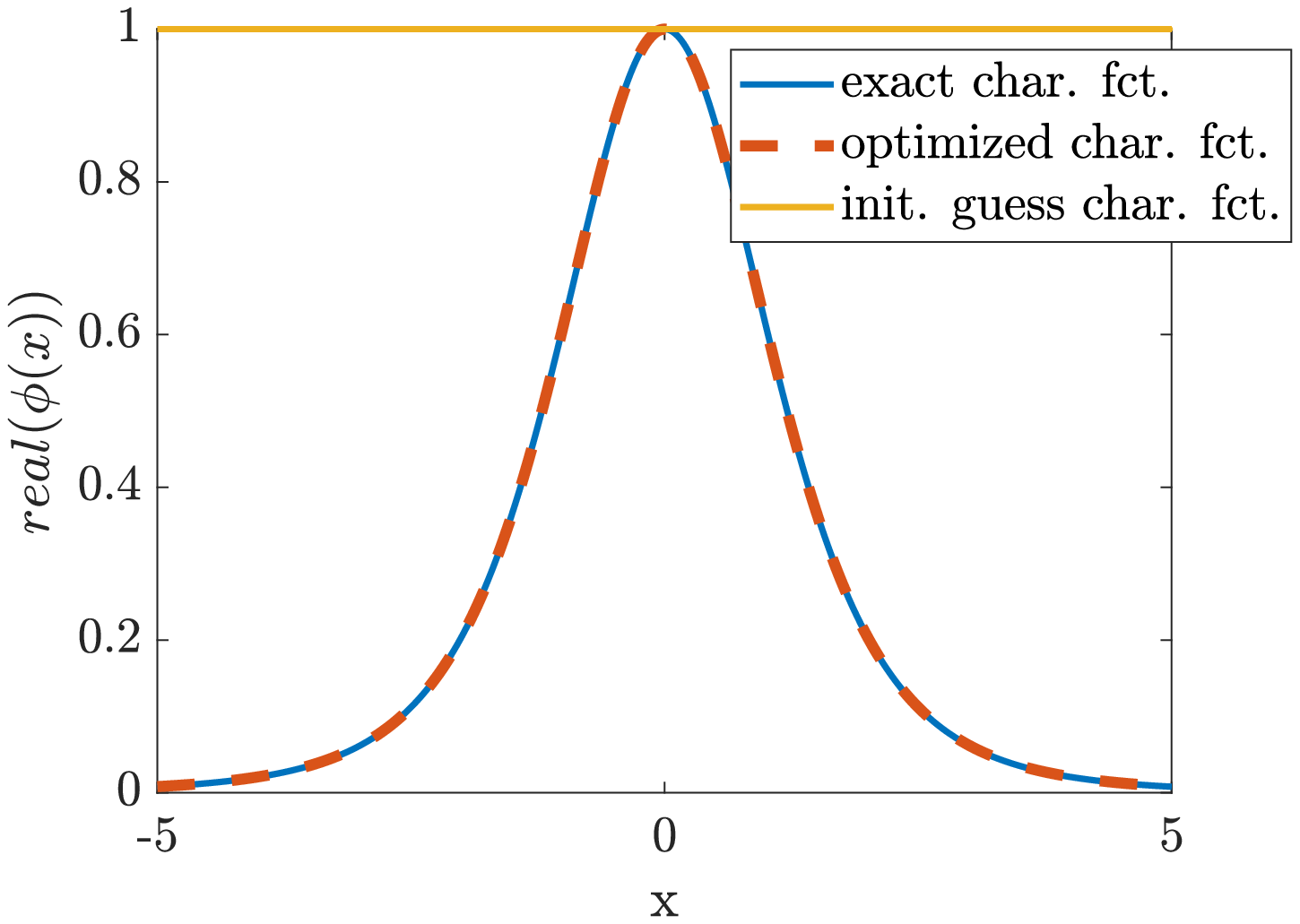}}
\subfigure{\includegraphics[scale=0.23]{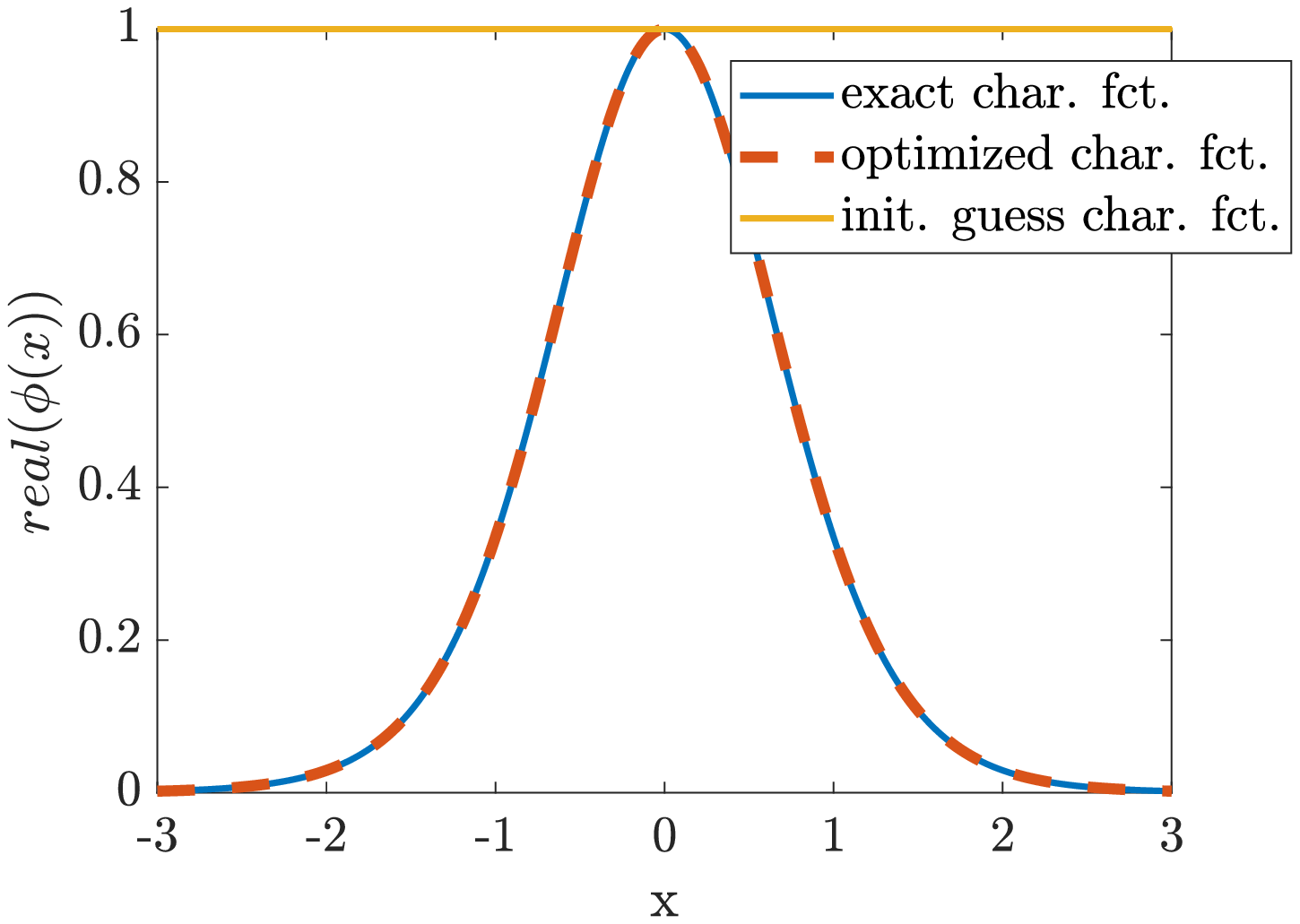}}
\caption{Real part of the characteristic functions corresponding to the initial parameters $a_G^{(1)}=0,~b_G^{(1)}=0,~a_G^{(2)}=0,~b_G^{(2)}=0,~\sigma=1$, the optimized parameters and the exact parameters of the Gamma-subordinated GRF at different points: $P_1=(0.1,0.1)$ (left), $P_2=(0.1,0.8)$ (second from left), $P_3=(0.7,0.2)$ (second from right), $P_4=(1,1)$ (right).}
\label{FIG:NumExFittingChar4}
\end{figure}

After these experiments for both approaches introduced in the beginning of Subsection \ref{SUBSEC:NumExStochFit}, we conclude that the characteristic function approach performs far better than the density approach. It is less computationally expensive, less sensitive to the parameter choice and yields better results, especially for the borderline cases.

\subsection{Pointwise Moments}

In this subsection we numerically validate Theorem \ref{TH:PointwiseMoments} which guarantees the existence of moments of the subordinated GRF if the GRF and the corresponding subordinators satisfy certain conditions. In our numerical examples we set $d=2$ and assume $W$ to be a Brownian sheet on $\mathbb{R}_+^2$. We use different Lévy processes with different stochastic regularity - in terms of the existence of moments - to subordinate the GRF $W$. In order to validate Theorem \ref{TH:PointwiseMoments} we use different statistical methods to verify or disprove the existence of moments of a specific order.

\subsubsection{Statistical methods to test the existence of moments of a random variable}

The existence of moments of a specific distribution is one of the most frequently formulated assumptions in statistical applications. For example, already the strong law of large numbers assumes finiteness of the first moment of the corresponding random variable. Nevertheless, in the literature only few statistical methods exist to verify or disprove the existence of moments, given a specific sample of random variables (see e.g. \cite{TheVariationOfCertaiNSpecalutivePrices,ASimpleGeneralApproachToInferenceAboutTheTailOfADistribution,MR3756232,MR3056087,Fedotenkov2014NoteOnBootstrap,Fedotenkov2013ASimpleNonparametric}) . One of the earlier methods to verify the existence of moments of a distribution was proposed in 1963 by Mandelbrot (see \cite{TheVariationOfCertaiNSpecalutivePrices} and \cite{EmpiricalPropertiesOfAssetReturnsStylizedFactsAndStatisticalIssues}). It is based on the simple observation that the estimated (sample-)moments will converge to a certain value for an increasing sample size if the theoretical moment exists. On the other side, if the theoretical moment does not exist, the estimated moment will diverge or behave unstable when the sample size increases. However, this quite intuitive method is rather heuristic and depends highly on the experience of the researcher (see also \cite{MR3056087}). Another popular direct way to investigate the existence of moments of a certain distribution is the sample-based estimation of a decay rate $\alpha$ for the corresponding density function proposed by Hill in \cite{ASimpleGeneralApproachToInferenceAboutTheTailOfADistribution}.
However, the Hill-estimator requires a parameter $k>0$ which specifies the sample values which are considered as the tail of the distribution and it turned out that the Hill-estimator is very sensitive to the choice of this parameter $k$. Further, the method makes the quite restrictive assumption that the underlying distribution is of Pareto-type (see 
\cite{MR3756232,MR3056087,Fedotenkov2014NoteOnBootstrap} and \cite{Fedotenkov2013ASimpleNonparametric}). In 2013, Fedotenkov proposed a bootstrap test for the existence of moments of a given distribution (see  \cite{MR3056087}). The test performs well for specific distributions, however, its accuracy deteriorates fast when moments of higher order are considered (see also \cite{Fedotenkov2014NoteOnBootstrap}). Recently, Ng and Yau proposed another sample-based bootstrap test for the existence of moments which outperforms the previously mentioned methods for many distributions (see \cite{MR3756232}). The test is based on a result from bootstrap asymptotic theory which states that the $m$ out of $n$ bootstrap sample mean (see \cite{BickelResampling}) converges weakly to a normal distribution. For details we refer to \cite{MR3756232}.

Based on these observations, we conduct a combination of direct moment estimation via Monte Carlo (see also Subsection \ref{SUBSEC:NumExCovFormula}) and the bootstrap test proposed by Ng and Yau to investigate the existence of (pointwise) moments of the subordinated GRF.

For our numerical examples we choose three different Lévy distributions to subordinate the Brownian sheet $W$: a Poisson distribution, a Gamma distribution and a Student-t distribution. Therefore, we use a discrete and a continuous distribution where all moments are finite and a continuous distribution, which only admits a limited number of moments and. Hence, we consider three fundamentally different situations. In all three experiments, we consider the evaluation point $\underline{x}=(x_1,x_2)=(1,1)\in \mathbb{R}_+^2$ for the subordinated GRF $L$. Note that the two-dimensional Brownian sheet satisfies Equation \eqref{EQ:TailEstGRFVar} in Theorem \ref{TH:PointwiseMoments} with $N=1$, $c_1=1$ and  $\underline{\alpha}^{(1)}=(1/2,1/2)$.

\subsubsection{Poisson-subordinated Brownian sheet}
In this example, we use Poisson($3$) processes to subordinate the two-dimensional Brownian sheet. We recall that a Poisson($\lambda$)-distributed random variable admits the density
\begin{align*}
k\mapsto e^{-\lambda} \frac{\lambda^k}{k!},  \text{ for } k\in \mathbb{N}_0.
\end{align*}
Hence, condition \eqref{EQ:TailEstSubordDiscrete} is satisfied for any $\eta_i>0$, $i=1,2$, since point evaluations of a Poisson process are Poisson distributed. Theorem \ref{TH:PointwiseMoments} implies the existence of the $p$-th moment of the evaluated field $L(1,1)$ for any $p<\infty$ (see Remark \ref{REM:PointwiseStochRegDiscrDist}).
We estimate the $p$-th moment for $p\in \{4,6,8\}$ by a Monte Carlo estimation using $M$ samples of the evaluated GRF $L(1,1)$ with $M\in \mathbb{N}$. Figure \ref{FIG:MomEstPoissSubordBS} shows the development of the MC-estimator for the $p$-th moment as a function of the number of samples. For every moment, we take $5$ independent MC-runs to validate that they converge to the same value. 
\begin{figure}[ht]
	\centering
\subfigure{\includegraphics[scale=0.18]{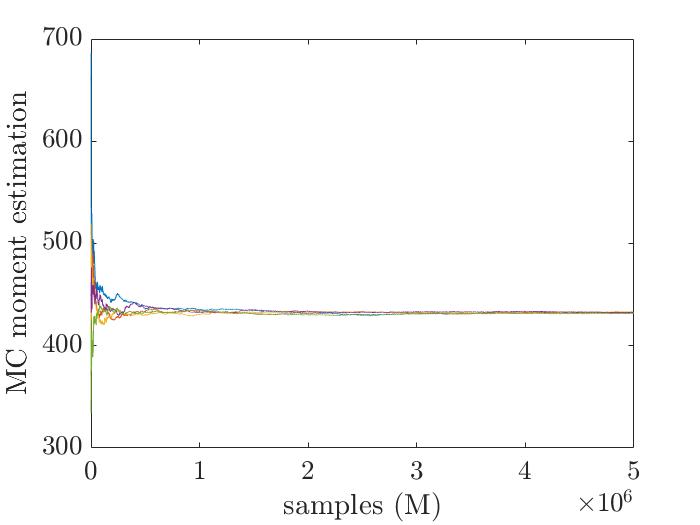}}
\subfigure{\includegraphics[scale=0.18]{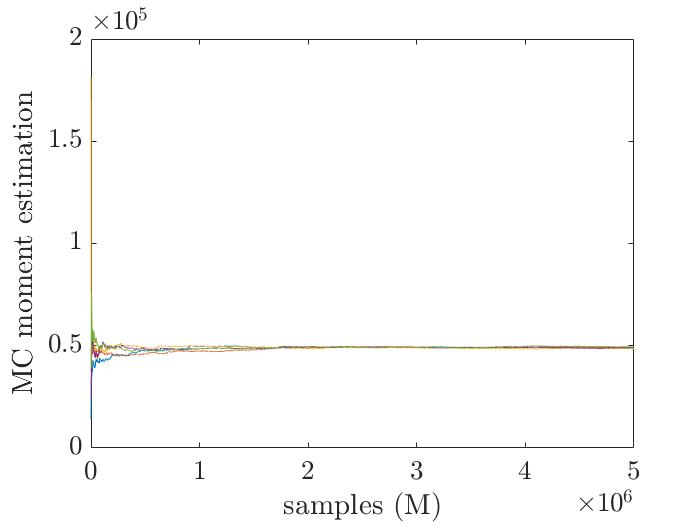}}
\subfigure{\includegraphics[scale=0.18]{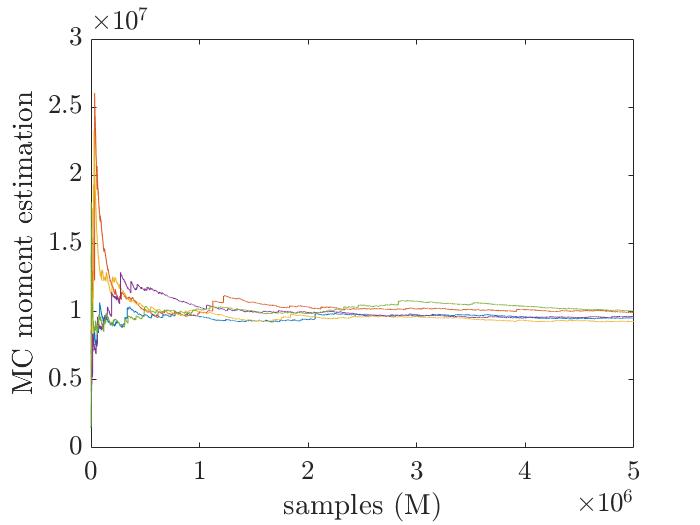}}
\caption{Five independent Monte Carlo estimations of $\mathbb{E}(|L(1,1)|^p)$ with a Poisson($3$)-subordinated Brownian sheet; $p=4$ (left), $p=6$ (middle), $p=8$ (right).}
\label{FIG:MomEstPoissSubordBS}
\end{figure}
As expected, Figure \ref{FIG:MomEstPoissSubordBS} shows a stable convergence of the MC estimator for a growing number of samples for every considered moment. Further, the different independent MC-runs converge to the same value - the theoretical $p$-th moment for $p\in \{4,6,8\}$.

In order to further confirm this result, we perform the bootstrap test (see \cite{MR3756232}). We test the existence of the $p$-th moment for $p\in\{1,2,3,4,5,6,7,8\}$ using $M = 10^7$ samples of the subordinated evaluated GRF $L(1,1)$. Hence, the null and alternative hypothesis are given by 
\begin{align*}
H_0: ~\mathbb{E}(|L(1,1)|^p)<+\infty \text{ vs. } H_1: ~\mathbb{E}(|L(1,1)|^p)=+\infty,
\end{align*}
for the different values of $p$. We choose the significance level $\alpha_s = 1\%$ and perform $10$ independent test runs. Figure \ref{FIG:BootstrapTestPoiss} shows the proportion of acceptance of the null hypothesis in the $10$ test runs as a function of the considered moment $p\in \{1,2,3,4,5,6,7,8\}$.

\begin{figure}[ht]
	\centering
	\subfigure{\includegraphics[scale=0.3]{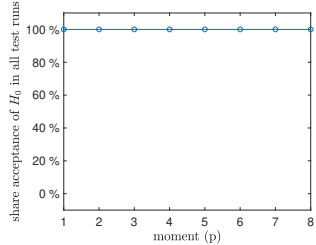}}\
\caption{Results for ten independent runs of the bootstrap test for the existence of the $p$-th moment using Poisson(3) processes to subordinate the Brownian sheet.}
\label{FIG:BootstrapTestPoiss}
\end{figure}

As we see in Figure \ref{FIG:BootstrapTestPoiss}, the bootstrap test accepts the null hypothesis $H_0$ in every test run for every considered moment $p\in\{1,2,3,4,5,6,7,8\}$. Therefore, the test results further support the expected observation that all of the considered moments exist.

\subsubsection{Gamma-subordinated Brownian sheet}
In our second numerical example we consider Gamma processes to subordinate the Brownian sheet. We recall that, for $a_G,b_G>0$, a $Gamma(a_G,b_G)$-distributed random variable admits the density function
\begin{align*}
 x\mapsto \frac{b_G^{a_G}}{\Gamma(a_G)}x^{a_G-1}\exp(-xb_G),~\text{ for }x>0,
\end{align*}
where $\Gamma(\cdot)$ denotes the Gamma function. A Gamma process $(l(t))_{t\geq 0}$ has independent Gamma distributed increments and $l(t)$ follows a $Gamma(a_G\cdot t,b_G)$-distribution for $t>0$. Therefore, condition \eqref{EQ:TailEstSubord} holds for any $\eta_i>0$, for $i=1,2$ and, hence, Theorem \eqref{TH:PointwiseMoments} again implies the existence of every moment, i.e. $\mathbb{E}(|L(1,1)|^p)<\infty$ for any $p\geq 1$. We choose $a_G=4$, $b_G=10$ and estimate the $p$-th moment of $L(1,1)$ with $p\in\{4,6,8\}$ by a Monte Carlo estimation using a growing number of samples $M\in \mathbb{N}$. Figure \ref{FIG:MomEstGammaSubordBS} shows the development of the MC-estimator for the $p$-th moment as a function of the number of samples. As in the first experiment, we take $5$ independent MC-runs to validate the convergence to a unique value. In line with our expectations, the results show a stable convergence of the MC estimations for the different moments.
\begin{figure}[ht]
	\centering
\subfigure{\includegraphics[scale=0.18]{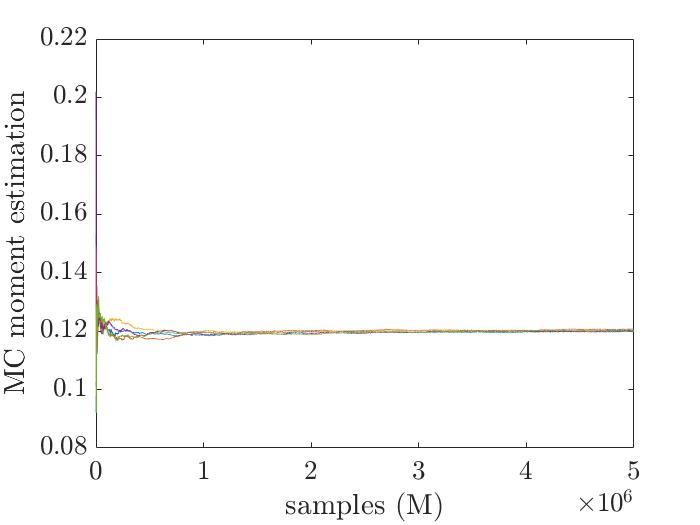}}
\subfigure{\includegraphics[scale=0.18]{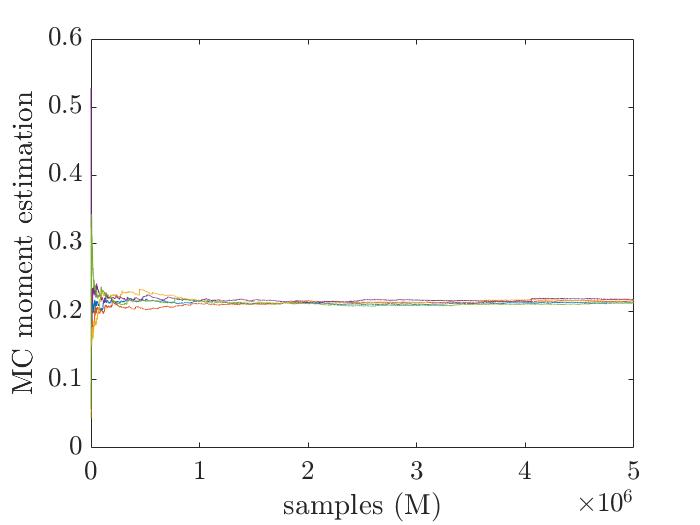}}
\subfigure{\includegraphics[scale=0.18]{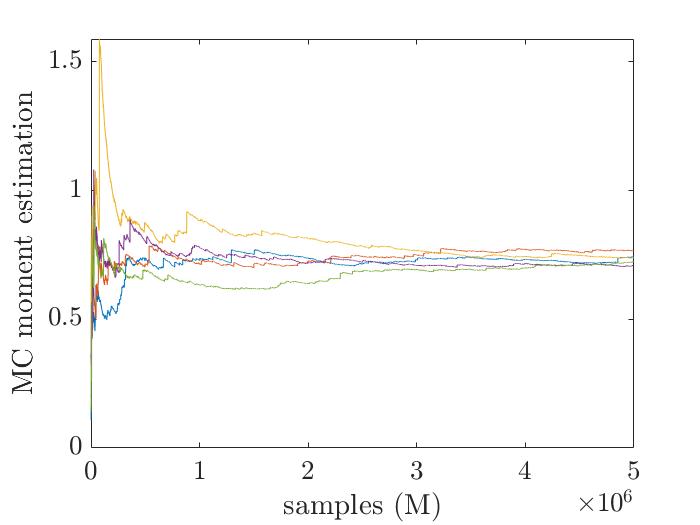}}
\caption{Five independent Monte Carlo estimations of $\mathbb{E}(|L(1,1)|^p)$ with a $Gamma(4,10)$-subordinated Brownian sheet; $p=4$ (left), $p=6$ (middle), $p=8$ (right).}
\label{FIG:MomEstGammaSubordBS}
\end{figure}

In this experiment we again perform the bootstrap test for the existence of the $p$-th moment for $p\in\{1,2,3,4,5,6,7,8\}$ using $M = 10^7$ samples of the subordinated evaluated GRF $L(1,1)$. Hence, the null and alternative hypothesis are given by 
\begin{align*}
H_0: ~\mathbb{E}(|L(1,1)|^p)<+\infty \text{ vs. } H_1: ~\mathbb{E}(|L(1,1)|^p)=+\infty,
\end{align*}
for the different values of $p$. We choose the significance level $\alpha_s = 1\%$ and perform 10 independent test runs. Figure \ref{FIG:BootstrapTestGamma} shows the proportion of acceptance of the null hypothesis in the 10 test runs as a function of the considered moment $p\in \{1,2,3,4,5,6,7,8\}$.

\begin{figure}[ht]
	\centering
	\subfigure{\includegraphics[scale=0.3]{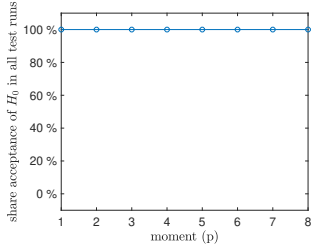}}
\caption{Results for ten independent runs of the bootstrap test for the existence of the $p$-th moment using Gamma(4,10) processes to subordinate the Brownian sheet.}
\label{FIG:BootstrapTestGamma}
\end{figure}
The test results again confirm our theoretical findings, since every test run accepts the null hypothesis for any moment $p\in \{1,2,3,4,5,6,7,8\}$, indicating that the moments exist. 

\subsubsection{Student t-subordinated Browinan Sheet}
In our last experiment we want to consider a Lévy process where the pointwise distribution only admits a finite number of moments. 
The Student's $t$-distribution with three degress of freedom admits the density function
\begin{align}\label{EQ:DensityTDist}
f_t(x) = \frac{\Gamma(2)}{\sqrt{3\pi} \Gamma(3/2)}\Big(1+\frac{x^2}{3}\Big)^{-2}, \text{ for }x\in \mathbb{R}.
\end{align}
It follows by \cite[Theorem 3]{InfiniteDivisibilityAndVarianceMixturesOfTheNormalDistribution}
that a Student-t distributed random variable with three degrees of freedom is infinitely divisible. Hence, we can define Lévy processes $l_j$, for $j=1,2$, such that $l_j(1)$ follows a Student-t distribution with three degrees of freedom for $j=1,2$ (see \cite[Theorem 7.10]{LevyProcessesAndInfinitelyDivisibleDistributions}).
Using these processes and the Brownian sheet $W$, we consider the subordinated GRF $L(x_1,x_2):=W(|l_1(x_1)|,|l_2(x_2)|)$ for $(x_1,x_2)\in [0,T_1]\times[0,T_2]$ (see Remark \ref{REM:PointwiseStochRegGeneralLavyPr}). For our numerical experiment we again evaluate the field at $(x_1,x_2)=(1,1)$. Using \eqref{EQ:DensityTDist} we obtain
\begin{align*}
f_t(x) \leq C|x|^{-4} \text{, for } x\in \mathbb{R}. 
\end{align*}
Therefore, condition \eqref{EQ:TailEstSubord} is satisfied for $\eta_i = 4$, for $i=1,2$, and it is violated for any $\eta_i >4$ (see also Remark \ref{REM:PointwiseStochRegGeneralLavyPr}). Since the Brownian sheet satisfies condition \eqref{EQ:TailEstGRFVar} with $N=1$, $c_1=1$ and $\underline{\alpha}^{(1)}=(1/2,1/2)$, Theorem \ref{TH:PointwiseMoments} yields that $\mathbb{E}(|L(1,1)|^p)<\infty$ for $p< 6$ and we expect that this boundary is sharp, i.e. we expect that $\mathbb{E}(|L(1,1)|^p)=\infty$ for $p\geq 6$.

We estimate the $p$-th moment for $p\in \{5,6,8\}$ by a MC-esimation of $M$ samples of the evaluated GRF $L(1,1)$ with $M\in \mathbb{N}$. In Figure \ref{FIG:MomEsttDistSubordBS} we see the development of the MC-estimator for the $p$-th moment as a function of the number of samples. For every moment, we take $5$ independent MC-runs. 
\begin{figure}[ht]
	\centering
\subfigure{\includegraphics[scale=0.18]{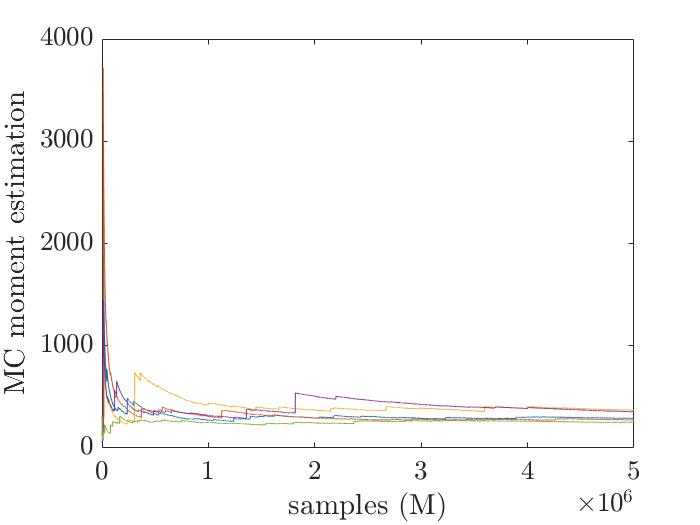}}
\subfigure{\includegraphics[scale=0.18]{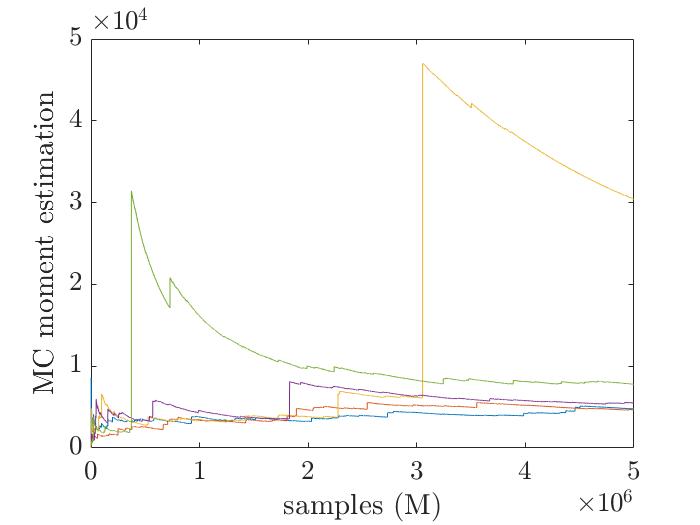}}
\subfigure{\includegraphics[scale=0.18]{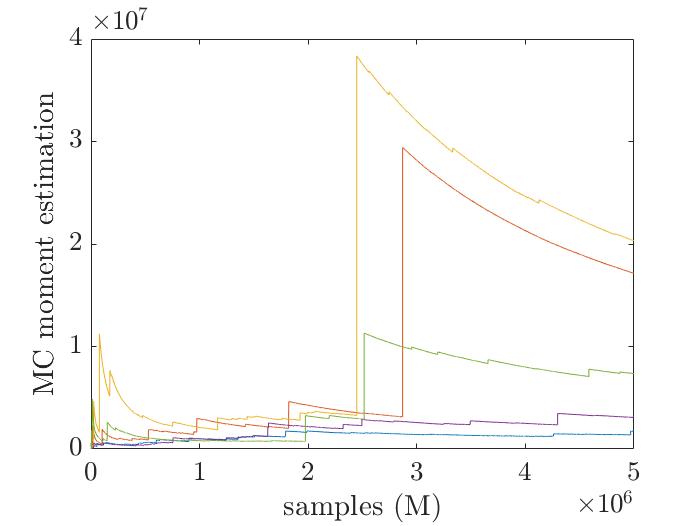}}
\caption{Five independent Monte Carlo estimations of $\mathbb{E}(|L(1,1)|^p)$ with a Student-$t$-subordinated Brownian sheet; 
$p=5$ (left), $p=6$ (middle), $p=8$ (right).}
\label{FIG:MomEsttDistSubordBS}
\end{figure}
The results indicate a convergence of the MC-estimations of the $p$-th moment for $p=5$: in this case the estimation stabilizes with growning sample size and all $5$ independent MC-estimations seem to converge to a unique value. However, for the higher moments $p=6$ and $p=8$, we see upward breakouts and instable behaviour of the corresponding MC-estimator for increasing sample sizes. Further, the $5$ independent MC-runs do not indicate a convergence to a unique value. Therefore, these experiments hit our expectations, since the $p$-th moment of the evaluated subordinated GRF $L(1,1)$ admits a $p$-th moment for $p<6$ and this boundary is sharp (see Theorem \ref{TH:PointwiseMoments}). 

We perform the bootstrap test for the existence of the $p$-th moment for \newline $p\in\{1,2,3,4,4.5,5,5.2,5.4,5.6,5.8,6,6.5,7,8\}$ using $M = 10^7$ samples of the subordinated GRF $L(1,1)$. Hence, the null and alternative hypothesis are again given by 
\begin{align*}
H_0: ~\mathbb{E}(|L(1,1)|^p)<+\infty \text{ vs. } H_1: ~\mathbb{E}(|L(1,1)|^p)=+\infty,
\end{align*}
for the different values of $p$. We choose the significance level $\alpha_s = 1\%$ and perform $10$ independent test runs. Figure \ref{FIG:BootstrapTesttDist} shows the proportion of acceptance of the null hypothesis in the $10$ test runs as a function of the considered moment $p$ and the test statistic values for the ten test runs.

\begin{figure}[ht]
	\centering
	\subfigure{\includegraphics[scale=0.4]{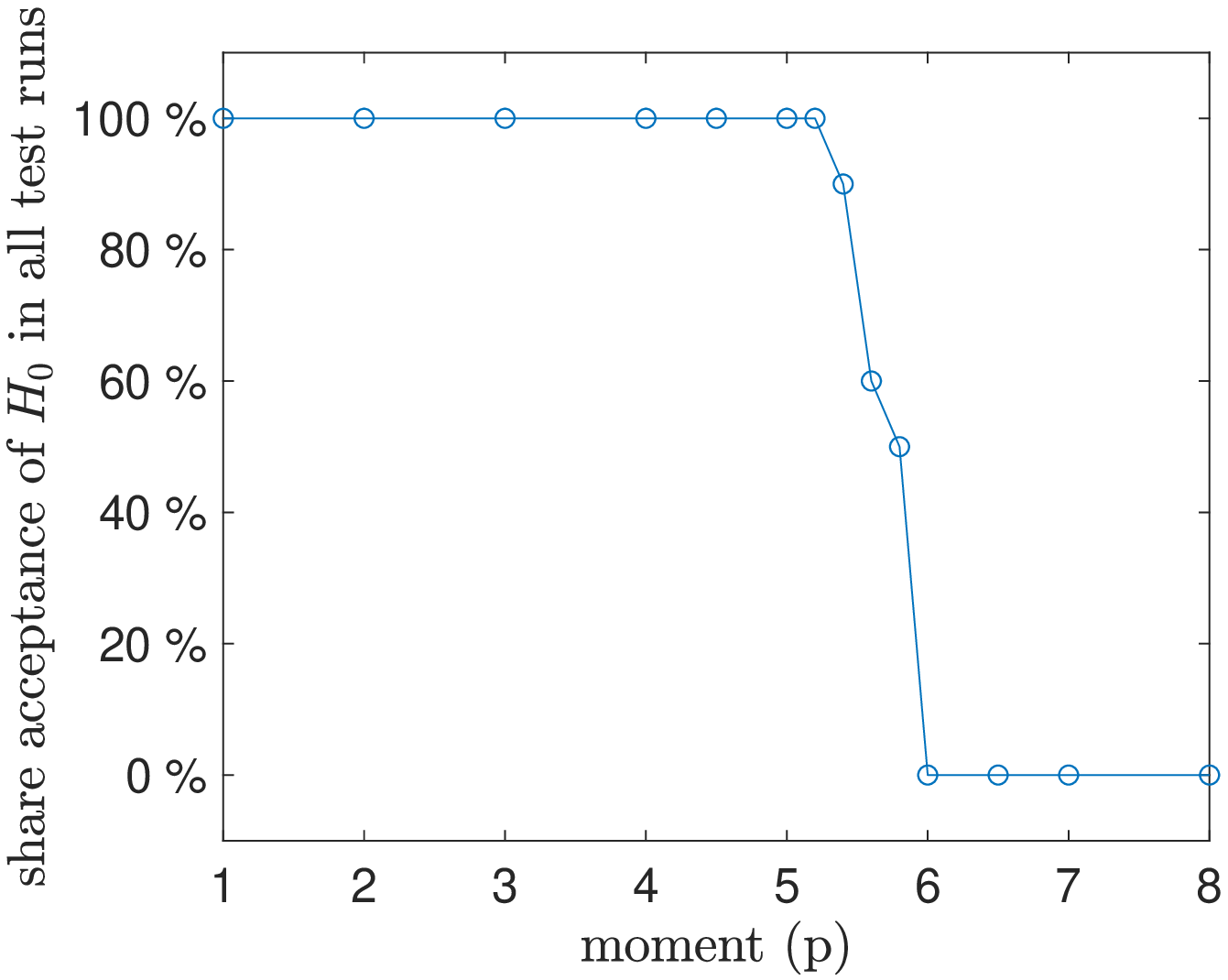}}
\subfigure{\includegraphics[scale=0.4]{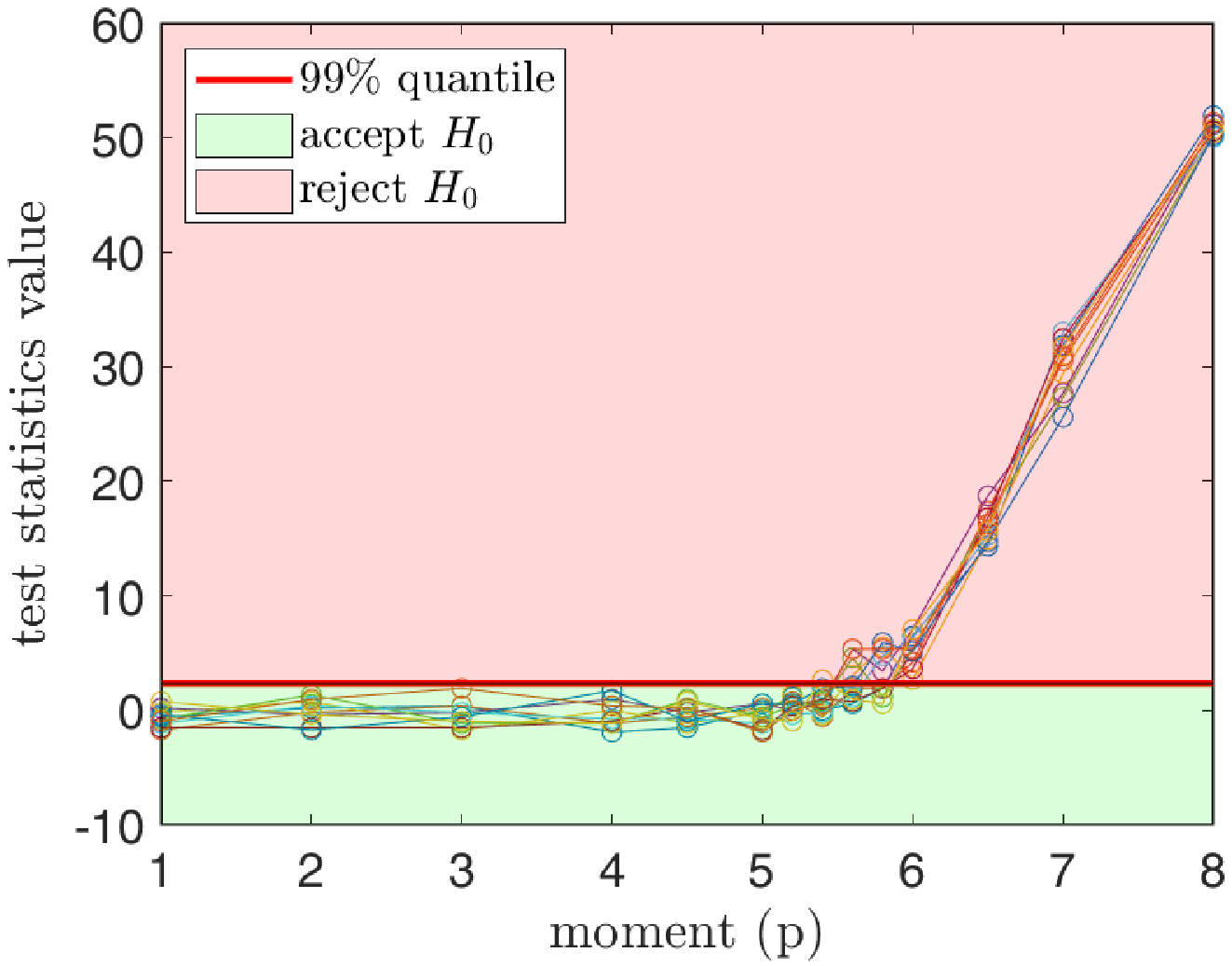}}
\caption{Bootstrap test for the existence of the $p$-th moment using Student-$t$-distributed random variables as subordinators: share acceptance of $H_0$ (left), test statistic values for the test runs (right).}
\label{FIG:BootstrapTesttDist}
\end{figure}

The results further confirm our theoretical finding: in all of the ten test runs the null hypothesis is accepted for the cases $p\in \{1,2,3,4,4.5,5\}$. Further, in all of the ten test runs $H_0$ is rejected for the cases $p\in\{6,6.5,7,8\}$, which is correct since the corresponding theoretical moments do not exist. Only the results for $p\in(5,6)$ are not perfectly accurate since the corresponding theoretical moments exist and the test rejects the null hypothesis in some cases. Nevertheless, the results are remarkable since even for $p=5.8$ the test accepts $H_0$ in 5 of the 10 testruns. This again implies that the power and accuracy of the test is impressively high in our numerical examples.

	
	\section*{Acknowledgments}
	Funded by Deutsche Forschungsgemeinschaft (DFG, German Research Foundation) under Germany's Excellence Strategy - EXC 2075 - 390740016.
	

	\bibliographystyle{siam}
	\bibliography{references}

\end{document}